\documentclass[reqno, 12pt]{amsart}

\usepackage{mathtools}
\usepackage[T1]{fontenc}
\usepackage{amsmath, amsthm, amssymb}
\usepackage[ansinew]{inputenc}
\usepackage{xcolor}
\usepackage{bbm}
\usepackage{mathrsfs}
\usepackage{cite}
\usepackage{xparse}
\usepackage{multirow}
\usepackage{fullpage}
\usepackage[toc,page]{appendix}

\newtheorem{theorem}{Theorem}[section]
\newtheorem{lemma}[theorem]{Lemma}
\newtheorem{prop}[theorem]{Proposition}
\newtheorem{cor}[theorem]{Corollary}

\theoremstyle{definition}
\newtheorem{definition}[theorem]{Definition}

\theoremstyle{remark}
\newtheorem{remark}[theorem]{Remark}

\numberwithin{equation}{section}

\DeclareMathOperator{\rg}{rg}
\DeclareMathOperator{\rank}{rank}

\begin{document}

\title{Stable blowup for the cubic wave equation in higher dimensions}

\author{Athanasios Chatzikaleas}
\address{Rheinische Friedrich - Wilhelms - Universit\"at Bonn, Mathematisches Institut, Endenicher Allee 60, D-53115 Bonn, Germany}
\email{achatzik@math.uni-bonn.de}

\author{Roland Donninger}
\address{Universit\"at Wien, Fakult\"at f\"ur Mathematik, Oskar-Morgenstern-Platz 1, A-1090 Vienna, Austria}
\email{roland.donninger@univie.ac.at}
\thanks{R.D.~is supported by the Austrian Science Fund FWF, Project P 30076-N32. The authors would like to thank Birgit Sch{\"o}rkhuber for fixing
a mistake in an earlier version of this paper.}

\dedicatory{}

\begin{abstract}
We consider the wave equation with a focusing cubic nonlinearity in
higher odd space dimensions without symmetry restrictions on the data. We prove that there exists an open set of initial data such that the corresponding solution exists in a backward light-cone and approaches the ODE blowup profile.
\end{abstract}

\maketitle


\section{Introduction} 

\noindent
\subsection{Cubic wave equation} 

In this paper we study the wave equation with a focusing cubic nonlinearity
\begin{align} \label{cubicwave}
\Box u (t,x) +u^{3} (t,x) = 0,
\end{align}
with $(t,x) \in  \mathbb{R}^{1+d}$. Here, $\Box$ stands for the Laplace-Beltrami operator on Minkowski space with signature $(-+++)$, i.e.,
\begin{align*}
\Box := - \partial ^{2}_{t} + \Delta _{x}.
\end{align*}
Equation \eqref{cubicwave} has the conserved energy
\begin{align*}
E[u](t):= \frac{1}{2} \int _{\mathbb{R}^{d}}    \left | \partial _{t} u(t,x) \right|^{2} dx + \frac{1}{2}  \int _{\mathbb{R}^{d}}  \left | \nabla _{x} u(t,x) \right|^{2} dx - \frac{1}{4} \int _{\mathbb{R}^{d}}   \left |  u(t,x) \right|^{4} dx.
\end{align*}
Obviously, equation \eqref{cubicwave} is invariant under time-translations. In addition, other symmetries of the equation that are relevant in our context are Lorentz boosts, 
namely, if $u$ is a solution to \eqref{cubicwave}, so is 
\begin{align}  \label{Lorentz}
u_{T,\alpha} (t,x):=u \circ \Lambda _{T} (\alpha)
  \begin{pmatrix}
    t \\
    x
  \end{pmatrix},
\end{align}
for $T \in \mathbb{R}$ and $\alpha = ( \alpha ^{1},\cdots,\alpha^{d} ) \in \mathbb{R}^{d}$. Here, we define the Lorentz transformations in a way that resembles circular rotations in $d$-dimensional space using hyperbolic functions, that is
\begin{align*} 
\Lambda _{T} (\alpha) : = \Lambda ^{d}_{T} (\alpha ^{d}) \circ \Lambda ^{d-1}_{T} (\alpha ^{d-1}) \circ \cdots \circ \Lambda ^{1}_{T} (\alpha ^{1}) 
\end{align*}
where the boost in the $j-$direction is given by
\begin{align*}
\Lambda ^{j}_{T} (\alpha ^{j}) 
  \begin{pmatrix}
    t \\
    x^{1} \\
    \vdots \\
    x^{j} \\
    \vdots \\
    x^{d}
  \end{pmatrix}
  :=
   \begin{pmatrix}
    (t-T) \cosh (\alpha ^{j}) +x^{j} \sinh (\alpha ^{j}) + T \\
    x^{1} \\
    \vdots \\
    (t-T) \sinh (\alpha ^{j}) + x^{j} \cosh (\alpha ^{j}) \\
    \vdots \\
    x^{d}
  \end{pmatrix}. 
\end{align*}
A Lorentz boost can be thought of as a hyperbolic rotation of spacetime coordinates of the $(1+d)-$dimensional Minkowski space. The parameter $\alpha \in \mathbb{R}^{d}$ (called rapidity) is the hyperbolic angle of rotation, analogous to the ordinary angle for circular rotations. Note in particular that the spacetime event $(T,0)$ is a fixed point of the transformation $\Lambda _{T} (\alpha)$ and the light-cones emanating from $(T , 0)$ are invariant under $\Lambda _{T} (\alpha)$.

\subsection{Blowup solutions}

Equation \eqref{cubicwave} exhibits finite-time blowup from smooth, compactly supported initial data. This fact is most easily seen by looking at spatially homogeneous blowup solutions. In other words, we ignore the Laplacian in the space variable in the equation and the remaining ordinary differential equation can be solved explicitly. This leads to the solution 
\begin{align*}
u_{1}(t,x):= \frac{\sqrt{2}}{1-t}.
\end{align*}
Using the symmetries of the equation we get a larger family of blowup solutions. Namely, time translation symmetry yields
\begin{align} \label{groundstate}
u_{T}(t,x):= \frac{\sqrt{2}}{T-t}
\end{align} 
and Lorentz symmetry implies that
\begin{align} \label{blowsol}
u_{T,\alpha} (t,x) =
  \frac{\sqrt{2}}{A_{0}(\alpha)(T-t)-A_{j}(\alpha) x^{j}}
\end{align}
is also a solution, see \eqref{Lorentz}. Here and in the following, we adopt the Einstein summation convention, namely
\begin{align*}
a_{j}b^{j} = \sum _{j=1}^{d}a_{j}b^{j}
\end{align*} 
and 
\begin{align*}
  \left \{
  \begin{aligned}
     A_{0}(\alpha) &:=\cosh (\alpha^{d})  \cdots   \cosh (\alpha^{3}) \cosh (\alpha^{2}) \cosh (\alpha^{1}), && \ \\
    A_{1}(\alpha) &:= \cosh (\alpha^{d}) \cdots   \cosh (\alpha^{3})\cosh (\alpha^{2}) \sinh (\alpha^{1}), && \ \\
    A_{2}(\alpha) &:=   \cosh (\alpha^{d}) \cdots  \cosh (\alpha^{3}) \sinh (\alpha^{2}), && \ \\
    &\quad \vdots && \ \\
    A_{d-1}(\alpha) &:=\cosh (\alpha^{d}) \sinh (\alpha^{d-1}), && \ \\
    A_{d}(\alpha) &:= \sinh (\alpha^{d}). && \ \\
  \end{aligned} \right.
\end{align*} 
Observe that $A_{0}(\alpha) = \mathcal{O}(1)$ whereas $A_{j}(\alpha) = \mathcal{O}(\alpha)$ for all sufficiently small $\alpha \in \mathbb{R}^{d}$.
\subsection{The Cauchy problem} 

Our intention is to study the future development of small perturbations of $u_{T_{0},\alpha_{0}}$ under \eqref{cubicwave} for fixed $T_{0} \in \mathbb{R}$ and $\alpha _{0} \in \mathbb{R}^{d}$. Hence, we consider the Cauchy problem 
\begin{align} \label{Cauchy}
\begin{cases}
\Box u (t,x) +u^{3} (t,x) = 0,\\
u[0]=(f,g),
\end{cases}
\end{align}
where
\begin{align} \label{datatilde}
(f,g) = u_{T_{0},\alpha_{0}}[0] + (\tilde{f},\tilde{g}).
\end{align}
Here, we use the abbreviation $u[t]=(u(t,\cdot),\partial_t u(t,\cdot))$ for convenience, $u_{T_{0},\alpha_{0}}$ is defined in \eqref{blowsol} and $(\tilde{f},\tilde{g})$ are small in a suitable sense. Furthermore, we restrict the evolution to the backward light-cone 
\begin{align*}
C_{T}:=\{ (t,x): 0 \leq t <T,~|x| \leq T-t \} = \bigcup_{t\in[0,T)} \{t\} \times \mathbb{B}_{T-t}^{d}.
\end{align*}

\subsection{Related results}

There is a lot of activity in the study of blowup for wave equations. The interest in \eqref{cubicwave} stems from the fact that this equation contains many features
common to a whole range of blow-up problems arising in mathematical physics, as for example in nonlinear optics \cite{BiTaSz11} and general relativity \cite{DoScSo12}. \\ 
By definition, $u$ is a solution to \eqref{Cauchy} if and only if it satisfies the equation in the integral form using Duhamel's principle, namely
\begin{align*}
u(t,\cdot)= \cos \left( t \left | \nabla \right | \right) f + \frac{ \sin \left( t \left | \nabla \right | \right)}{\left | \nabla \right |} g  
+ \int _{0}^{t} \frac{ \sin \left( (t-s) \left | \nabla \right | \right)}{\left | \nabla \right |} u^{3} (s,\cdot) ds,
\end{align*}
for initial data
\begin{align*}
(f,g) \in  H^{s} ( \mathbb{R}^{d} ) \times H^{s-1} ( \mathbb{R}^{d} ).
\end{align*}
Using this formula, one can show that \eqref{Cauchy} is locally well-posed for initial data in $\dot{H}^{s} ( \mathbb{R}^{d} ) \times \dot{H}^{s-1} ( \mathbb{R}^{d} )$ for $s>\frac{d}{2}$, see \cite{TaoNonlinear06}. On the one hand, equation \eqref{cubicwave} is invariant under the scaling transformation
\begin{align}  \label{scaling}
u _{\lambda} (t,x):= \frac{1}{\lambda} u \left( \frac{t}{\lambda},\frac{x}{\lambda} \right),~ \lambda > 0
\end{align}
and
\begin{align*}
\left \| u _{\lambda}(t,\cdot) \right \|_{ \dot{H}^{s} \left ( \mathbb{R}^{d} \right)  } = \lambda ^{ \frac{d}{2} -1 - s } \left \| u \left( \frac{t}{\lambda}, \cdot \right) \right \|_{ \dot{H}^{s} \left ( \mathbb{R}^{d} \right)  }. 
\end{align*}
This scaling property is closely related to the existence of a suitable local theory for the problem and distinguishes the space $\dot H^{s_{3} }( \mathbb{R}^{d} ) \times \dot H^{s_{3}-1 } ( \mathbb{R}^{d} ),~s_{3}:=\frac{d}{2}-1$ as the critical Sobolev space, the unique $L^2$-based homogeneous Sobolev space preserved by the scaling \eqref{scaling}. Indeed, Strichartz theory shows that \eqref{Cauchy} is locally well-posed for initial data in the critical Sobolev space $\dot{H}^{s_{3}} ( \mathbb{R}^{d} ) \times \dot{H}^{s_{3}-1} ( \mathbb{R}^{d} )$, \cite{Sogge08}, \cite{Lindblad95}. On the other hand, equation \eqref{cubicwave} has the conserved energy
\begin{align*}
E[u](t):= \frac{1}{2} \int _{\mathbb{R}^{d}}    \left | \partial _{t} u(t,x) \right|^{2} dx + \frac{1}{2}  \int _{\mathbb{R}^{d}}  \left | \nabla _{x} u(t,x) \right|^{2} dx - \frac{1}{4} \int _{\mathbb{R}^{d} }   \left |  u(t,x) \right|^{4} dx
\end{align*}
which distinguishes the space $\dot{H}^{1}( \mathbb{R}^{d} ) \times L^{2}( \mathbb{R}^{d} )$ as the energy space, that is, the space of initial data for which the energy is known to be finite. For $d \geq 5$, the critical regularity $s_{3}=\frac{d}{2}-1$ is larger than the energy-critical regularity $s=1$ and equation \eqref{cubicwave} is energy-supercritical. \\ \\
The one-dimensional case has been completely understood, see \cite{MeZa07}, \cite{MeZa08}, \cite{MeZa12}, \cite{MeZa12b} where Merle and Zaag exhibited a universal one-parameter family of functions which yields the blowup profile in self-similar variables for general initial data. In higher dimensions, the situation is less clear. In three space dimensions, Bizo\'{n} together with Breitenlohner, Maison and Wasserman in \cite{BiMaDiWa07}, \cite{BiBrMaWa10} showed that equation \eqref{cubicwave} admits infinitely many radial self-similar blowup solutions of the
form 
\begin{align*}
\frac{1}{T-t} f_{n} \left( \frac{|x|}{T-t} \right).
\end{align*}
Here, the ground-state solution \eqref{groundstate} corresponds to
$f_{0}=\sqrt{2}$. Levine \cite{Le74} used energy methods and a
convexity argument to show that initial data with negative energy and
finite $L^{2}-$norm lead to blowup in finite time, see also
\cite{KiBeVi14} for generalizations to the Klein-Gordon equation. We
also mention the works of Alinhac \cite{Se95} and Caffarelli and Friedman \cite{CaFr86}, \cite{CaFr85} where more blowup results can be found. The stability of the ground-state has been studied extensively by Sch\"{o}rkhuber and the second author in three space dimensions (in \cite{DonSch12}, \cite{DonSch14} for radial initial data and in \cite{DonSch16a} without symmetry restrictions) and later in \cite{DonSch17} for all space dimensions and for radial initial data. Some numerical results are available in a series of papers by Bizo\'{n}, Chmaj, Tabor and Zengino\u{g}lu, see \cite{Bi01}, \cite{BiChTa04}, \cite{BiZe09}. Furthermore, in the superconformal and Sobolev subcritical range, an upper bound on the blowup rate was proved by Killip, Stoval and Vi\c{s}an in \cite{KiBeVi14}, then refined by Hamza and Zaag in \cite{HaZa13}. In a series of papers \cite{MeZa05}, \cite{MeZa16}, \cite{MeZa15}, \cite{MerZa05c}, \cite{MerZa03c}, Merle and Zaag obtained sharp upper and lower bounds on the blowup rate of the $H^{1}-$norm of the solution inside cones that terminate at the singularity, see also the work of Alexakis and Shao \cite{AleSh17}. We also mention the recent work by Dodson-Lawrie \cite{DoLa15} on large-data scattering for the cubic equation in five dimensions.


\section{The main result}
\noindent
By finite speed of propagation one can use $u_{T,\alpha}$ to construct smooth, compactly supported initial data which lead to a solution that blows up as $t \longrightarrow T$.  In the present work, we study the asymptotic nonlinear stability of $u_{T,\alpha}$. As a matter of fact, we prove that all initial data from an open, sufficiently small region centered at $u_{T,\alpha}$ lead to the same type of blowup described by the ODE blowup profile. First, we need a definition for our notion of the blowup time.
\begin{definition}
Given initial data $(f,g)$, we define
\begin{align*}
T_{(f,g)} := \sup  \left\{
T >0 \middle|
  \begin{subarray}{c}  
  \exists \text{~solution~} u : C_{T} \longrightarrow \mathbb{R} \text{~to~} \eqref{Cauchy} \text{~in the sense of}   \\ 
   \text{~Definition~} \ref{def} \text{~with initial data~} u[0]=(f,g) |_{\mathbb{B}_{T}^{d}} 
  \end{subarray} \right\} \cup  \{0\}.
\end{align*}
In the case where $T_{(f,g)} < \infty$, we call $T = T_{(f,g)}$ the blowup time at the origin.
\end{definition}
The main result of this work is the following.
\begin{theorem} \label{mainresult}
Fix $d \in \{5,7,9,11,13\}$, $T_{0}>0$ and $\alpha _{0} \in \mathbb{R}^{d}$. There exist constants $M,\delta >0$ such that the following holds. Suppose that the initial data
\begin{align*}
(f,g)  \in H^{\frac{d+1}{2}} (\mathbb{B}_{T_{0}+\delta}^{d}) \times H^{\frac{d-1}{2} } (\mathbb{B}_{T_{0}+\delta}^{d}  )
\end{align*}
satisfy
\begin{align*}
\Big \|  (f,g) -  u_{T_{0},\alpha_{0}} [0] \Big \|_{ H^{\frac{d+1}{2}} (\mathbb{B}_{T_{0}+\delta}^{d}) \times H^{\frac{d-1}{2} } (\mathbb{B}_{T_{0}+\delta}^{d}  )} \leq \frac{\delta}{M}.
\end{align*}
Then, $T = T_{ u[0] } \in [T_{0}-\delta,T_{0}+\delta]$ and there exists an $\alpha \in \mathbb{B}_{3M\delta}^{d}(\alpha_{0})$ such that the solution $u: C_{T} \longrightarrow \mathbb{R}$ to \eqref{Cauchy} satisfies the estimates
\begin{align*}
& (T-t)^{k-\frac{d}{2} +1} \Big \|    u (t,\cdot) - u _{T,\alpha} (t, \cdot) \Big \|_{ \dot{H}^{k}(\mathbb{B}^{d}_{T-t} ) }  \leq \delta (T-t)^{\frac{1}{2}}, \\
& (T-t)^{\ell-\frac{d}{2}+2} \Big \| \partial_t u(t,\cdot) - \partial_t u _{T,\alpha} (t, \cdot) \Big \|_{ \dot{H}^{\ell}(\mathbb{B}^{d}_{T-t} ) }  \leq \delta (T-t)^{\frac{1}{2}},
\end{align*}
for all $k=0,1,\cdots,\frac{d+1}{2}$ and $\ell=0,1,\cdots,\frac{d-1}{2}$.
\end{theorem}

\begin{remark}
Theorem \ref{mainresult} shows that the future development of small perturbations of the blowup solution $u_{T_{0},\alpha_{0}}$ defined in \eqref{blowsol} converge back to $u_{T_{0},\alpha_{0}}$ up to symmetries of the 
equation. 
\end{remark}

\begin{remark}
Note that the normalizing factors on the left-hand sides appear naturally and reflect the behavior of the solution $u_{T,\alpha}$ in the respective homogeneous Sobolev norms. Namely, for
\begin{align} \label{static}
\psi _{\alpha} (\xi) := \frac{\sqrt{2}}{A_{0}(\alpha)-A_{j}(\alpha) \xi ^{j}}
\end{align}
we have
\begin{align*}
& (T-t)^{k-\frac{d}{2} +1} \left \|  u _{T,\alpha} (t, \cdot) \right \|_{ \dot{H}^{k}(\mathbb{B}^{d}_{T-t} ) }  =
(T-t) \left \|  u _{T,\alpha} (t, (T-t) \cdot ) \right \|_{ \dot{H}^{k}(\mathbb{B}^{d}_{1} ) } \simeq
\left \| \psi _{\alpha} \right \|_{ \dot{H}^{k}(\mathbb{B}^{d}_{1} ) }, \\
&  (T-t)^{l-\frac{d}{2} +2} \left \| \partial _{t} u _{T,\alpha} (t, \cdot) \right \|_{ \dot{H}^{\ell}(\mathbb{B}^{d}_{T-t} ) }  =
(T-t)^2 \left \| \partial _{t} u _{T,\alpha} (t, (T-t) \cdot ) \right \|_{ \dot{H}^{\ell}(\mathbb{B}^{d}_{1} ) } \simeq
\left \| \nabla \psi _{\alpha} \right \|_{ \dot{H}^{\ell}(\mathbb{B}^{d}_{1} ) },
\end{align*}
for all $k,\ell \in \mathbb{N}_{0}$ and $\alpha \neq 0$.
\end{remark}

\begin{remark}
We strongly believe that the result holds true in all odd dimensions
$d$ and the restriction on $d$ is not essential and for technical reasons
only. Similarly, the restriction to the cubic power is for the sake of
simplicity only. Similar results are true for any focusing power and
can be proved by straightforward adaptations of our method.
\end{remark}

\begin{remark}
The corresponding result in $d=3$ was proved in \cite{DonSch16a} and
relied on a delicate identity that only holds in 3 dimensions. In this
paper we show that our method is robust enough to extend to
all odd dimensions.
  \end{remark}


\section{Formulation as a first-order system in time} 
Without loss of generality we assume that $T_{0}=1$ and $\alpha_{0}=0$.
\subsection{First-order system}

To start our analysis, we write the Cauchy problem \eqref{Cauchy} as a first-order system in time. First, we change coordinates and map the backward light-cone 
\begin{align*}
C_{T}=\{ (t,x): 0 \leq t <T,~ |x| \leq T-t \} = \bigcup_{t \in [0,T)} \{ t \} \times \mathbb{B}_{T-t}^{d} 
\end{align*}
diffeomorphically into the cylinder 
\begin{align*}
\mathcal{C}:=\{ (\tau,\xi): 0 \leq \tau < +\infty,~|\xi| \leq 1\}= [0,\infty)\times\mathbb{B}^{d}.
\end{align*} 
Specifically, we introduce the similarity coordinates 
\begin{align*}
(t,x) \longmapsto \mu(t,x):=\left( \tau (t,x),\xi(t,x) \right):= \left( \log \left( \frac{T}{T-t} \right),  \frac{x}{T-t} \right)
\end{align*}
and derivatives translate according to
\begin{align*}
\partial _{t} & = \frac{e^{\tau}}{T} \left( \partial _{\tau} + \xi ^{j}  \partial _{\xi ^{j}} \right), \\
\partial _{t}^{2} & =  \frac{e^{2 \tau}}{T^{2}} \left( \partial _{\tau}^{2} +\partial _{\tau} + 2 \xi ^{j}  \partial _{\xi ^{j}} \partial _{\tau} + \xi ^{j} \xi ^{k} \partial _{\xi ^{i}}  \partial _{\xi ^{k}}  +2 \xi ^{j} \partial _{\xi ^{j}} \right), \\
\partial _{x^{j}} & = \frac{e^{\tau}}{T} \partial _{\xi ^{j}}, \\
\partial _{x^{j}} \partial _{x_{j}} & = \frac{e^{2 \tau}}{T^{2}}\partial _{\xi ^{j}}  \partial _{\xi _{j}}.
\end{align*}
Notice in particular that the blowup time $T$ is mapped to $\infty$. Now, equation \eqref{cubicwave} can be written equivalently as 
\begin{align*}
 \frac{e^{2T}}{T^{2}} \Big(- \partial _{\tau}^{2} - \partial _{\tau} - 2 \xi ^{j}   \partial _{\xi ^{j}} & \partial _{\tau} + \left( \delta ^{jk} -\xi ^{j} \xi ^{k}  \right) \partial _{\xi ^{j}} \partial _{\xi ^{k}} - 2  \xi ^{j}  \partial _{\xi ^{j}}  \Big) U(\tau,\xi) = 
- U^{3}(\tau,\xi) ,
\end{align*}
for $U:=u \circ \mu ^{-1}$. Next, we remove the $\tau-$dependent weight on the left hand side by rescaling,
\begin{align*}
\psi(\tau,\xi):=T e^{-\tau} U(\tau,\xi),
\end{align*}
which implies
\begin{align*} 
\Big( \partial _{\tau}^{2} +3 \partial _{\tau} + 2 \xi ^{j}   \partial _{\xi ^{j}} \partial _{\tau} - ( \delta ^{jk}  -  \xi ^{j} \xi ^{k}  ) \partial _{\xi ^{j}} \partial _{\xi ^{k}} +4  \xi ^{j}  \partial _{\xi ^{j}} +2  \Big) \psi(\tau,\xi) = 
  \psi^{3}(\tau,\xi).
\end{align*}
Finally, we set
\begin{align*}
\left( \mathbf{\Psi} (\tau) \right)(\xi) :=
\begin{pmatrix}
\psi _{1} (\tau,\xi) \\
\psi _{2} (\tau,\xi)
\end{pmatrix} 
:= 
\begin{pmatrix}
\psi  (\tau,\xi) \\
\partial _{\tau} \psi (\tau,\xi) + \xi ^{j} \partial _{\xi ^{j}} \psi (\tau,\xi) + \psi (\tau,\xi)
\end{pmatrix}
\end{align*}
which yields
\begin{align} \label{systempsi}
\partial _{\tau}  \mathbf{\Psi} (\tau) =\widetilde{\mathbf{L}} \left( \mathbf{\Psi} (\tau)  \right) + \mathbf{N} \left( \mathbf{\Psi} (\tau)  \right) 
\end{align}
where
\begin{align*}
&\widetilde{ \mathbf{L} } \left( \mathbf{u}  \right) (\xi):=
\begin{pmatrix}
- \xi \cdot \nabla u_{1}(\xi) - u_{1}(\xi) + u_{2} (\xi) \\
\Delta ^{\mathbb{R}^{d}} u _{1} (\xi) - \xi \cdot \nabla u_{2}(\xi) - 2 u_{2} (\xi)
\end{pmatrix}, \\[5pt]
& \mathbf{N} \left( \mathbf{u} \right)(\xi):=
 \begin{pmatrix}
0 \\
u_{1}^{3}(\xi)  
\end{pmatrix}.
\end{align*}

\subsection{Static blowup solution}

Now, starting from \eqref{blowsol}, we switch to similarity coordinates and rescale the function appropriately as before to find a $d-$parameter family $\mathbf{\Psi}_{\alpha}$ of static blowup solutions to \eqref{systempsi}, i.e.,
\begin{align}  \label{staticbold}
\mathbf{\Psi}_{\alpha}(\xi):= \begin{pmatrix}
\psi _{\alpha}(\xi) \\
\xi^{j} \partial _{j} \psi_{\alpha}(\xi)+\psi_{\alpha}(\xi),
\end{pmatrix}
\end{align}
where $\psi_{\alpha}$ is defined in \eqref{static}. We emphasize that there is no trace of the blowup time $T$ in the definition of $\psi _{\alpha}$.


\section{The linear free evolution in the backward light-cone}

In this section, we focus on the evolution of the free linear equation and obtain a useful decay estimate for the solution operator. To this end, we need to find a norm 
\begin{align*}
\left \| \cdot \right \|: \mathcal{H}  \longrightarrow \mathbb{R}
\end{align*}
on the function space
\begin{align*}
\mathcal{H}:=H^{\frac{d+1}{2}} \left( \mathbb{B}^{d} \right) \times H^{\frac{d-1}{2}} \left( \mathbb{B}^{d} \right)
\end{align*}
which yields the sharp decay for the free evolution. Specifically, we define
\begin{align*}
\mathcal{D} \big(\widetilde{ \mathbf{L}} \big):= C^{\frac{d+3}{2}} \big( \overline{ \mathbb{B}^{d} } \big) \times C^{\frac{d+1}{2}} \big(  \overline{ \mathbb{B}^{d}  }\big)
\end{align*}
and work towards proving the following result.
\begin{prop} \label{Lumer}
The free operator $\widetilde{\mathbf{L}}:\mathcal{D} \big(\widetilde{ \mathbf{L}} \big) \subseteq \mathcal{H} \longrightarrow \mathcal{H}$ is densely defined, closable and its closure 
$\mathbf{L}:\mathcal{D} \big( \mathbf{L} \big) \subseteq \mathcal{H} \longrightarrow \mathcal{H}$ generates a strongly continuous one-parameter semigroup of bounded operators 
$\mathbf{S}: [0,\infty) \longrightarrow \mathcal{B} \left( \mathcal{H} \right)$ which satisfies the decay estimate
\begin{align*}
\| \mathbf{S} \left( \tau \right) \| \leq M e^{-\tau} 
\end{align*}
for all $\tau \geq 0$ and for some constant $M \geq 1$.
\end{prop}
To proceed, we fix $d=5$ and construct a suitable inner product on $ \mathcal{H}=H^{3} \left( \mathbb{B}^{5} \right) \times H^{2} \left( \mathbb{B}^{5} \right)$.
\subsection{Inner Product} \label{InnerProduct}
We define 
\begin{align*}
\mathcal{\widetilde{H}}=C^{3} \big ( \overline{ \mathbb{B}^{5} } \big) \times C^{2} \big ( \overline{ \mathbb{B}^{5} } \big)
\end{align*}
and consider the sesquilinear forms 
\begin{align*} 
& \left( \mathbf{u} \big{|} \mathbf{v}  \right)_{1}:= \int _{\mathbb{B}^{5}} 
\partial _{i} \partial _{j} \partial _{k} u_{1} (\xi)
\overline{ \partial ^{i} \partial ^{j} \partial ^{k} v_{1} (\xi) } d \xi +  \int _{\mathbb{B}^{5}} \partial _{i} \partial _{j} u_{2} (\xi) \overline{\partial ^{i} \partial ^{j} v_{2}(\xi) } d \xi 
+ \int _{\mathbb{S}^{4}} \partial _{i} \partial _{j} u_{1} (\omega) \overline{ \partial ^{i} \partial ^{j} v_{1} (\omega) } d \sigma (\omega), \\
& \left( \mathbf{u} \big{|} \mathbf{v}  \right)_{2}:= \int _{\mathbb{B}^{5}} \partial _{i} \partial ^{k} \partial _{k} u_{1}(\xi)  \overline{ \partial ^{i} \partial ^{j} \partial _{j} v_{1} (\xi) } d \xi 
+  \int _{\mathbb{B}^{5}} \partial _{i} \partial _{j} u_{2}(\xi)  \overline{ \partial ^{i} \partial ^{j} v_{2} (\xi) } d \xi 
+ \int _{\mathbb{S}^{4}} \partial _{j} u_{2}  (\omega) \overline{ \partial ^{j} v_{2} (\omega) } d \sigma (\omega), \\
& \left( \mathbf{u} \big{|} \mathbf{v}  \right)_{3}:=
%
%
%
%
%
%
%
%
%
%
5
\left( \mathbf{u} \big{|} \mathbf{v}  \right)_{1} + \left( \mathbf{u} \big{|} \mathbf{v}  \right)_{2} + \int _{\mathbb{S}^{4}} u_{2} (\omega) \overline{ v_{2}  (\omega) } d \sigma (\omega), \\
& \left( \mathbf{u} \big{|} \mathbf{v}  \right)_{4}:=\left( \mathbf{u} \big{|} \mathbf{v}  \right)_{1} + \left( \mathbf{u} \big{|} \mathbf{v}  \right)_{2} + \int _{\mathbb{S}^{4}} \partial _{i} u_{1}  (\omega) \overline{ \partial ^{i}  v_{1} (\omega) }d \sigma (\omega),
\end{align*}
for all $\mathbf{u},\mathbf{v} \in \mathcal{\widetilde{H}}$. All these sesquilinear forms are derived from a higher energy of the free wave equation but neither of them defines an inner product on $\mathcal{H}$. To fix this, we also define
\begin{align*}
\left( \mathbf{u} \big{|} \mathbf{v}  \right)_{5}:= 
\left(  \int _{\mathbb{S}^{4}} \zeta \left(\omega,\mathbf{u} (\omega) \right) d \sigma (\omega)  \right) 
\left(  \int _{\mathbb{S}^{4}} \overline{ \zeta \left(\omega,\mathbf{v}(\omega)  \right) } d \sigma (\omega)  \right) 
\end{align*}
where
\begin{align*}
\zeta \left(\omega,\mathbf{w}(\omega)  \right):= D_{5} w_{1}(\omega) + \tilde{D}_{5} w_{2}(\omega)
\end{align*}
and
\begin{align*}
& D_{5} w_{1}(\omega) :=  \omega ^{i}  \omega ^{j} \partial _{i} \partial _{j} w_{1} (\omega) + 5 \omega ^{i} \partial _{i} w_{1} (\omega)+3 w_{1} (\omega), \\ 
& \tilde{D}_{5} w_{2} (\omega):=   \omega ^{j} \partial _{j} w_{2} (\omega)+ 3 w_{2}(\omega).
\end{align*}
Finally, let
\begin{align} \label{sesquilinear1}
\left( \cdot \big{|} \cdot \right): \mathcal{\widetilde{H}} \times \mathcal{\widetilde{H}} \longrightarrow \mathbb{R},\quad \left( \mathbf{u} \big{|} \mathbf{v}  \right):=\sum _{i =1}^{5}  \left( \mathbf{u} \big{|} \mathbf{v}  \right)_{i}
\end{align}
and
\begin{align} \label{sesquilinear2}
\| \cdot \|: \mathcal{\widetilde{H}} \longrightarrow \mathbb{R},\quad \| \cdot \|:=\sqrt{ \left( \cdot \big{|} \cdot \right) }.
\end{align}
Now, we will show that the norm \eqref{sesquilinear2} induced by the inner product \eqref{sesquilinear1} defines indeed a norm equivalent to $\| \cdot \|_{  H^{3} \left( \mathbb{B}^{5} \right) \times H^{2} \left( \mathbb{B}^{5} \right) }$. However, we first need the following technical result.
\begin{lemma} \label{Hnormequiv}
For all $(u_{1},u_{2}) \in \mathcal{\widetilde{H}}$, we have
\begin{align*}
& \| u_{1} \| _{H^{3} \left( \mathbb{B}^{5} \right)} \simeq \| \partial ^{3} u_{1} \|_{L^{2} \left( \mathbb{B}^{5} \right) } + \| \partial^{2} u_{1} \|_{L^{2} \left( \mathbb{S}^{4} \right) } + 
\| \partial u_{1} \|_{L^{2} \left( \mathbb{S}^{4} \right) } + \|  u_{1} \|_{L^{2} \left( \mathbb{S}^{4} \right) },\\
& \| u_{2} \| _{H^{2} \left( \mathbb{B}^{5} \right)} \simeq \| \partial^{2} u_{2} \|_{L^{2} \left( \mathbb{B}^{5} \right) } + 
\| \partial u_{2} \|_{L^{2} \left( \mathbb{S}^{4} \right) } + \|  u_{2} \|_{L^{2} \left( \mathbb{S}^{4} \right) }.
\end{align*}
\end{lemma}
\begin{proof}
The process is the same for both estimates and so we illustrate it on the second estimate only. Note that, for a generic function $f \in L^{2}(\mathbb{B}^{5})$, we have
\begin{align*}
\| f \|_{L^{2}(\mathbb{B}^{5})}^{2} & = \int_{0}^{1} \int _{\mathbb{S}^{4}} r^{4} | f (r \omega) |^{2} d \sigma (\omega) dr.
\end{align*}
By density, it suffices to consider $u_{2} \in C^{\infty}( \overline{ \mathbb{B}^{5}}) $. Now, the fundamental theorem of calculus, Jensen's inequality and integration by parts yield
\begin{align*}
 r^{4} | u_{2} (r \omega) |^{2}&  = \left | \int_{0}^{r} \partial_{s} \left(  s^2  u_{2} (s \omega)  \right) ds \right |^{2}  \leq  \left(  \int_{0}^{r} \left | \partial_{s} \left(  s^2 u_{2} (s \omega)  \right)  \right | ds  \right)^{2} \\
&  \leq r  \int_{0}^{r} \left | \partial_{s} \left(  s^2  u_{2} (s \omega)  \right)  \right |^{2} ds \leq  \int_{0}^{1} \left | \partial_{s} \left(  s^2  u_{2} (s \omega)  \right)  \right |^{2} ds \\
& =  \int_{0}^{1} \left |   2 s  u_{2} (s \omega) +    s^2 \partial_{s}  u_{2} (s \omega)    \right |^{2} ds \\
  &=  \int_{0}^{1} \Big(    4 s^2 | u_{2} (s \omega) |^2 +    s^4 \left | \partial_{s}  u_{2} (s \omega) \right |^2 + 
  2 s^3 \left(  u_{2} (s \omega) \overline{ \partial_{s}  u_{2} (s \omega) }  + \overline{  u_{2} (s \omega) } \partial_{s}  u_{2} (s \omega)  \right)  \Big) ds \\
  & =  \int_{0}^{1} \left ( 4 s^2 | u_{2} (s \omega) |^{2} + s^{4}  | \partial_{s}  u_{2} (s \omega) |^{2} + 2 s^3 \partial_{s}   | u_{2} (s \omega) |^{2}  \right ) ds \\
& =  2 | u_{2} (\omega) |^{2} +  \int_{0}^{1} \Big( - 2 s^2 | u_{2} (s \omega) |^2  +  s^{4} | \partial _{s} u_{2} (s \omega) |^{2} \Big) ds \\
 & \leq 2 | u_{2} (\omega) |^{2} +  \int_{0}^{1}  s^{4} | \partial _{s} u_{2} (s \omega) |^{2} ds \\
  & = 2 | u_{2} (\omega) |^{2} +  \int_{0}^{1}  s^{4} | \omega ^{j} \partial _{j} u_{2} (s \omega) |^{2} ds \\
&  \leq 2 | u_{2} (\omega) |^{2} +  \int_{0}^{1}  s^{4} | \partial u_{2} (s \omega) |^{2} ds.
\end{align*}
Integrating this inequality with respect to $r \in [0,1]$ and $\omega \in \mathbb{S}^{4}$ yields the estimate
\begin{align*}
\left \| u_{2} \right \|_{L^{2}(\mathbb{B}^{5})} \lesssim \left \| \partial u_{2} \right \|_{L^{2}(\mathbb{B}^{5})} + \left \| u_{2} \right \|_{L^{2}(\mathbb{S}^{4})}.
\end{align*}
Replacing $u_{2}$ by $\partial _{i} u_{2}$, we find
\begin{align*}
 r^{4} | \partial _{i} u_{2} (r \omega) |^{2}   \leq 2 | \partial _{i} u_{2} (\omega) |^{2} +  \int_{0}^{1}  s^{4} | \partial \partial _{i} u_{2} (s \omega) |^{2} ds,
\end{align*}
for all $i \in \{ 1,2,3,4,5 \}$, and hence
\begin{align*}
\left \| \partial u_{2} \right \|_{L^{2}(\mathbb{B}^{5})} \lesssim \left \| \partial ^{2} u_{2} \right \|_{L^{2}(\mathbb{B}^{5})} + \left \| \partial u_{2} \right \|_{L^{2}(\mathbb{S}^{4})}.
\end{align*}
In summary, we get
\begin{align*}
\left \| u_{2} \right \|_{H^{2}(\mathbb{B}^{5})} \lesssim \left \| \partial ^{2} u_{2} \right \|_{L^{2}(\mathbb{B}^{5})} + \left \| \partial u_{2} \right \|_{L^{2}(\mathbb{S}^{4})} +  \left \| u_{2} \right \|_{L^{2}(\mathbb{S}^{4})}.
\end{align*}
This concludes the proof since the reverse inequality is a direct consequence of the trace inequality (see Theorem 1, page 258, \cite{Evans}).
\end{proof}
\begin{lemma} \label{Sobolevequiv}
The sesquilinear form $\left( \cdot \big{|} \cdot \right)$ in \eqref{sesquilinear1} defines an inner product on $\mathcal{\widetilde{H}}$. Furthermore, the completion of $\mathcal{\widetilde{H}}$ is a Hilbert space which is 
equivalent to $\mathcal{H}$.
\end{lemma}
\begin{proof}
From \eqref{sesquilinear1} and \eqref{sesquilinear2}, we get
\begin{align*}
\| \mathbf{u} \|^{2} & \simeq  \int _{\mathbb{B}^{5}} \partial_{i} \partial_{j} \partial_{k} u_{1}(\xi) \overline{  \partial^{i} \partial^{j} \partial^{k} u_{1}(\xi)  }  d \xi + 
 \int _{\mathbb{B}^{5}}  
 \partial_{i} \partial_{j} u_{2}(\xi) \overline{  \partial^{i} \partial^{j} u_{2}(\xi)  } d \xi \\
 & + \int _{\mathbb{S}^{4}} 
\partial_{i} \partial_{j} u_{1}(\omega) \overline{  \partial^{i} \partial^{j} u_{1}(\omega)  }
 d \sigma (\omega) 
 + \int _{\mathbb{S}^{4}} 
\partial_{i}  u_{1}(\omega) \overline{  \partial^{i} u_{1}(\omega)  }
 d \sigma (\omega) \\
 & +  \int _{\mathbb{B}^{5}}
 \partial_{i} \partial_{j}  \partial^{j} u_{1}(\xi) \overline{  \partial^{i}  \partial^{k} \partial_{k} u_{1}(\xi)  } d \xi + 
\int _{\mathbb{S}^{4}}
 \partial_{j} u_{2} (\omega) \overline{ \partial^{j} u_{2} (\omega) }
  d \sigma (\omega) \\
& + \int _{\mathbb{S}^{4}} \big| u_{2}(\omega) \big|^{2}   ~d \sigma (\omega)
+ \left | \int _{\mathbb{S}^{4}} \zeta \left( \omega, \mathbf{u}(\omega) \right) d \sigma (\omega) \right |^{2},
\end{align*}
for all $\mathbf{u} \in \mathcal{\widetilde{H}}$. We need to show that $\| \mathbf{u} \| \simeq \| \mathbf{u} \|_{H^{3}\left ( \mathbb{B}^{5} \right) \times H^{2} \left ( \mathbb{B}^{5} \right)}$, for all $\mathbf{u} \in \mathcal{\widetilde{H}}$. First, note that it suffices to prove 
$ \| \mathbf{u} \|_{H^{3}\left ( \mathbb{B}^{5} \right) \times H^{2} \left ( \mathbb{B}^{5} \right)} \lesssim \| \mathbf{u} \| $ since the reverse inequality is a direct consequence of the trace theorem (see Theorem 1, page 258, \cite{Evans}) and the embedding $L^{2} \left( \mathbb{S}^{4} \right) \xhookrightarrow{} L^{1} \left( \mathbb{S}^{4} \right)$. From Lemma \ref{Hnormequiv}, we get
\begin{align*}
\| \mathbf{u} \|_{H^{3}\left ( \mathbb{B}^{5} \right) \times H^{2} \left ( \mathbb{B}^{5} \right)} \lesssim \| \mathbf{u} \| + \| u_{1} \|_{L^{2} \left( \mathbb{S}^{4} \right)}
\end{align*}
and the Poincare inequality on the $4-$sphere (see Theorem 2.10, page 40, \cite{Hebey}), 
\begin{align*}
\left \| u_{1} - \frac{2}{\pi ^{2}} \int _{\mathbb{S}^{4}} u_{1}(\omega) d \sigma (\omega) \right \|_{L^{2} \left( \mathbb{S}^{4} \right)} \lesssim \| \nabla  u_{1} \|_{L^{2} \left( \mathbb{S}^{4} \right)},
\end{align*}
together with the embedding $L^{2} \left( \mathbb{S}^{4} \right) \xhookrightarrow{} L^{1} \left( \mathbb{S}^{4} \right)$ yield
\begin{align*}
 \| u_{1} \|_{L^{2} \left( \mathbb{S}^{4} \right)} & \lesssim  \| \nabla  u_{1} \|_{L^{2} \left( \mathbb{S}^{4} \right)} + \left | \int _{\mathbb{S}^{4}} u_{1} (\omega) d \sigma (\omega) \right | \\
& \leq \| \mathbf{u} \| + \left |   \int_{\mathbb{S}^{4}} u_{1} (\omega)d \sigma (\omega)  \right | \\
& \lesssim  \| \mathbf{u} \| + \left |   \int_{\mathbb{S}^{4}} \zeta \left( \omega, \mathbf{u}(\omega) \right) d \sigma (\omega)  \right |
+ \left |   \int_{\mathbb{S}^{4}} \omega ^{i} \omega ^{j} \partial _{i} \partial _{j} u_{1} (\omega) d \sigma (\omega)  \right |  \\
& +  \left |   \int_{\mathbb{S}^{4}}\omega ^{j} \partial _{j} u_{1} (\omega) d \sigma (\omega)  \right | 
+  \left |   \int_{\mathbb{S}^{4}} \omega ^{j} \partial _{j} u_{2} (\omega) d \sigma (\omega)  \right | 
+  \left |   \int_{\mathbb{S}^{4}}  u_{2}(\omega) d \sigma (\omega)  \right | \\
& \lesssim  \| \mathbf{u} \| 
+ \left (  \int_{\mathbb{S}^{4}} \big| \partial^{2} u_{1}(\omega) \big|^{2} d \sigma (\omega)  \right )^{\frac{1}{2}}  
+  \left (   \int_{\mathbb{S}^{4}} \big|   \partial u_{1} (\omega) \big|^{2} d \sigma (\omega)  \right )^{\frac{1}{2}} \\
& +  \left (   \int_{\mathbb{S}^{4}} \big|  \partial  u_{2}(\omega)  \big|^{2} d \sigma (\omega)  \right )^{\frac{1}{2}} 
+  \left (   \int_{\mathbb{S}^{4}} \big| u_{2}(\omega) \big|^{2}~ d \sigma (\omega)  \right )^{\frac{1}{2}} \\
& \lesssim  \| \mathbf{u} \|, 
\end{align*}
which concludes the proof.
\end{proof}
\subsection{Free evolution and decay in time}
Now, we focus on the proof of Proposition \ref{Lumer} and show that a semigroup (solution operator) is generated and decays in time with a sharp decay estimate. We specify the domain of $\widetilde{\mathbf{L}}$,
\begin{align} \label{domain}
\mathcal{D} \big( \widetilde{\mathbf{L}} \big):= C^{4} \big( \overline{ \mathbb{B}^{5} } \big) \times C^{3} \big( \overline{ \mathbb{B}^{5} } \big).
\end{align}
To prove Proposition \ref{Lumer}, we intend to apply the Lumer-Phillips theorem (see Theorem 3.15, page 83, \cite{EngNag00}). The following two Lemmas constitute the key property of the sesquilinear forms defined above and verify the first part of the hypothesis of the Lumer-Phillips theorem. First, we define
\begin{align*}
\| \cdot \|_{j}: \mathcal{\widetilde{H}} \longrightarrow \mathbb{R},\quad \| \cdot \|_{j}:=\sqrt{ \left( \cdot \big{|} \cdot \right)_{j} },
\end{align*}
for all $j \in \{1,2,3,4,5\}$, where the sesquilinear forms $\left( \cdot \big{|} \cdot \right)_{j}$ are defined in section \ref{InnerProduct}.
\begin{lemma} \label{Lumer1a}
For all $\mathbf{u} \in \mathcal{D} \big( \widetilde{ \mathbf{L} } \big)$ and $i \in \{1,2,3,4 \}$, we have
\begin{align*}
\mathrm{Re} \big( \widetilde{ \mathbf{L} }  \mathbf{u}  \big{|} \mathbf{u}  \big)_{i} \leq - \frac{3}{2} \| \mathbf{u} \|_{i}^{2}.
\end{align*}
\end{lemma}
\begin{proof}
To begin with, fix an arbitrary $\mathbf{u} \in C^{4} \big( \overline{ \mathbb{B}^{5} } \big) \times C^{3} \big( \overline{ \mathbb{B}^{5} } \big)$. On the one hand, the divergence theorem implies
\begin{align*}
\mathrm{Re} \int _{\mathbb{B}^{5}} \partial _{i} \partial _{j} \partial _{k} \big( \widetilde{ \mathbf{L} } \mathbf{u} \big)_{1} ( \xi )  \overline{ \partial ^{i} \partial ^{j} \partial ^{k} u_{1}  ( \xi ) } d \xi & = 
-\frac{3}{2} \int _{\mathbb{B}^{5}} 
\partial _{i} \partial _{j} \partial _{k} u_{1} (\xi) \overline{ \partial ^{i} \partial ^{j} \partial ^{k} u_{1} (\xi) } d \xi \\
& - \frac{1}{2} \int _{\mathbb{S}^{4}} 
\partial _{i} \partial _{j} \partial _{k} u_{1} (\omega) \overline{  \partial ^{i} \partial ^{j} \partial ^{k} u_{1} (\omega) }
d \sigma (\omega) \\
& +\mathrm{Re} \int _{\mathbb{B}^{5}}  \partial _{i} \partial _{j} \partial _{k} u_{1}( \xi )  \overline{  \partial ^{i} \partial ^{j} \partial ^{k} u_{2}  ( \xi ) }d \xi, \\ 
\mathrm{Re} \int _{\mathbb{B}^{5}} \partial _{i} \partial _{j} \big( \widetilde{ \mathbf{L} }  \mathbf{u}  \big)_{2} ( \xi ) \overline{ \partial ^{i} \partial ^{j} u_{2} ( \xi ) } d \xi & =  
 - \frac{3}{2} \int _{\mathbb{B}^{5}} 
 \partial _{i} \partial _{j} u_{2} (\xi) \overline{ \partial ^{i} \partial ^{j} u_{2} (\xi)  } d \xi \\
 & - \mathrm{Re} \int _{\mathbb{B}^{5}}  \partial _{i} \partial _{j} \partial _{k} u_{1} ( \xi )  \overline{  \partial ^{i} \partial ^{j} \partial ^{k} u_{2} ( \xi ) }d \xi \\
&+\mathrm{Re} \int _{\mathbb{S}^{4}} \omega ^{k}  \partial _{k} \partial _{i} \partial _{j} u_{1} (\omega) \overline{ \partial ^{i} \partial ^{j} u_{2} (\omega) }d \sigma( \omega) \\
& - \frac{1}{2}   \int _{\mathbb{S}^{4}}
 \partial _{i} \partial _{j} u_{2} (\omega) \overline{ \partial ^{i} \partial ^{j}  u_{2} (\omega) }
 d \sigma (\omega) 
\end{align*}
and, on the other hand, we have
\begin{align*}
\mathrm{Re} \int _{\mathbb{S}^{4}} \partial _{i} \partial _{j} \big( \widetilde{ \mathbf{L} } \mathbf{u} \big)_{1} (\omega)  \overline{ \partial ^{i} \partial ^{j} u_{1} (\omega)} d \sigma (\omega) & = 
 - 3 \int _{\mathbb{S}^{4}} 
 \partial _{i} \partial _{j} u_{1} (\omega) \overline{ \partial ^{i} \partial ^{j} u_{1} (\omega)  }
 d \sigma( {\omega})  \\
 & -  \mathrm{Re} \int _{\mathbb{S}^{4}} \omega ^{k}  \partial _{k} \partial _{i} \partial _{j} u_{1} (\omega)  \overline{  \partial ^{i} \partial ^{j} u_{1} (\omega)  }d \sigma( \omega) \\
 &+ \mathrm{Re} \int _{\mathbb{S}^{4}} \partial _{i} \partial _{j} u_{1} (\omega) \overline{ \partial ^{i} \partial ^{j} u_{2} (\omega) }d \sigma( \omega). 
\end{align*}
Hence, we obtain
\begin{align*}
\mathrm{Re}~ \big( \widetilde{ \mathbf{L} }  \mathbf{u}  \big{|} \mathbf{u}  \big)_{1}  + \frac{3}{2} \| \mathbf{u} \|_{1}^{2} & =
- \int _{\mathbb{S}^{4}} 
\partial _{i} \partial _{j}  u_{1} (\omega) \overline{ \partial ^{i} \partial ^{j} u_{1} (\omega) }
d \sigma (\omega) +   \int _{\mathbb{S}^{4}} A(\omega) d \sigma (\omega),
\end{align*}
where
\begin{align*}
A(\omega)&:= -\frac{1}{2} 
\partial _{i} \partial _{j} \partial _{k} u_{1} (\omega) \overline{ \partial ^{i} \partial ^{j} \partial ^{k} u_{1} (\omega) }
-\frac{1}{2} 
\partial _{i} \partial _{j} u_{2} (\omega) \overline{ \partial ^{i} \partial ^{j}  u_{2} (\omega) }
- \frac{1}{2}
 \partial _{i} \partial _{j}  u_{1} (\omega) \overline{ \partial ^{i} \partial ^{j}  u_{1} (\omega) }
  \\
& +\mathrm{Re} \Big( \omega ^{k} \partial _{k} \partial _{i} \partial _{j} u_{1} (\omega) \overline{ \partial ^{i} \partial ^{j} u_{2}(\omega) } \Big) + 
\mathrm{Re} \Big(\partial _{i} \partial _{j} u_{1} (\omega) \overline{ \partial ^{i} \partial ^{j} u_{2} (\omega) } \Big) \\
& -  \mathrm{Re} \Big( \omega ^{k} \partial _{k} \partial _{i} \partial _{j} u_{1}(\omega)  \overline{ \partial ^{i} \partial ^{j} u_{1}(\omega) } \Big). 
\end{align*}
Now, we use the inequality
\begin{align} \label{ReIneq}
\mathrm{Re} ( a \overline{b} ) + \mathrm{Re} \left( a \overline{c} \right) - \mathrm{Re} \left( b \overline{c} \right) \leq \frac{1}{2} |a|^{2} + \frac{1}{2} |b|^{2} + \frac{1}{2} |c|^{2},
\end{align}
which holds for all $a,b,c \in \mathbb{C}$, together with Cauchy-Schwarz inequality
\begin{align*}
\left | \sum_{k} \omega^{k} \partial_{k} \partial _{i} \partial _{j} u_{1}(\omega) \right |^{2}  
\leq  \sum_{k} \left( \omega ^{k} \right)^{2} \sum_{k} \left | \partial_{k} \partial _{i} \partial _{j} u_{1}(\omega) \right |^{2} 
= \sum_{k} \left | \partial_{k}\partial _{i} \partial _{j} u_{1}(\omega) \right |^{2}
\end{align*}
to obtain $A(\omega) \leq 0$ for all $\omega \in \mathbb{S}^{4}$ and the desired estimate for $ \big( \widetilde{ \mathbf{L} }  \mathbf{u}  \big{|} \mathbf{u}  \big)_{1}$ follows. For the second estimate, the divergence theorem yields
\begin{align*}
\mathrm{Re} \int _{\mathbb{B}^{5}} \partial _{i} \partial ^{k} \partial _{k} \big( \widetilde{ \mathbf{L} }  \mathbf{u}  \big)_{1} (\xi )  \overline{ \partial _{i} \partial ^{j} \partial _{j} u_{1} (\xi ) } d \xi & = 
-\frac{3}{2} \int _{\mathbb{B}^{5}} 
\partial _{i} \partial ^{j} \partial _{j} u_{1} (\xi) \overline{ \partial ^{i} \partial _{k} \partial ^{k} u_{1} (\xi) } d \xi \\
& - \frac{1}{2} \int _{\mathbb{S}^{4}} 
 \partial _{i} \partial ^{j} \partial _{j} u_{1} (\omega) \overline{ \partial ^{i} \partial _{k} \partial ^{k} u_{1} (\omega) } 
 d \sigma (\omega) \\
& +\mathrm{Re} \int _{\mathbb{B}^{5}}  \partial ^{i} \partial ^{j} \partial _{j} u_{1}(\xi )   \overline{  \partial _{i} \partial ^{k} \partial _{k} u_{2} (\xi )  }d \xi, \\
\mathrm{Re} \int _{\mathbb{B}^{5}} \partial _{i} \partial _{j} \big( \widetilde{ \mathbf{L} }  \mathbf{u} \big)_{2} (\xi )   \overline{ \partial ^{i} \partial ^{j} u_{2}(\xi )  } d \xi & =
-\frac{3}{2} \int _{\mathbb{B}^{5}}
 \partial _{i} \partial _{j}  u_{2} (\xi) \overline{ \partial ^{i} \partial ^{j} u_{2} (\xi) } 
 d \xi \\
& - \frac{1}{2} \int _{\mathbb{S}^{4}} 
\partial _{i} \partial _{j}  u_{2} (\omega) \overline{ \partial ^{i} \partial ^{j} u_{2} (\omega) } 
d \sigma (\omega) \\
& -\mathrm{Re} \int _{\mathbb{B}^{5}}  \partial ^{i} \partial ^{k} \partial _{k} u_{1} (\xi ) \overline{  \partial _{i} \partial ^{j} \partial _{j} u_{2} (\xi ) }d \xi \\
& +\mathrm{Re} \int _{\mathbb{S}^{4}} \omega ^{j}  \partial _{i }\partial _{j} u_{2} (\omega) \overline{  \partial ^{i} \partial ^{k} \partial _{k} u_{1} (\omega) }d \sigma(\omega),
\end{align*}
and, in addition, we have
\begin{align*}
\mathrm{Re} \int _{\mathbb{S}^{4}} \partial _{j} \big( \widetilde{ \mathbf{L} }  \mathbf{u}  \big)_{2}(\omega) \overline{ \partial ^{j} u_{2} (\omega) } d \sigma (\omega) & =
 - 3 \int _{\mathbb{S}^{4}} 
 \partial _{i}   u_{2} (\omega) \overline{ \partial ^{i}  u_{2} (\omega) } 
 d \sigma( {\omega})  \\
& - \mathrm{Re} \int _{\mathbb{S}^{4}} \omega ^{k}  \partial _{k} \partial _{j} u_{2} (\omega) \overline{  \partial ^{j} u_{2} (\omega)  }d \sigma( {\omega}) \\
 &+ \mathrm{Re} \int _{\mathbb{S}^{4}} \partial ^{i} \partial _{i} \partial _{j} u_{1}(\omega)  \overline{ \partial ^{j} u_{2}(\omega) }d \sigma( {\omega}).  
\end{align*}
Therefore, we get
\begin{align*}
\mathrm{Re}~ \big( \widetilde{ \mathbf{L} }  \mathbf{u}  \big{|} \mathbf{u}  \big)_{2}  + \frac{3}{2} \| \mathbf{u} \|_{2}^{2} & =
- \int _{\mathbb{S}^{4}} 
\partial _{i} u_{2} (\omega) \overline{ \partial ^{i}  u_{2} (\omega) } 
d \sigma (\omega) +   \int _{\mathbb{S}^{4}} B(\omega) d \sigma (\omega),
\end{align*}
where
\begin{align*}
B(\omega)&:= -\frac{1}{2}
 \partial _{i}  \partial ^{j}  \partial _{j}  u_{1} (\omega) \overline{ \partial ^{i}   \partial ^{k}  \partial _{k} u_{1} (\omega) } 
 -\frac{1}{2} 
 \partial _{i}  \partial_{j} u_{2} (\omega) \overline{ \partial ^{i} \partial^{j} u_{2} (\omega) } 
 -\frac{1}{2}
   \partial _{i} u_{2} (\omega) \overline{ \partial ^{i}u_{2} (\omega) }  \\
& +\mathrm{Re} \Big( \omega ^{j} \partial _{j} \partial _{i} u_{2}(\omega) \overline{ \partial ^{i} \partial ^{k} \partial _{k} u_{1}(\omega) } \Big) \\
& + \mathrm{Re} \Big(\partial ^{j} u_{2}(\omega) \overline{ \partial ^{i} \partial _{i} \partial _{j} u_{1} (\omega) } \Big)
-  \mathrm{Re} \Big( \omega ^{k} \partial _{k} \partial _{j} u_{2} (\omega) \overline{ \partial ^{j} u_{2} (\omega) } \Big). 
\end{align*}
As before, we use \eqref{ReIneq} together with Cauchy-Schwarz inequality
\begin{align*}
\left | \sum_{k} \omega^{k} \partial_{k} \partial _{i} u_{2}(\omega) \right |^{2}  
\leq  \sum_{k} \left( \omega ^{k} \right)^{2} \sum_{k} \left | \partial_{k} \partial _{i} u_{2}(\omega) \right |^{2} 
= \sum_{k} \left | \partial_{k}\partial _{i} u_{2}(\omega) \right |^{2}
\end{align*}
to get $B(\omega) \leq 0$ for all $\omega \in \mathbb{S}^{4}$ and the claim for $ \big(  \widetilde{ \mathbf{L} }  \mathbf{u}  \big{|} \mathbf{u}  \big)_{2}$ follows. For the third estimate, we use the previous estimates together with Cauchy-Schwarz inequalities
\begin{align*}
& \left | \sum_{i} \partial_{i} \partial^{i} u_{1}(\omega) \right |^{2}  \leq  \sum_{i} 1^{2} \sum_{i} \left | \partial_{i} \partial^{i} u_{1}(\omega) \right |^{2} \leq 5 \sum_{i,j} \left | \partial_{i} \partial_{j} u_{1}(\omega) \right |^{2},  
\\
& \left | \sum_{k} \omega^{k} \partial_{k} u_{2}(\omega) \right |^{2}  \leq  \sum_{k} \left( \omega ^{k} \right)^{2} \sum_{k} \left | \partial_{k} u_{2}(\omega) \right |^{2} = \sum_{k} \left | \partial_{k} u_{2}(\omega) \right |^{2},
\end{align*}
and Young's inequality to obtain
\begin{align*}
\mathrm{Re}~ \big( \widetilde{ \mathbf{L} }  \mathbf{u}  \big{|} \mathbf{u}  \big)_{3}  + \frac{3}{2} \| \mathbf{u} \|_{3}^{2} & =
5 \left( \mathrm{Re}~ \big( \widetilde{ \mathbf{L} }  \mathbf{u}  \big{|} \mathbf{u}  \big)_{1}  + \frac{3}{2} \| \mathbf{u} \|_{1}^{2} \right)+ 
\mathrm{Re}~ \big( \widetilde{ \mathbf{L} }  \mathbf{u}  \big{|} \mathbf{u}  \big)_{2}  + \frac{3}{2} \| \mathbf{u} \|_{2}^{2} + \\
& +\mathrm{Re}~  \int _{\mathbb{S}^{4}} \left( \big( \widetilde{ \mathbf{L} }  \mathbf{u}(\omega)  \big)_{2} \overline{ u_{2} (\omega) } + \frac{3}{2} \big| u_{2}(\omega) \big|^{2}  \right) d \sigma (\omega) \\
& \leq -5  \int _{ \mathbb{S}^{4} }  
 \partial _{i}  \partial _{j}  u_{1} (\omega) \overline{ \partial ^{i}   \partial ^{j}  u_{1} (\omega) } 
d \sigma (\omega) 
- \int _{ \mathbb{S}^{4} } 
  \partial _{i}   u_{2} (\omega) \overline{ \partial ^{i}  u_{2} (\omega) } 
 d \sigma (\omega) \\
& + \mathrm{Re}~ \int _{\mathbb{S}^{4}} \left( \partial ^{i} \partial _{i} u_{1}(\omega) \overline{u_{2} (\omega) }  -   \omega ^{k} \partial _{k} u_{2}(\omega) \overline{  u_{2} (\omega)} 
 -\frac{1}{2} 
 \big| u_{2}(\omega) \big|^{2} \right) d \sigma (\omega) \\
 & \leq - \int _{ \mathbb{S}^{4} }  
\partial _{i}  \partial ^{i}  u_{1} (\omega) \overline{ \partial _{j}   \partial ^{j}  u_{1} (\omega) } 
d \sigma (\omega) 
- \int _{ \mathbb{S}^{4} } 
  \partial _{i}   u_{2} (\omega) \overline{ \partial ^{i}  u_{2} (\omega) } 
 d \sigma (\omega) \\
& + \mathrm{Re}~ \int _{\mathbb{S}^{4}} \left( \partial ^{i} \partial _{i} u_{1}(\omega) \overline{u_{2} (\omega) }  -   \omega ^{k} \partial _{k} u_{2}(\omega) \overline{  u_{2} (\omega)} 
 -\frac{1}{2} 
 \big| u_{2}(\omega) \big|^{2} \right) d \sigma (\omega) \\
&= -\frac{1}{2}  \int _{ \mathbb{S}^{4} } 
 \partial _{i}  \partial ^{i}  u_{1} (\omega) \overline{ \partial _{j}   \partial ^{j}  u_{1} (\omega) }   
 d \sigma (\omega) -\frac{1}{2}  \int _{ \mathbb{S}^{4} }  
   \partial _{i}  u_{2} (\omega) \overline{ \partial ^{i}   u_{2} (\omega) } 
  d \sigma (\omega) \\
& -\mathrm{Re}~   \int _{ \mathbb{S}^{4} } \partial ^{i} \partial _{i} u_{1} (\omega)  \overline{ \omega ^{k} \partial _{k} u_{2} (\omega) } d \sigma (\omega)
+\int _{\mathbb{S}^{4}} C(\omega) d \sigma (\omega) \\
& \leq \int _{\mathbb{S}^{4}} C(\omega) d \sigma (\omega)
\end{align*}
where
\begin{align*}
C(\omega)&:=  -\frac{1}{2}
\partial _{i}  \partial ^{i}  u_{1} (\omega) \overline{ \partial _{j}   \partial ^{j}  u_{1} (\omega) }  
 -\frac{1}{2} 
  \partial _{i}  u_{2} (\omega) \overline{ \partial ^{i} u_{2} (\omega) } 
 -\frac{1}{2} \big| u_{2}(\omega) \big|^{2} \\
& +\mathrm{Re} \Big( u_{2} (\omega) \overline{ \partial ^{i} \partial _{i} u_{1} (\omega) } \Big) + \mathrm{Re} \Big(\partial ^{i} \partial _{i} u_{1} (\omega) \overline{ \omega ^{k} \partial _{k} u_{2} (\omega) } \Big) 
-  \mathrm{Re} \Big( \omega ^{k} \partial _{k} u_{2}(\omega) \overline{  u_{2}(\omega) } \Big).
\end{align*}
Inequality \eqref{ReIneq} implies $C(\omega)  \leq 0$ for all $\omega \in \mathbb{S}^{4}$ and the claim for $ \big( \widetilde{ \mathbf{L} }  \mathbf{u}  \big{|} \mathbf{u}  \big)_{3}$ follows. Finally, for the last estimate, we use the previous estimates together Cauchy-Schwarz inequality 
\begin{align*}
 \left | \sum_{k} \omega^{k} \partial_{k} \partial_{i} u_{1}(\omega) \right |^{2}  
 \leq  \sum_{k} \left( \omega ^{k} \right)^{2} \sum_{k} \left | \partial_{k} \partial_{i} u_{1}(\omega) \right |^{2} 
 = \sum_{k} \left | \partial_{k} \partial_{i} u_{1}(\omega) \right |^{2}
\end{align*}
and Young's inequality once more to obtain
\begin{align*}
\mathrm{Re}~ \big( \widetilde{ \mathbf{L} }  \mathbf{u}  \big{|} \mathbf{u}  \big)_{4}  + \frac{3}{2} \| \mathbf{u} \|_{4}^{2} & =
\mathrm{Re}~ \big( \widetilde{ \mathbf{L} }  \mathbf{u}  \big{|} \mathbf{u}  \big)_{1}  + \frac{3}{2} \| \mathbf{u} \|_{1}^{2} + 
\mathrm{Re}~ \big( \widetilde{ \mathbf{L} }  \mathbf{u}  \big{|} \mathbf{u}  \big)_{2}  + \frac{3}{2} \| \mathbf{u} \|_{2}^{2} + \\
& +\mathrm{Re}~  \int _{\mathbb{S}^{4}} \left( \partial _{i} \big( \widetilde{ \mathbf{L} }  \mathbf{u}   \big)_{1}  (\omega)\overline{ \partial ^{i} u_{1}(\omega) } + 
\frac{3}{2} 
 \partial _{i} u_{1} (\omega) \overline{ \partial ^{i}  u_{1} (\omega) } \right) d \sigma (\omega) \\
& \leq - \int _{ \mathbb{S}^{4} } 
  \partial _{i}  \partial _{j}  u_{1} (\omega) \overline{ \partial ^{i}   \partial ^{j} u_{1} (\omega) } 
 d \sigma (\omega) 
- \int _{ \mathbb{S}^{4} }  
 \partial _{i}  u_{2} (\omega) \overline{ \partial ^{i} u_{2} (\omega) } 
d \sigma (\omega) \\
& + \mathrm{Re}~ \int _{\mathbb{S}^{4}} \left( \partial _{i} u_{2}(\omega) \overline{\partial ^{i} u_{1} (\omega) }  -   \omega ^{k} \partial _{k} \partial _{i} u_{1} (\omega) \overline{ \partial ^{i} u_{1}(\omega) }  
- \frac{1}{2} 
 \partial _{i}    u_{1} (\omega) \overline{ \partial ^{i}  u_{1} (\omega) } 
\right) d \sigma (\omega) \\
&= -\frac{1}{2}  \int _{ \mathbb{S}^{4} } 
  \partial _{i}  \partial _{j}  u_{1} (\omega) \overline{ \partial ^{i}   \partial ^{j}  u_{1} (\omega) } 
 d \sigma (\omega) 
-\frac{1}{2}  \int _{ \mathbb{S}^{4} } 
  \partial _{i}  u_{2} (\omega) \overline{ \partial ^{i}   u_{2} (\omega) } 
 d \sigma (\omega) \\
& -\mathrm{Re}~   \int _{ \mathbb{S}^{4} } \partial _{i} u_{2} (\omega) \overline{ \omega ^{k} \partial _{k} \partial ^{i} u_{1}(\omega) } d \sigma (\omega)
+\int _{\mathbb{S}^{4}} D(\omega) d \sigma (\omega) \\
& \leq \int _{\mathbb{S}^{4}} D(\omega) d \sigma (\omega),
\end{align*}
where
\begin{align*}
D(\omega)&:=  -\frac{1}{2} 
 \partial _{i}  u_{2} (\omega) \overline{ \partial ^{i} u_{2} (\omega) } 
 -\frac{1}{2} 
  \partial _{i}  u_{1} (\omega) \overline{ \partial ^{i}   u_{1} (\omega) } 
 -\frac{1}{2} 
  \partial _{i}  \partial _{j} u_{1} (\omega) \overline{ \partial ^{i}   \partial ^{j} u_{1} (\omega) } 
 \\
& +\mathrm{Re} \Big( \partial _{i} u_{2} (\omega) \overline{ \partial ^{i} u_{1} (\omega) } \Big) + 
\mathrm{Re} \Big(\partial _{i} u_{2}(\omega) \overline{ \omega ^{k} \partial _{k} \partial ^{i} u_{1} (\omega) } \Big) 
-  \mathrm{Re} \Big( \omega ^{k} \partial _{k} \partial _{i} u_{1}(\omega) \overline{ \partial ^{i} u_{1}(\omega) } \Big).
\end{align*}
As before, \eqref{ReIneq} implies $D(\omega)  \leq 0$ for all $\omega \in \mathbb{S}^{4}$ and the claim for $ \big( \widetilde{ \mathbf{L} }  \mathbf{u}  \big{|} \mathbf{u}  \big)_{4}$ follows.
\end{proof}
\begin{lemma} \label{Lumer1b}
For all $\mathbf{u} \in \mathcal{D} \big( \widetilde{ \mathbf{L} } \big)$, we have
\begin{align*}
\mathrm{Re} \big( \widetilde{ \mathbf{L} }  \mathbf{u}  \big{|} \mathbf{u}  \big)_{5} = -  \| \mathbf{u} \|_{5}^{2}.
\end{align*}
\end{lemma}
\begin{proof}
Fix $\mathbf{u} \in C^{4} \big( \overline{ \mathbb{B}^{5} } \big) \times C^{3} \big( \overline{ \mathbb{B}^{5} } \big)$. A long but straight-forward calculation yields
\begin{align} \label{identityzeta5}
 \zeta \left(\omega, \widetilde{\mathbf{L} }  \mathbf{u} (\omega) \right) = - \zeta \left(\omega,  \mathbf{u} (\omega) \right) +
 \Delta ^{\mathbb{S}^{4}}_{\omega} \Big( u_{1} (\omega) + \omega ^{j} \partial _{j} u_{1} (\omega) \Big),
\end{align}
where $\Delta^{\mathbb{S}^{4}}_{\omega}  $ stands for the Laplace-Beltrami operator on the $4-$sphere, namely
\begin{align*}
\Delta ^{\mathbb{S}^{4}}_{\omega}  =\left( \delta ^{jk} - \omega ^{j} \omega ^{k} \right) \partial _{\omega ^{j}} \partial _{\omega ^{k}} -4 \omega ^{j} \partial _{\omega^{j}}.
\end{align*}
Now, Stoke's theorem yields 
\begin{align*}
\int _{\mathbb{S}^{4}} \Delta ^{\mathbb{S}^{4}}_{\omega}  \Big( u_{1} (\omega)+ \omega ^{j} \partial _{j}  u_{1}(\omega) \Big) = 0
\end{align*}
which implies the initial claim.
\end{proof}
Summarizing the results of the two previous Lemmas, we get
\begin{cor} \label{Lumer1}
For all $\mathbf{u} \in \mathcal{D} \big( \widetilde{ \mathbf{L} } \big)$, we have
\begin{align*}
\mathrm{Re} \big( \widetilde{ \mathbf{L} }  \mathbf{u}  \big{|} \mathbf{u}  \big) \leq -  \| \mathbf{u} \|^{2}.
\end{align*}
\end{cor}
Next, we prove that the range of $\lambda - \widetilde{ \mathbf{L} }$ is dense in $\mathcal{H}$ for some $\lambda > -1$ which verifies the second and last hypothesis of the Lumer-Phillips theorem. However, we will first need a technical result.
\begin{lemma} \label{technical}
For any $F \in H^{2} ( \mathbb{B}^{5} )$ and $\epsilon >0$, there exists $v \in C^{4} ( \overline{ \mathbb{B}^{5} } )$ such that the function defined by
\begin{align*}
h(\xi):=-\left( \delta ^{ij} - \xi ^{i} \xi^{j} \right) \partial _{i} \partial _{j}  v (\xi) + 7 \xi ^{j} \partial _{j} v(\xi) + \frac{35}{4}v (\xi)
\end{align*}
satisfies $h \in C^{2} ( \overline{ \mathbb{B}^{5} } )$ and $\| h-F \| _{ H^{2} ( \mathbb{B}^{5} ) } < \epsilon$.
\end{lemma}
\begin{proof}
To begin with, we pick an arbitrary $F \in H^{2} ( \mathbb{B}^{5} )$ and $\epsilon >0$. Since $C^{\infty} ( \overline{ \mathbb{B} ^{5} } )$ is dense in $H^{2} ( \mathbb{B}^{5} )$, we pick a function 
$\tilde{h} \in C^{\infty} ( \overline{ \mathbb{B} ^{5} } )$ such that $\| F-\tilde{h} \|_{H^{2} (\mathbb{B}^{5})}<\frac{\epsilon}{2}$. We consider the equation
\begin{align} \label{1}
-(\delta ^{ij} - \xi ^{i} \xi ^{j}) \partial _{i} \partial _{j} v (\xi) + 7 \xi ^{j} \partial _{j} v (\xi)+ \frac{35}{4}v(\xi) = \tilde{h}(\xi).
\end{align}
To solve \eqref{1}, we switch to spherical coordinates $\xi = \rho \omega$, where $\rho =|\xi|$ and $\omega=\frac{\xi}{|\xi|}$. Then,
\begin{align*}
\partial _{j} \rho (\xi) = \omega _{j} (\xi),\quad \partial _{j} \omega ^{k} (\xi) = \frac{ \delta_{j}^{k} - \omega_{j}(\xi) \omega ^{k}(\xi)  }{\rho(\xi)}
\end{align*}
and derivatives transform according to
\begin{align*} 
& \xi ^{j} \partial _{j} u(\xi) = \rho \partial _{\rho} u (\rho \omega), \\[5pt]
& \xi ^{i} \xi ^{j} \partial _{i} \partial _{j} u(\xi) = \rho ^{2} \partial _{\rho}^{2} u (\rho \omega), \\[5pt]  
&  \partial ^{j} \partial _{j} u(\xi) = \left( \partial _{\rho }^{2}+\frac{4}{\rho} \partial _{\rho}  +\frac{1}{\rho ^{2}} \Delta _{\omega}^{\mathbb{S}^{4}} \right) u (\rho \omega).
\end{align*}
Hence, \eqref{1} can be written equivalently as
\begin{align} \label{2}
\Bigg( -\left( 1- \rho ^{2} \right) \partial _{\rho}^{2} + \left( -\frac{4}{\rho} + 7 \rho \right) \partial _{\rho} +\frac{35}{4} - \frac{1}{\rho ^{2}}  \Delta _{\omega}^{\mathbb{S}^{4}} \Bigg) v(\rho \omega) = \tilde{h} (\rho \omega).
\end{align}
The Laplace-Beltrami operator $\Delta^{\mathbb{S}^{4}}$ is self-adjoint on $L^{2} \left( \mathbb{S}^{4} \right)$ and its spectrum coincides with the point spectrum
\begin{align*}
\sigma \big( - \Delta ^{\mathbb{S}^{4}} \big) = \sigma _{p} \big( - \Delta ^{\mathbb{S}^{4}} \big) = \big \{ l(l+3): l \in \mathbb{N}_{0} \big \}.
\end{align*}
For each $l \in \mathbb{N}_{0}$, the eigenspace to the eigenvalue $l(l+3)$ is finite dimensional and spanned by the spherical harmonics $\{ Y_{l,m} : m \in \Omega_{l}  \}$ which are obtained by restricting harmonic homogeneous polynomials in $\mathbb{R}^{5}$ to $\mathbb{S}^{4}$. Here, $\Omega_{l} \subseteq \mathbb{Z}$ stands for the set of admissible indices $m$. Since $\tilde{h} \in C^{\infty} ( \overline{ \mathbb{B} ^{5} } )$, we can expand
\begin{align*}
\tilde{h} \left( \rho \omega \right) = \sum _{l=0}^{\infty} \sum _{m \in \Omega_{l}} \tilde{h} _{l,m}(\rho) Y_{l,m} (\omega)
\end{align*}
and we define $\tilde{h} _{N} \in C^{\infty} ( \overline{ \mathbb{B} ^{5} } )$ by
\begin{align*}
\tilde{h}_{N} \left( \rho \omega \right) = \sum _{l=0}^{N} \sum _{m \in \Omega_{l}} \tilde{h} _{l,m}(\rho) Y_{l,m} (\omega),
\end{align*}
for all $N \in \mathbb{N}$. It is well known that
\begin{align*}
\big \| \tilde{h} - \tilde{h}_{N} \big \|_{H^{2} \left( \mathbb{B}^{5} \right) } \longrightarrow 0,~~ \text{~as~}N\longrightarrow \infty
\end{align*}
and therefore we can pick $N \in \mathbb{N}$ large enough so that $\big \| \tilde{h} - \tilde{h}_{N} \big \|_{H^{2} \left( \mathbb{B}^{5} \right) } <\frac{\epsilon}{2}$. Then, \eqref{2} and the linear independence of $Y_{l,m}$ yield the decoupled system of elliptic ordinary differential equations
\begin{align} \label{3}
\Bigg( -\left( 1- \rho ^{2} \right) \frac{d^{2}}{d {\rho}^{2} } + \left( -\frac{4}{\rho} + 7 \rho \right) \frac{d}{d {\rho} }  +\frac{35}{4} + \frac{l(l+3)}{ \rho^{2} } \Bigg) v_{l,m}(\rho ) = \tilde{h}_{l,m} (\rho).
\end{align}
Setting $u_{l,m}(\rho) = \rho v_{l,m}(\rho)$, \eqref{3} yields an equation for $u_{l,m}$, that is
\begin{align} \label{4}
\Bigg( -\left( 1- \rho ^{2} \right) \frac{d^{2}}{d {\rho}^{2} }  + \left( -\frac{2}{\rho} + 5 \rho \right) \frac{d}{d {\rho} }  +\frac{15}{4} + \frac{(l+1)(l+2)}{ \rho^{2} } \Bigg) u_{l,m}(\rho ) =\rho  \tilde{h}_{l,m} (\rho).
\end{align}
Note that this is a second-order linear ordinary differential equation with four regular singular points, $\rho =-1,0,1$ and $\infty$. By the reflection symmetry, these four singular points can be reduced to three and therefore, \eqref{4} can be transformed into a hypergeometric differential equation. First, consider the homogeneous version of this equation, namely we set the right hand side equal to zero. Now, we introduce a new dependent variable. The transformation $u_{l,m} (\rho) = \rho^{l+1} w_{l,m}(z)$, $z=\rho^2$ brings \eqref{4} to a hypergeometric differential equation in its canonical form
\begin{align} \label{hypergeometric}
z(1-z) w_{l,m}^{\prime \prime } (z) + \Big( c-(a+b+1)z \Big) w_{l,m}^{\prime} (z) -a b w_{l,m}(z)=0,
\end{align}
where
\begin{align*}
a=\frac{5+2l}{4},~b=a+\frac{1}{2}=\frac{7+2l}{4},~c=2a=\frac{5+2l}{2}.
\end{align*}
Then, \eqref{hypergeometric} admits two solutions 
\begin{align*}
\phi _{0,l} (z) = {}_2 F_1 \left( a,a+\frac{1}{2},2a;z\right),\quad  \phi _{1,l} (z) = {}_2 F_1 \left( a,a+\frac{1}{2},\frac{3}{2};1-z\right) , 
\end{align*}
which are analytic around $z=0$ and $z=1$ respectively, see \cite{handbook}, page 395, 15.10.2 and 15.10.4. First, notice that both $\phi _{0,l}$ and $ \phi _{1,l}$ can be expressed in closed forms as 
\begin{align} \label{psi1l2}
&  \phi _{0,l} (z) = \frac{1}{\sqrt{1-z}} \left( \frac{2}{ 1+\sqrt{1-z} } \right)^{\frac{3}{2}+l}, \\
& \phi _{1,l} (z) = \frac{1}{(3+2l)\sqrt{1-z}}\left( \left( \frac{1}{1-\sqrt{1-z}} \right)^{\frac{3}{2}+l} -  \left( \frac{1}{1+ \sqrt{1-z}} \right)^{\frac{3}{2}+l} \right),
\end{align}
see \cite{handbook}, page 387, 15.4.18 and \cite{handbook}, page 386, 15.4.9. Second, we argue that $ \phi _{0,l}$ and $ \phi _{1,l}$ are linearly independent. Indeed, we assume that there exist constants $c_{0,l},c_{1,l} \in \mathbb{C}$ such that 
\begin{align*}
c_{0,l}  \phi _{0,l} (z)  + c_{1,l}  \phi _{1,l} (z) =0.
\end{align*}
Now, $c_{1,l}=0$ since $\lim_{z \rightarrow 0^{+}} \phi _{1,l} (z) = \infty$ whereas $\lim_{z \rightarrow 0^{+}} \phi _{0,l} (z) < \infty$. Furthermore,  $c_{0,l}=0$ since $\lim_{z \rightarrow 1^{-}} \sqrt{1-z} \phi _{1,l} (z) < \infty$. For later reference, we note that the function
\begin{align*}
\tilde{\phi} _{1,l} (z) = (1-z)^{-\frac{1}{2}} \hat{\phi_{1}} (z)
\end{align*}
is also a solution to \eqref{hypergeometric}, see \cite{handbook}, page 395, 15.10.4, where
\begin{align*}
\hat{\phi_{1}}(z):= {}_2 F_1 \left(a,a-\frac{1}{2},\frac{1}{2};1-z \right)
\end{align*}
is analytic around $z=1$, see \cite{handbook}, page 384, 15.2.1. Since $\{ \phi _{0,l}, \phi _{1,l} \}$ is a fundamental system for \eqref{hypergeometric}, we get that there exist constants $\alpha_{l},\beta_{l} \in \mathbb{C}$ such that
\begin{align*}
\phi _{0,l} (z) =\alpha_{l} \phi _{1,l} (z) +  \beta_{l} \tilde{\phi}_{1,l}(z).
\end{align*}
Transforming back, we obtain two linearly independent solutions $\psi_{j,l} (\rho) = \rho ^{l+1} \phi _{j,l} (\rho^{2})$, $j \in \{0,1\}$ to the homogeneous version of equation \eqref{4} as well as $\tilde{\psi}_{1,l} (\rho) = \rho ^{l+1} 
\tilde{\phi _{1,l}} (\rho^{2})$. In particular, we get that there exist constants $\alpha_{l},\beta_{l} \in \mathbb{C}$ such that
\begin{align*}
\psi_{0,l} (\rho) = \alpha_{l} \psi_{1,l} (\rho) + \beta_{l} (1-\rho)^{- \frac{1}{2}} \hat{\psi}_{1,l} (\rho), 
\end{align*}
where $\hat{\psi}_{1,l}$ is analytic around $\rho=1$. Moreover, $\psi_{1,l}$ is analytic around $p=1$ since $\phi_{1,l}$ is analytic around $z=1$, see \cite{handbook}, page 384, 15.2.1. Next, we find the Wronskian. A straightforward calculation yields
\begin{align} \label{wronski}
W(\psi_{0,l},\psi_{1,l}  ) (\rho)  = 2 \rho^{3l+2} W(\phi_{0,l},\phi_{1,l}  ) (\rho^2) =
\frac{-2^{l+\frac{3}{2}}}{\rho^{2} \left( 1 -\rho^{2} \right)^{\frac{3}{2}}}.
\end{align} 
By the variation of constants formula, a particular solution to equation \eqref{4} is given by
\begin{align} \label{particularsol}
u_{l,m}(\rho) = - \psi _{0,l} (\rho) I_{1,l} (\rho)  - \psi _{1,l} (\rho) I_{0,l} (\rho), 
\end{align}
where
\begin{align*}
 I_{0,l} (\rho) := \int_{0}^{\rho}  \psi _{0,l} (s)  Z_{l,m}(s)   ds, \quad I_{1,l} (\rho) := \int_{\rho}^{1} \psi _{1,l} (s)  Z_{l,m}(s) ds,
\end{align*}
and
\begin{align*}
Z_{l,m}(s):=\frac{ s \tilde{h}_{l,m} (s) }{(1-s^2) W(\psi_{0,l},\psi_{1,l}  ) (s) } =(1-s)^{\frac{1}{2}} \xi_{l,m}(s),\quad \xi_{l,m}(s):= - \frac{1}{2^{l+\frac{3}{2}}} (1+s)^{\frac{1}{2}} s^3 \tilde{h}_{l,m} (s).
\end{align*}
Notice that $\xi _{l,m} \in C^{\infty}([0,1])$ since $\tilde{h}_{l,m} \in C^{\infty}([0,1])$. We claim that $u_{l,m} \in C^{\infty}\left( 0,1\right]$. To prove this, we first observe that the quantity
\begin{align*}
c_{l,m}:=\int_{0}^{1}(1-s)^{ \frac{1}{2}} \psi_{0,l} (s) \xi_{l,m}(s)= \alpha_{l} \int_{0}^{1} (1-\rho)^{ \frac{1}{2}}  \psi_{1,l} (\rho) \xi_{l,m}(s) ds+ \beta_{l}  \int_{0}^{1}\hat{\psi}_{1,l} (s) \xi_{l,m}(s) ds 
\end{align*}
is a real number since both integrands are continuous functions on the closed interval $[0,1]$. Hence, we can write
\begin{align*}
I_{0,l} (\rho) = c_{l,m} -  \alpha_{l}  \int_{\rho}^{1} (1-s)^{\frac{1}{2}} \psi_{1,l} (s)  \xi_{l,m}(s) ds - \beta_{l} \int_{\rho}^{1} \hat{\psi}_{1,l} (s)  \xi_{l,m}(s) ds.
\end{align*}
Moreover,
\begin{align*}
\psi_{1,l}(\rho) I_{0,l}(\rho) & = c_{l,m} \psi_{1,l}(\rho)  \\
&-  \alpha_{l} \psi_{1,l}(\rho)  \int_{\rho}^{1} (1-s)^{\frac{1}{2}} \psi_{1,l} (s)  \xi_{l,m}(s) ds -\beta_{l}  \psi_{1,l}(\rho)  \int_{\rho}^{1} \hat{\psi}_{1,l} (s)  \xi_{l,m}(s) ds, \\
\psi_{0,l}(\rho) I_{1,l}(\rho) & = \alpha_{l} \psi_{1,l}(\rho)  \int_{\rho}^{1} (1-s)^{\frac{1}{2}} \psi_{1,l} (s)  \xi_{l,m}(s) ds \\
&+ \beta_{l} (1-\rho)^{-\frac{1}{2}} \hat{\psi}_{1,l} (\rho) \int_{\rho}^{1} (1-s)^{\frac{1}{2}} \psi_{1,l} (s)  \xi_{l,m}(s) ds,
\end{align*}
and hence
\begin{align*}
u_{l,m}(\rho) &= - c_{l,m} \psi_{1,l}(\rho) \\
& + \beta_{l}  \psi_{1,l}(\rho)  \int_{\rho}^{1} \hat{\psi}_{1,l} (s)  \xi_{l,m}(s) ds -  \beta_{l} (1-\rho)^{-\frac{1}{2}} \hat{\psi}_{1,l} (\rho) \int_{\rho}^{1} (1-s)^{\frac{1}{2}}  \psi_{1,l} (s)  \xi_{l,m}(s) ds. 
\end{align*}
Obviously, the first and the second terms belong to $C^{\infty}\left( 0,1\right]$. Therefore, we focus on the third term and define 
\begin{align*}
U_{l,m} (\rho):=(1-\rho)^{-\frac{1}{2}} \hat{\psi}_{1,l} (\rho) \int_{\rho}^{1} (1-s)^{\frac{1}{2}}  \psi_{1,l} (s)  \xi_{l,m}(s) ds.
\end{align*}
Now, we choose an arbitrary $N\in\mathbb{N}$ and show that the limit
\begin{align} \label{limitU}
\lim_{\rho \rightarrow 1^{-}} \frac{d^{N}}{d \rho^{N}} U_{l,m} (\rho)
\end{align}
exists. Fix sufficiently small $\delta>0$, $\rho \in (1-\delta,1)$. Then, the Taylor series expansion yields
\begin{align*}
\xi_{l,m}(\rho)=\sum_{i=0}^{N} a_{i,l,m}(1-\rho)^{i} + R_{N+1}(1-\rho), 
\end{align*}
for some coefficients $a_{i,l,m}$. Here, $R_{M}(1-\rho)$ stands for a remainder term which may change from line to line and satisfies the estimates
\begin{align*}
 \left | R_{M}(1-\rho) \right | \leq K (1-\rho)^{M}, \quad \left | \partial_{\rho}^{k} R_{M}(1-\rho) \right | \leq \Lambda (1-\rho)^{M-k},  
\end{align*}
for all $k=0,1,\dots,M$ and $\rho \in (1-\delta,1)$ and for some constants $M\in \mathbb{R}$, $K,\Lambda \geq 0$. Recall that $\psi_{1,l}$ and $\hat{\psi}_{1,l}$ are analytic functions around $\rho=1$ and hence we can write
\begin{align*}
\psi_{1,l} (\rho) = \sum_{i=0}^{\infty} b_{i,l} (1-\rho)^{i}, \quad \hat{\psi}_{1,l} (\rho) = \sum_{i=0}^{\infty} \epsilon_{i,l} (1-\rho)^{i}
\end{align*}
for some coefficients $b_{l,l}$ and $\epsilon_{i,l}$. Then, we have
\begin{align*}
& (1-\rho)^{\frac{1}{2}} \psi_{1,l}(\rho) \xi_{l,m}(\rho) = \sum_{k=0}^{\infty} \gamma_{i,l,m} (1-\rho)^{k+\frac{1}{2}} + R_{N+1+\frac{1}{2}}(1-\rho), \\
& \int_{\rho}^{1}  (1-s)^{\frac{1}{2}} \psi_{1,l}(s) \xi_{l,m}(s) = \sum_{k=0}^{\infty} \frac{2 \gamma_{i,l,m}}{2k+3} (1-\rho)^{k+\frac{3}{2}} + R_{N+2+\frac{1}{2}}(1-\rho), \\
&(1-\rho)^{-\frac{1}{2}} \int_{\rho}^{1}  (1-s)^{\frac{1}{2}} \psi_{1,l}(s) \xi_{l,m}(s) = \sum_{k=0}^{\infty} \frac{2 \gamma_{i,l,m}}{2k+3} (1-\rho)^{k+1} + R_{N+2}(1-\rho), \\
&  \hat{\psi}_{1,l} (\rho) (1-\rho)^{-\frac{1}{2}} \int_{\rho}^{1}  (1-s)^{\frac{1}{2}} \psi_{1,l}(s) \xi_{l,m}(s) = \sum_{k=0}^{\infty} \zeta_{k,l,m} (1-\rho)^{k+1} + R_{N+2}(1-\rho), 
\end{align*}
for some coefficients $\gamma_{i,l,m}$ and $\zeta_{k,l,m}$. Therefore,
\begin{align*}
\frac{d^{N}}{d \rho^{N}} U_{l,m} (\rho)  = \frac{d^{N}}{d \rho^{N}} \left( \sum_{k=0}^{\infty} \zeta_{k,l,m} (1-\rho)^{k+1} +R_{N+2} (1-\rho) \right) =  \sum_{i=0}^{\infty} \eta_{i,l,m} (1-\rho)^{i} +R_{2} (1-\rho),
\end{align*}
for some coefficients $\eta_{i,l,m}$. Consequently, the limit \eqref{limitU} exists and we get that $u_{l,m} \in C^{\infty}\left( 0,1\right]$. Finally, $u \in H^{2} (\mathbb{B}^{5}) \cap C^{\infty} ( \overline{ \mathbb{B}^{5} } \setminus \{0\} )$ and translating back we get $v \in H^{2} (\mathbb{B}^{5}) \cap C^{\infty} ( \overline{ \mathbb{B}^{5} } \setminus \{0\} )$. By elliptic regularity, we infer
$v \in C^{\infty} (\mathbb{B}^{5}) \cap C^{\infty} ( \overline{ \mathbb{B}^{5} } \setminus \{0\} )$ which implies $v \in C^{\infty} ( \overline{ \mathbb{B}^{5} } )$ as desired.
\end{proof}
\begin{lemma} \label{Lumer2}
The range of $\frac{3}{2}-\widetilde{ \mathbf{L} }$ is dense in $\mathcal{H}$.
\end{lemma}
\begin{proof}
Since $\big( C^{\infty} ( \overline{ \mathbb{B}^{5}  }) \big)^{2}$ is dense in $\mathcal{H}$, it suffices to show that
\begin{align*}
\forall ~\mathbf{f} \in \big( C^{\infty} ( \overline{ \mathbb{B}^{5}  }) \big)^{2} \text{~and~} \forall \epsilon >0,~\exists \mathbf{g} \in \text{rg}\left( \frac{3}{2} -  \widetilde{ \mathbf{L} } \right): 
\left \| \mathbf{f} -\mathbf{g}  \right\| <\epsilon.
\end{align*}
First note that, for any $\mathbf{u} \in \mathcal{D} \big(  \widetilde{ \mathbf{L} }  \big)$, the equation $\left( \frac{3}{2} -  \widetilde{ \mathbf{L} } \right) \mathbf{u} = \mathbf{g}$ reads
\begin{align*}
    \begin{cases}
      u_{2}  (\xi) = \frac{5}{2} u_{1} (\xi)+ \xi ^{j} \partial _{j} u_{1}  (\xi)- g_{1} (\xi), &  \\
      -\partial ^{j} \partial _{j} u_{1}  (\xi)+ \xi^{i} \partial _{i} u_{2}  (\xi)+\frac{7}{2} u_{2}  (\xi)= g_{2} (\xi)
    \end{cases}
\end{align*}
Inserting $u_{2}$ into the second equation, we obtain
\begin{align*}
-(\delta ^{ij} - \xi ^{i} \xi ^{j}) \partial _{i} \partial_{j} u_{1} (\xi)+ 7 \xi^{i} \partial _{i} u_{1} (\xi)+ \frac{35}{4} u_{1} (\xi)= G(\xi),
\end{align*}
where
\begin{align*}
G  (\xi) = g_{2} (\xi)+ \frac{7}{2} g_{1} (\xi)+  \xi ^{j} \partial _{j} g_{1}(\xi).
\end{align*}
Now, pick an arbitrary $\mathbf{f} \in \big( C^{\infty} ( \overline{ \mathbb{B}^{5}  }) \big)^{2}$, $\epsilon >0$ and apply Lemma \ref{technical} to the function
\begin{align*}
F  (\xi) = f_{2} (\xi)+ \frac{7}{2} f_{1} (\xi)+  \xi ^{j} \partial _{j} f_{1}(\xi).
\end{align*}
We infer the existence of a function $v \in C^{4} ( \overline{ \mathbb{B}^{5} } )$ such that 
\begin{align*}
h(\xi):=-\left( \delta ^{ij} - \xi ^{i} \xi^{j} \right)\partial _{i} \partial_{j}v (\xi)+ 7 \xi ^{j} \partial _{j} v (\xi)+ \frac{35}{4}v(\xi)
\end{align*}
satisfies $h \in C^{2} ( \overline{ \mathbb{B}^{5} } )$ and $\| h-F \| _{ H^{2} ( \mathbb{B}^{5} ) } < \epsilon$. Now, define
\begin{align*}
&
  \begin{cases}
u_{1}(\xi):= v(\xi), & \\
u_{2} (\xi) := \frac{5}{2} u_{1}(\xi)+ \xi ^{j} \partial _{j} u_{1}(\xi) - f_{1}(\xi), 
   \end{cases} \\
&
\begin{cases}
g_{1} (\xi) := f_{1} (\xi), & \\
g_{2} (\xi) := h(\xi)-F(\xi) +f_{2} (\xi).
   \end{cases}
\end{align*}
Then, by construction, we have $\mathbf{u} \in \mathcal{D} \big( \widetilde{\mathbf{L}} \big)$, $\left( \frac{3}{2} -  \widetilde{ \mathbf{L} } \right) \mathbf{u} = \mathbf{g}$ and $\|\mathbf{f} - \mathbf{g} \| < \epsilon$.
\end{proof}
\begin{proof}[Proof of Proposition $\ref{Lumer}$]
It follows immediately from Corollary \ref{Lumer1} and Lemma \ref{Lumer2}. 
\end{proof}


\section{Modulation ansatz} 

To account for the Lorentz symmetry we use a modulation ansatz. To be precise, we allow for the unknown parameter $\alpha$ to depend on $\tau$, set $\alpha(0) = 0$ initially and assume (and later verify) that 
$\alpha _{\infty} := \lim_{\tau \rightarrow \infty} \alpha(\tau)$ exists. Then, we define
\begin{align} \label{modulationansatz}
\mathbf{\Phi}(\tau):=\mathbf{\Psi}(\tau)- \mathbf{\Psi}_{\alpha(\tau)}
\end{align}
where $\mathbf{\Psi}_{\alpha}$ are the Lorentz transformations defined in \eqref{staticbold} of the static blowup solution solution $\mathbf{\Psi}_{0}$. This ansatz leads to an equivalent description as an evolution equation for the perturbation term $\mathbf{\Phi}$, that is
\begin{align} \label{modulation}
\partial _{\tau} \mathbf{\Phi}(\tau) - \big( \mathbf{L} + \mathbf{L}^{\prime}_{\alpha _{\infty}} \big) \mathbf{\Phi}(\tau) = 
\hat{\mathbf{L}}_{\alpha(\tau)} \mathbf{\Phi}(\tau) + \mathbf{N}_{\alpha(\tau)} (\mathbf{\Phi}(\tau)) - \partial _{\tau} \mathbf{\Psi}_{\alpha(\tau)},
\end{align}
where  
\begin{align}
\hat{\mathbf{L}}_{\alpha(\tau)}:= \mathbf{L}^{\prime}_{\alpha(\tau)} - \mathbf{L}^{\prime}_{\alpha _{\infty}} 
\end{align}
and $\mathbf{L}^{\prime}_{\alpha (\tau)}$ denotes the linearized part of the nonlinearity $\mathbf{N}$, i.e 
\begin{align} \label{Lprime}
 \mathbf{L}^{\prime}_{\alpha (\tau)} (\mathbf{u}(\xi)):=
 \begin{pmatrix}
0 \\
V_{\alpha (\tau)} (\xi) u_{1}(\xi)  
\end{pmatrix},\quad
V_{\alpha (\tau)} (\xi):=\frac{6}{\left( A_{0}(\alpha(\tau))-A_{j}(\alpha(\tau))\xi^{j} \right)^2}
\end{align}
and $\mathbf{N}_{\alpha(\tau)}$ stands for the remaining full nonlinearity
\begin{align} \label{nonlinear}
 \mathbf{N}_{\alpha(\tau)} (\mathbf{u}) :=  \mathbf{N} (\mathbf{u} + \mathbf{\Psi}_{\alpha(\tau)})  - \mathbf{N} (\mathbf{\Psi}_{\alpha(\tau)})-\mathbf{L}^{\prime}_{\alpha(\tau)} \mathbf{u}.  
\end{align}
The advantage of this formulation is that the left hand side of \eqref{modulation} consists (besides $\partial_{\tau} \mathbf{\Phi}$) only of linear and $\tau-$independent operations on $\mathbf{\Phi}$, whereas the right hand side is expected to be small for large $\tau$. Therefore, the right hand side of the equation \eqref{modulation} will be treated perturbatively. Note that, for sufficiently small $\alpha$, we have $A_{0}(\alpha) = \mathcal{O}(1)$ whereas $A_{j}(\alpha) = \mathcal{O}(\alpha)$ which shows that 
\begin{align} \label{boundV}
\sup _{j \in \{ 0,1,2 \} }\left \| \partial ^{j} V_{\alpha} \right \|_{L^{\infty}(\mathbb{B}^{5})} \lesssim 1
\end{align}
provided that $\alpha$ is sufficiently small. As we will now prove, this fact, together with the compactness of the Sobolev embedding yields the compactness of the operator $\mathbf{L} ^{\prime}_{\alpha}$ for sufficiently small $\alpha$.
\begin{lemma} \label{compact}
Let $\alpha \in \mathbb{R}^{5}$ be sufficiently small. Then, the operator $\mathbf{L} ^{\prime}_{\alpha}$ defined in \eqref{Lprime} is compact. In particular, the operator
\begin{align} \label{Lalpha}
 \mathbf{L}_{\alpha} := \mathbf{L} + \mathbf{L} ^{\prime}_{\alpha}
 \end{align}
generates a strongly continuous one parameter semigroup of bounded operators $\mathbf{S}_{\alpha}:[0,\infty) \longrightarrow \mathcal{B}(\mathcal{H})$.
\end{lemma}
\begin{proof}
To begin with, we fix $\alpha$ sufficiently small. First, we prove that $\mathbf{L} ^{\prime}_{\alpha}$ is compact. We pick a bounded sequence $\{ \mathbf{u}_{n} \}_{n=1}^{\infty} \subseteq \mathcal{H}$. The compactness of the Sobolev embedding
$H^{3} \big( \mathbb{B}^{5} \big) \hookrightarrow  H^{2} \big( \mathbb{B}^{5} \big)$ yields the existence of a subsequence $\{ \mathbf{u}_{k_{n}} \}_{n=1}^{\infty}$ in $H^{3} \big( \mathbb{B}^{5} \big)$ which is Cauchy in 
$H^{2} \big( \mathbb{B}^{5} \big)$. Now, \eqref{boundV} together with H{\"o}lder's inequality imply
\begin{align*}
\left \| \mathbf{L} ^{\prime}_{\alpha} \mathbf{u}_{k_{n}} -   \mathbf{L} ^{\prime}_{\alpha}  \mathbf{u}_{k_{m}} \right \| =
\left \| V_{\alpha} \left( u_{1,k_{n}} - u_{1,k_{m}}  \right) \right \| _{H^{2}(\mathbb{B}^{5}  )} \lesssim \left \|   u_{1,k_{n}} - u_{1,k_{m}} \right \| _{H^{2}(\mathbb{B}^{5}  )}
\end{align*}
for sufficiently large $n,m \in \mathbb{N}$. This proves that $\{ \mathbf{L} ^{\prime} \mathbf{u}_{k_{n}} \}_{n=1}^{\infty}$ is Cauchy in $\mathcal{H}$ and the claim follows. It remains to apply the Bounded Perturbation Theorem 
(see Theorem 1.3, page 158, \cite{EngNag00}) to show that $\mathbf{L}_{\alpha} := \mathbf{L} + \mathbf{L} ^{\prime}_{\alpha}$ is the generator of a strongly continuous semigroup $\left(\mathbf{S}_{\alpha}(\tau) \right)_{\tau>0}$.
\end{proof}
\subsection{Solution to the full linear problem}
Due to Lemma \ref{compact}, we can write the solution to the linear part of \eqref{modulation},
\begin{align*}
  \left \{
  \begin{aligned}
    & \partial _{\tau} \mathbf{\Phi}(\tau) = \left( \mathbf{L} + \mathbf{L}^{\prime}_{\alpha _{\infty}} \right) \mathbf{\Phi}(\tau), && \ \\
    & \mathbf{\Phi}(0)= \mathbf{u} \in \mathcal{H}, && 
  \end{aligned} \right.
\end{align*} 
as 
\begin{align} \label{Duhamel}
\mathbf{\Phi}(\tau)= \mathbf{S}_{\alpha_{\infty}}(\tau) \mathbf{u},
\end{align}
provided that $\alpha_{\infty}$ is sufficiently small which is verified later, see \eqref{ainfty}. In addition to the existence of the semigroup $ \mathbf{S}_{\alpha_{\infty}}$, we need growth estimates in time. By Proposition \ref{Lumer} and Lemma \ref{compact}, the Bounded Perturbation Theorem (see Theorem 1.3, page 158, \cite{EngNag00}) yields 
\begin{align*}
\| \mathbf{S}_{\alpha_{\infty}} (\tau) \|\leq M e^{  \left( -1+M \|  \mathbf{L}_{\alpha _{\infty} }^{\prime} \|   \right) \tau  },
\end{align*}
as long as $\alpha_{\infty}$ is sufficiently small. However, such a growth estimate would not suffice and hence we turn our attention to the spectrum of the generator $\mathbf{L}_{\alpha}$.


\section{Spectral Analysis} 

In this section, we intend to establish a useful growth estimate for $\mathbf{S}_{\alpha}$ for sufficiently small $\alpha$ and therefore we turn our attention to the spectrum of the generator $\mathbf{L}_{\alpha}$. We start our analysis with the case $\alpha=0$ where the Lorentz boost $\Lambda(0)$ is the identity. Therefore, the potential $V_{0}$ in the definition of $\mathbf{L}^{\prime}_{0}$, see \eqref{Lprime}, is constant in $\xi$. Consequently, the spectral equation can be solved explicitly and solutions belong to the hypergeometric class, as it turns out. The advantage here is that we can use the connection formula which is well known for this class. Then, we proceed to the case where $\alpha \neq 0$ but we are only interested in small $\alpha$ which allows for a perturbative approach, as already explained above.

\subsection{The spectrum of the free operator.}
We can use the decay estimate for the free semigroup $\left(\mathbf{S}(\tau) \right)_{\tau>0}$ from Proposition \ref{Lumer} to locate the spectrum of the closure $\mathbf{L}$ of the free operator $\widetilde{ \mathbf{L} }$. 
As a matter of fact, by \cite{EngNag00}, p. 55, Theorem 1.10, we immediately infer
\begin{align} \label{sigmaL}
\sigma \left(  \mathbf{L} \right) \subseteq \left \{ \lambda \in \mathbb{C}: \mathrm{Re}\lambda \leq -1  \right \}.
\end{align}
\subsection{The spectrum of the full linear operator for $\alpha=0$.} \label{spectrumfulllinear}
To begin with, we use the fact that $\mathbf L ^{\prime}_{\alpha}$ is compact for sufficiently small $\alpha$ to see that it suffices to consider the point spectrum of $\mathbf{L}_{\alpha}$.
\begin{lemma} \label{sigmaLfull}
Let $\alpha \in \mathbb{R}^{5}$ be sufficiently small. We have
 \begin{align*}
\sigma (\mathbf{L}_{\alpha}) \setminus \sigma(\mathbf{L}) \subseteq \sigma _{p} (\mathbf{L}_{\alpha}).
\end{align*}
\end{lemma}
\begin{proof}
Fix $\alpha \in \mathbb{R}^{5}$ sufficiently small and pick $\lambda \in \sigma (\mathbf{L}_{\alpha}) \setminus \sigma(\mathbf{L})$. From the identity $\lambda-\mathbf {L}_{\alpha}=[1-\mathbf L^{\prime}_{\alpha} \mathbf R_{\mathbf L}(\lambda)](\lambda-\mathbf L)$ we see that $1\in \sigma(\mathbf L^{\prime}_{\alpha} \mathbf R_{\mathbf L}(\lambda))$. Since $\mathbf L^{\prime}_{\alpha}\mathbf R_{\mathbf L}(\lambda)$ is compact, it follows that $1\in \sigma_p(\mathbf L^{\prime}_{\alpha}\mathbf R_{\mathbf L}(\lambda))$ and thus, there exists a nontrivial $\mathbf f \in \mathcal H$ such that $[1-\mathbf L^{\prime}_{\alpha}\mathbf R_{\mathbf L}(\lambda)]\mathbf f=0$.
Consequently, $\mathbf u:=\mathbf R_{\mathbf L}(\lambda)\mathbf f\not= 0$
satisfies $(\lambda-\mathbf {L}_{\alpha})\mathbf u=0$ and thus, $\lambda\in \sigma_p(\mathbf {L}_{\alpha})$.
\end{proof}
Now, we prove the following result.
\begin{prop} \label{pointspectrum}
We have
\begin{align*}
\sigma \left(  \mathbf{L}_{0} \right) \subseteq \left \{ \lambda \in \mathbb{C}: \mathrm{Re}\lambda \leq -1  \right \} \cup \{ 0,1\}. 
\end{align*}
\end{prop}
\begin{proof}
To prove this result, we argue by contradiction. To begin with, fix a spectral point $\lambda \in \sigma  \left(  \mathbf{L}_{0} \right)$ with Re$\lambda > -1$ and $\lambda \ne 0,1$. Then, \eqref{sigmaL} implies that 
$\lambda \notin \sigma  \left(  \mathbf{L} \right)$ and Lemma \ref{sigmaLfull} yields $\lambda \in \sigma _{p} \left(  \mathbf{L}_{0} \right)$. Consequently, there exists a non-trivial element $\mathbf{v} \in  \mathcal{D} \big(\mathbf{L}_{0} \big) \subseteq \mathcal{H}$ such that $\big(\lambda - \mathbf{L}_{0} \big) \mathbf{v} = 0$. Then, for $\mathbf{v}=(v_{1},v_{2})$, we get
\begin{align*}
    \begin{cases}
      v_{2} (\xi)= (\lambda +1) v_{1} (\xi)+ \xi ^{j} \partial _{j} v_{1} (\xi), &  \\
      -\partial ^{j} \partial _{j} v_{1} (\xi)+ \xi^{i} \partial _{i} v_{2} (\xi)+(\lambda +2) v_{2}(\xi) - 6 v_{1} (\xi)= 0.
    \end{cases}
\end{align*}
Inserting $v_{2}$ into the second equation, we obtain
\begin{align*}
- \left( \delta ^{ij} - \xi ^{i} \xi ^{j} \right) \partial _{i} \partial _{j} v_{1} (\xi)+ 2 (\lambda +2) \xi ^{i} \partial _{i} v_{1} (\xi) +  \Big( (\lambda +1) (\lambda +2) - 6 \Big) v_{1} (\xi) = 0.
\end{align*}
To solve this equation, we switch to spherical coordinates $\xi = \rho \omega$, where $\rho =|\xi|$ and $\omega=\frac{\xi}{|\xi|}$. Then,
\begin{align*}
\partial _{j} \rho (\xi) = \omega _{j} (\xi),\quad \partial _{j} \omega ^{k} (\xi) = \frac{ \delta_{j}^{k} - \omega_{j}(\xi) \omega ^{k}(\xi)  }{\rho(\xi)}
\end{align*}
and derivatives transform according to
\begin{align*} 
& \xi ^{j} \partial _{j} v_{1}(\xi) = \rho \partial _{\rho} v_{1} (\rho \omega), \\[5pt]
& \xi ^{i} \xi ^{j} \partial _{i} \partial _{j} v_{1}(\xi) = \rho ^{2} \partial _{\rho}^{2} v_{1} (\rho \omega), \\[5pt]  
&  \partial ^{j} \partial _{j} v_{1}(\xi) = \left( \partial _{\rho }^{2}+\frac{4}{\rho} \partial _{\rho}  +\frac{1}{\rho ^{2}} \Delta _{\omega}^{\mathbb{S}^{4}} \right) v_{1}(\rho \omega).
\end{align*}
Hence, the spectral equation above can be written equivalently as 
\begin{align*}
\Bigg[ -\left( 1- \rho^{2} \right) \partial _{\rho}^{2}  - \Bigg( \frac{4}{\rho} -2(\lambda +2) \rho  \Bigg) \partial _{\rho} 
+ \Big( (\lambda +1)(\lambda +2)  - 6 \Big)   -\frac{1}{\rho^{2}} \Delta _{\omega}^{\mathbb{S}^{4}} \Bigg] v_{1}(\rho \omega)= 0.
\end{align*}
By elliptic regularity, we infer $v_{1} \in C^{\infty}(\mathbb{B}^{5}) \cap H^{3}(\mathbb{B}^{5})$. Therefore, we may expand 
\begin{align*}
v_{1}(\rho \omega) = \sum _{l=0}^{\infty} \sum _{m \in \Omega_{l}} v_{1,l,m} (\rho) Y_{l,m} (\omega). 
\end{align*}
Inserting this ansatz into the spectral equation above, we obtain the decoupled system of ordinary differential equations
\begin{align} \label{55}
\Bigg[ -\left( 1- \rho^{2} \right) \frac{d^{2}}{d \rho ^{2}}   - \Bigg( \frac{4}{\rho} -2(\lambda +2) \rho  \Bigg) \frac{d}{d \rho} 
+ \Big( (\lambda +1)(\lambda +2)  - 6 + \frac{l(l+3)}{\rho^{2}}  \Big)  \Bigg] v_{1,l,m}(\rho )= 0.
\end{align}
From now on we suppress the subscripts. First note that this is a second order ordinary differential equation with four regular singular points: $-1,0,1$ and $\infty$. Again, by the reflection symmetry, these four singular points can be reduced to three and therefore, \eqref{55} can be transformed into a hypergeometric differential equation. To do so, we introduce the change of variables $v(\rho)=\rho ^{l} w(z)$ with $z=\rho^{2}$ and we get
\begin{align} \label{hyper}
z(1-z)w^{\prime \prime} (z)+ \Big( c-(a+b+1)z \Big) w^{\prime} (z) - ab w(z) = 0
\end{align}
where
\begin{align*}
a:=\frac{1}{2}(\lambda +l -1),~~~ b:=\frac{1}{2}(\lambda + l + 4),~~~ c:=\frac{5}{2}+l.
\end{align*}
The functions
\begin{align*}
& w_{0}(z):= {}_2F_1\left( a,b; c; z \right), \\
& \widetilde{w}_{0}(z):=z^{1-c} {}_2F_1\left( a-c+1,b-c+1; 2-c; z \right), \\
& w_{1}(z):= {}_2F_1\left( a,b; a+b+1-c; 1-z \right), \\
& \widetilde{w}_{1}(z):=(1-z)^{c-a-b} {}_2F_1\left( c-a,c-b; c-a-b+1; 1-z \right),
\end{align*}
are all solutions to \eqref{hyper}, see \cite{handbook}. First, note that $ \widetilde{w}_{1}$ is not admissible since the initial condition Re$\lambda>-1$ yields
\begin{align*}
\mathrm{Re}(c-a-b)=1-\mathrm{Re}\lambda<2
\end{align*}
and thus $\widetilde{w}_{1} \notin H^{3} (\frac{1}{2},1)$ whereas $\mathcal{D} \big(\mathbf{L}_{0} \big) \subseteq \mathcal{H}$. Similarly, $\widetilde{w}_{0}$ is not admissible either since it would lead to a solution $v_{l,m}$ that behaves like $\rho^{-\frac{3}{2}}$ as $\rho \rightarrow 0+$ which contradicts $v_{l,m} \in C^{\infty}[0,1)$. Hence, we are left with $ w_{0}$ and $ w_{1}$ and since both $\{ w_{0}, \widetilde{w}_{0} \}$ and $\{ w_{1}, \widetilde{w}_{1} \}$ are fundamental systems for the hypergeometric equation \eqref{hyper} we infer that $w_{0}$ and $w_{1}$ must be linearly dependent. In view of the connection formula \cite{handbook},
\begin{align*}
w_{0}(z)=\frac{\Gamma(c)\Gamma(c-a-b)}{\Gamma(c-a)\Gamma(c-b)} w_{1}(z) + \frac{\Gamma(c)\Gamma(a+b-c)}{\Gamma(a)\Gamma(b)} \widetilde{w}_{1}(z), 
\end{align*}
the linear dependence of $w_{0}$ and $w_{1}$ implies that
\begin{align*}
 \frac{\Gamma(c)\Gamma(a+b-c)}{\Gamma(a)\Gamma(b)} =0.
\end{align*}
However, the gamma function has no zeros and therefore we see that $a$ or $b$ must be a pole of $\Gamma$. The latter means $-a \in \mathbb{N}_{0}$ or $-b \in \mathbb{N}_{0}$. The first condition, $-a=n$ for some $n \in \mathbb{N}_{0}$, yields $2n<2-l$ which is possible only if $n=0$ and $l \in \{0,1\}$ which in turn implies $\lambda \in \{0,1\}$ and refutes the initial assumption. The second condition, $-b=m$ for some $m \in \mathbb{N}_{0}$, yields 
$\lambda = -2m-4-l$ and the initial hypothesis on $\lambda$ yields $-1<$Re$\lambda=-2m-4-l$ which is a contradiction, namely $3<-(2m+l)$.
\end{proof}
\begin{remark}
The spectral equations for $\lambda =0$ and $\lambda =1$ respectively read
\begin{align*}
& - \left( \delta ^{ij} - \xi ^{i} \xi ^{j} \right) \partial _{i} \partial _{j} v_{1} (\xi)+ 4 \xi ^{i} \partial _{i} v_{1} (\xi) -4 v_{1} (\xi) = 0, \\
& - \left( \delta ^{ij} - \xi ^{i} \xi ^{j} \right) \partial _{i} \partial _{j} v_{1} (\xi)+ 6 \xi ^{i} \partial _{i} v_{1} (\xi)  = 0.
\end{align*}
It is straightforward to check that, for all fixed $j\in \{1,2,3,4,5\}$, $v_{1}(\xi)=\xi^{j}$ solves the first equation whereas the constant function $v_{1}(\xi)=1$ solves the second equation. Consequently, the eigenspaces for the isolated eigenvalues $\lambda =0$ and $\lambda=1$ of the operator $\mathbf{L}_{0}$ are spanned respectively by
\begin{align*}
& \mathbf{h}_{0,j} (\xi)= \partial _{\alpha^{j}} \mathbf{\Psi}_{\alpha} (\xi) |_{\alpha=0} = \sqrt{2}
  \begin{pmatrix}
    \xi ^{j} \\
    2\xi^{j}
  \end{pmatrix},~~~j\in \{1,2,3,4,5\} \\
 &   \mathbf{g}_{0} (\xi)= 
  \begin{pmatrix}
    1 \\
    2
  \end{pmatrix},
\end{align*}
and hence $\{0,1\} \subseteq \sigma _{p} (\mathbf{L}_{0})$. Finally, notice that the above derivation shows that the geometric eigenspaces of $0$ and $1$ are $5-$dimensional and $1-$dimensional, respectively. 
\end{remark}
Note that since the operator $\mathbf{L}_{0}$ is highly non self-adjoint, it is not straightforward to see that the algebraic multiplicity of the isolated eigenvalues $\lambda=0$ and $\lambda =1$ are equal to $5$ and $1$, respectively. Now, we focus on proving this result rigorously. To be precise, we follow \cite{DonSch16a} and use the fact that the eigenvalues $\lambda =0$ and $\lambda=1$ are isolated to introduce two (non-orthogonal) Riesz projections $\mathbf{Q}_{0}$ and $\mathbf{P}_{0}$, namely
\begin{align*}
& \mathbf{Q}_{0}:=\frac{1}{2\pi i} \int_{\gamma_{0}} \mathbf{R}_{\mathbf{L}_{0}}(\zeta) d \zeta, \\
& \mathbf{P}_{0}:=\frac{1}{2\pi i} \int_{\gamma_{1}} \mathbf{R}_{\mathbf{L}_{0}}(\zeta) d \zeta,
\end{align*}
where $\gamma _{0},\gamma_{1}:[0,1] \rightarrow \mathbb{C}$ stand for the circles centered at $\lambda=0$ and $\lambda=1$,
\begin{align*}
\gamma _{0}(s):=\frac{1}{2} e^{2\pi i s},\quad \gamma _{1}(s):=1+\frac{1}{2} e^{2\pi i s},
\end{align*}
respectively. These projections decompose the Hilbert space of initial data $\mathcal{H}$ into rg$(1-\mathbf{Q}_{0})$ (stable space for $\lambda=0$) and rg$\mathbf{Q}_{0}$ (unstable space for $\lambda=0$),
\begin{align*}
\mathcal{H} =\text{rg}(1-\mathbf{Q}_{0})  \oplus \text{rg} (\mathbf{Q}_{0}).
\end{align*}
Similarly, for $\mathbf{P}_{0}$. We show that
\begin{align*}
& m_{a} (\lambda =0):=\mathrm{rank}\, \mathbf{Q}_{0} =\dim\mathrm{rg}\, \mathbf{Q}_{0}, \\
& m_{a} (\lambda =1):=\mathrm{rank}\, \mathbf{P}_{0} =\dim\mathrm{rg}\, \mathbf{P}_{0},
\end{align*}
are equal to $5$ and $1$ respectively.
\begin{lemma} \label{rank}
We have $\dim\mathrm{rg}\, \mathbf{Q}_{0}=5$ and $\dim\mathrm{rg}\, \mathbf{P}_{0}=1$.
\end{lemma}
\begin{proof}
Since the process is the same for both quantities, we illustrate it on $\mathbf{Q}_{0}$ only. We refer the reader to \cite{Kat95} for the following standard results. The projection $\mathbf {Q}_{0}$ commutes with the operator $\mathbf{L}_{0}$ and thus with the semigroup $\mathbf S_{0}(\tau)$. Moreover, $\mathbf{Q}_{0}$ decomposes the Hilbert space as $\mathcal{H} =\mathcal M \oplus \mathcal N$, where $\mathcal M:=\rg \mathbf {Q}_{0}$ and $\mathcal N:=\ker \mathbf{Q}_{0}=\rg(1-\mathbf {Q}_{0})$. Most importantly, the operator $\mathbf{L}_{0}$ is decomposed accordingly into the parts $\mathbf{L}_{0,\mathcal M}$ and $\mathbf{L}_{0,\mathcal N}$ on $\mathcal M$ and $\mathcal N$, respectively. The spectra of these operators are given by
\begin{align} \label{spectrum}
\sigma \left( \mathbf L_{0,\mathcal N} \right ) = \sigma (\mathbf{L}_{0}) \setminus \{0\},\qquad \sigma \left( \mathbf{L}_{0,\mathcal M} \right ) = \{0\}.
\end{align}
Finally, $\rg \mathbf{Q}_{0} \subseteq \mathcal{D}(\mathbf{L})$. To proceed, we break down the proof into the following steps: \\ \\
Step 1: We prove that $\rank\mathbf{Q}_{0}:=\dim\rg\mathbf{Q}_{0}<+\infty$. We argue by contradiction and assume that $\rank\mathbf{Q}_{0}=+\infty$. Using \cite{Kat95}, p.~239, Theorem 5.28, the fact that $\mathbf{L}_{0}^{\prime}$ is compact (see Lemma $\ref{compact}$), and the fact that the essential spectrum is stable under compact perturbations (\cite{Kat95}, p.~244, Theorem 5.35), we obtain
\begin{align*}
\mathrm{rank}\,\mathbf{Q}_{0} = +\infty  \Longrightarrow 1 \in \sigma _{e} (\mathbf{L}_{0}) = \sigma _{e} (\mathbf{L}_{0} -\mathbf{L}_{0}^{\prime})=\sigma _{e}(\mathbf{L}) \subseteq \sigma (\mathbf{L}),
\end{align*}
which clearly contradicts \eqref{sigmaL}. 
\\ \\
Step 2: We prove that $\langle \mathbf{h}_{0,1}, \mathbf{h}_{0,2}, \mathbf{h}_{0,3}, \mathbf{h}_{0,4}, \mathbf{h}_{0,5} \rangle=\mathrm{rg}\,\mathbf{Q}_{0}$. It suffices to show $\mathrm{rg}\,\mathbf{Q}_{0} \subseteq \langle \mathbf{h}_{0,1}, \mathbf{h}_{0,2}, \mathbf{h}_{0,3}, \mathbf{h}_{0,4}, \mathbf{h}_{0,5} \rangle$ since the reverse inclusion follows from the abstract theory. From Step 1, the operator $\mathbf{L}_{0,\mathcal M}$ acts on the finite-dimensional Hilbert space $\mathcal M=\rg \mathbf {Q}_{0}$ and, from $\eqref{spectrum}$, $\lambda =0$ is its only spectral point. Hence, $\mathbf{L}_{0,\mathcal M}$ is nilpotent, i.e., there exists a minimal $k\in \mathbb N$ such that
\begin{align*}
\big(  \mathbf{L}_{0,\mathcal M}  \big)^{k}  \mathbf{u}= 0
\end{align*}
for all $\mathbf{u} \in \mathrm{rg}\, \mathbf{Q}_{0}$. Now, the claim follows immediately if $k=1$. Indeed, if $k=1$, then $\mathrm{rg}\ \mathbf{Q}_{0}=\ker \mathbf{L}_{0}=\langle \mathbf{h}_{0,1}, \mathbf{h}_{0,2}, \mathbf{h}_{0,3}, \mathbf{h}_{0,4}, \mathbf{h}_{0,5} \rangle$ which shows that $\dim\mathrm{rg}\, \mathbf{Q}_{0}=5$. We proceed by contradiction and assume that $k\geq 2$. Then, there exists a nontrivial function $\mathbf{u} \in \rg \mathbf{Q}_{0} \subseteq \mathcal{D}( \mathbf{L})$ such that 
$(\mathbf L_{0,\mathcal M})\mathbf u$ is nonzero and belongs to $\ker(\mathbf L_{0,\mathcal M})\subseteq \ker(\mathbf {L}_{0})$.
This means that $\mathbf{ u } \in \rg\mathbf{Q}_{0} \subseteq \mathcal{D} (\mathbf{L})$ satisfies $\mathbf{ L }_{0} \mathbf{ u } = \mathbf{ f}$, for some $\mathbf{f} \in \ker {\mathbf{L}_{0}}$. A straightforward computation shows that the first component of $\mathbf u$ solves the second order differential equation
\begin{align*}
- \left( \delta ^{ij} - \xi ^{i} \xi ^{j} \right) \partial _{i} \partial _{j} u_{1} (\xi)+ 4 \xi ^{i} \partial _{i} u_{1} (\xi) -4 u_{1} (\xi) = - f(\xi),
\end{align*}
where 
\begin{align*}
f(\xi):= \xi ^{j} \partial _{j} f_{1} (\xi) +2 f_{1}(\xi) + f_{2}(\xi)
\end{align*}
and $\mathbf{f}=(f_{1},f_{2})$. We switch to hyper-spherical coordinates $\xi=\rho \omega$ where $\rho=|\xi|$ and $\omega=\frac{\xi}{|\xi|}$. Then,
\begin{align*}
\Bigg [  -(1-\rho^2) \frac{d^2}{d \rho ^2} - \left( \frac{4}{\rho} -4 \rho \right) \frac{d}{d \rho} -4  -\frac{1}{\rho ^2} \Delta ^{\mathbb{S}^{4}}_{\omega} \Bigg ] u_{1} (\rho \omega) = f (\rho \omega).
\end{align*}
Since 
\begin{align*}
\mathbf{f} \in \ker(\mathbf{L}_{0})=\langle \mathbf{h}_{0,1}, \mathbf{h}_{0,2}, \mathbf{h}_{0,3}, \mathbf{h}_{0,4}, \mathbf{h}_{0,5} \rangle = 
\langle 
\sqrt{2} 
\begin{pmatrix}
   \xi^{1} \\
    2 \xi^{1}
\end{pmatrix}, 
\cdots,
\sqrt{2} 
\begin{pmatrix}
   \xi^{5} \\
    2 \xi^{5}
\end{pmatrix}
 \rangle,
\end{align*}
we infer that
\begin{align*}
f(\xi)=\widetilde{a}_{j} \xi^{j} =|\xi| \widetilde{a}_{j} \omega ^{j} = |\xi| \sum _{m=-2}^{2} a_{m} Y_{1,m}(\omega).
\end{align*}
Here, $a_{m} \neq 0$ for at least one $m\in \{-2, -1,0,1,2\}$. Without loss of generality we assume that $a_{0}=1$. An angular momentum decomposition as in the proof of Proposition \ref{pointspectrum} leads to the inhomogeneous ordinary differential equation
\begin{align} 
\Bigg [  -(1-\rho^2) \frac{d^2}{d \rho ^2} - \left( \frac{4}{\rho} -4 \rho \right) \frac{d}{d \rho} -4  + \frac{4}{\rho ^2}  \Bigg ] u_{1,1,0} (\rho) = \rho,
\end{align}
which can be simplified to 
\begin{align} \label{ODE}
u_{1,1,0} ^{\prime \prime} (\rho) + \frac{4}{\rho} u_{1,1,0} ^{\prime} (\rho) - \frac{4}{\rho ^2} u_{1,1,0} (\rho) = - \frac{\rho}{1-\rho^2}.
\end{align}
This is a second order ordinary differential equation and one can readily verify that $\{ \phi(\rho)=\rho,~ \psi (\rho)=\rho^{-4} \}$ is a fundamental system for the homogeneous version of \eqref{ODE}. We calculate the Wronskian $W(\phi,\psi)(\rho)=-5\rho^{-4}$ and the variation of constants formula yields 
\begin{align*}
u_{1,1,0} (\rho) & = \frac{ c_{1} }{ \rho^4 } + c_{0} \rho + \frac{ \rho }{10} \log \left( 1-\rho^2 \right) +
 \frac{1}{10 \rho^4}
 \log \left (  \frac{1+\rho}{1-\rho} \right) 
 -  \frac{1}{5 \rho^4}
 \left( \rho +\frac{1}{3} \rho^{3}+\frac{1}{5} \rho^{5}  \right)
\end{align*}
for some constants $c_{0},c_{1} \in \mathbb{C}$. Now, $(\cdot)^{-4} \notin L^{2}(0,1)$ whereas $u_{1,1,0} \in L^{2}(0,1)$ and therefore we must have
$c_{1}=0$.
This fact leaves us with
\begin{align*}
u_{1,1,0} (\rho) & =  c_{0} \rho + \frac{ \rho }{10} \log \left( 1-\rho^2 \right) +
 \frac{1}{10 \rho^4}
 \log \left (  \frac{1+\rho}{1-\rho} \right) 
 -  \frac{1}{5 \rho^4}
 \left( \rho +\frac{1}{3} \rho^{3}+\frac{1}{5} \rho^{5}  \right)
\end{align*}
which behaves like $(1-\rho)\log(1-\rho)$ near $\rho=1$ and thus, does
not belong to $H^3$. This contradiction shows that we must have $k=1$ and thus $\mathbf{Q}_{0}$ has rank equal to $5$. Similarly, one can show that $\mathbf{P}_{0}$ has rank equal to $1$.
\end{proof}

\subsection{The spectrum of the full linear operator for $\alpha \neq 0$.}

Now, we assume that $\alpha \neq 0$ is sufficiently small and we will show that the spectrum $\sigma (\mathbf{L}_{\alpha})$ is close to $\sigma (\mathbf{L}_{0})$. More precisely, we work towards proving the following result.

\begin{prop}   \label{spectrumfulllineara}
Let $\alpha \in \mathbb{R}^{5}$ be sufficiently small. Then,
\begin{align*}
\sigma(\mathbf{L}_{\alpha}) \subseteq \left  \{ \lambda \in \mathbb{C}: \mathrm{Re} \lambda \leq - \frac{3}{4}  \right \} \cup \{ 0,1\}.
\end{align*}
\end{prop}
However, we start with some useful properties of $\mathbf{L}_{\alpha}$. The first crucial observation is that $\mathbf{L}_{\alpha}^{\prime}$ depends continuously on $\alpha$. 
\begin{lemma} \label{continuously}
There exists $\delta>0$ sufficiently small such that
\begin{align*}
\| \mathbf{L}_{\alpha}^{\prime}  - \mathbf{L}_{\beta}^{\prime} \| \lesssim |\alpha - \beta|,
\end{align*}
for all $\alpha,\beta \in\overline{ \mathbb{B}_{\delta}^{5} }$.
\end{lemma}
\begin{proof}
It follows from the fundamental theorem of calculus, see Lemma 4.4 in \cite{DonSch16a}.
\end{proof}
The second observation is that spectrum of $\mathbf{L}_{\alpha}$ does not differ too much from the spectrum of $\mathbf{L}_{0}$ when $\alpha$ varies in sufficiently small and compact domains of $\mathbb{R}^{5}$.
\begin{lemma} \label{compactdomains}
There exists $\delta>0$ sufficiently small such that
\begin{align*}
\lambda \in \varrho (\mathbf{L}_{0}) \Longrightarrow \lambda \in  \varrho (\mathbf{L}_{\alpha})
\end{align*}
provided $|\alpha| \leq \delta \min \{ 1, \| \mathbf{R}_{\mathbf{L}_{0}}(\lambda) \| ^{-1}\}$.
\end{lemma}
\begin{proof}
It follows from Lemma \ref{continuously} and the identity 
\begin{align*}
\lambda - \mathbf{L}_{\alpha} = \left(  1+ \left( \mathbf{L}_{0}^{\prime} - \mathbf{L}_{\alpha}^{\prime}  \right) \mathbf{R}_{\mathbf{L}_{0}}(\lambda) \right) (\lambda - \mathbf{L}_{0}),
\end{align*} 
see Corollary 4.5 in \cite{DonSch16a}.
\end{proof}
The next result shows absence of spectrum points outside a sufficiently large neighbourhood of the origin. To be precise, we provide a uniform bound on the resolvent operator of $\mathbf{L}_{\alpha}$ on the set
\begin{align*}
\Omega _{k_{0},\omega_{0}}^{\prime} :=\left  \{  \lambda \in \mathbb{C}:~~~\mathrm{Re} \lambda \geq -\frac{3}{4} \right \}\setminus \Omega _{k_{0},\omega_{0}},  
\end{align*}
where
\begin{align*}
\Omega _{k_{0},\omega_{0}} :=\left  \{  \lambda \in \mathbb{C}:~~~\mathrm{Re} \lambda \in \left [-\frac{3}{4},k_{0} \right ],~~~ \text{Im} \lambda \in [-\omega_{0},\omega_{0}] \right \},  
\end{align*}
\begin{lemma} \label{awayorigin}
Let $k_{0},\omega_{0}>0$ be sufficiently large and $\delta>0$ sufficiently small. Then there exists a positive constant $C$ such that the resolvent $\mathbf{R}_{\mathbf{L}_{\alpha}}$ exists on $\Omega _{k_{0},\omega_{0}}^{\prime} $ and satisfies the uniform bound
\begin{align*}
\| \mathbf{R} _{\mathbf{L}_{\alpha} } (\lambda)\| \leq C,
\end{align*}
for all $\lambda  \in \Omega _{k_{0},\omega_{0}}^{\prime} $ and $\alpha \in \overline{  \mathbb{B}_{\delta}^{5} }$.
\end{lemma}
\begin{proof}
Let $\lambda \in \Omega _{k_{0},\omega_{0}}^{\prime}$. The identity 
\begin{align*}
(\lambda-\mathbf {L}_{\alpha})=[1-\mathbf L'_{\alpha} \mathbf R_{\mathbf L}(\lambda)](\lambda-\mathbf L)
 \end{align*}
implies that it suffices to show smallness of $\mathbf{L}_{\alpha}^{\prime} \mathbf{R}_{\mathbf{L}} (\lambda)$ which in turn follows from choosing $k_{0},\omega_{0}>0$ sufficiently large and the bound 
\[ \|\mathbf R_{\mathbf L}(\lambda)\mathbf f\|\leq \frac{1}{\mathrm {Re}\lambda+1}\|\mathbf f\| \]
which follows from semigroup theory, see \cite{EngNag00}, page 55, Theorem 1.10. For more details see Lemma 4.6 in \cite{DonSch16a}.
\end{proof}
%
%
%
%
%
%
%
%
%
%
%
%
%
%
%
%
%
%
%
%
%
%
%
%
%
%
%
%
%
%
%
%
%
%
%
%
%
%
%

%
%
%
%
%
%
%
%
%
\begin{remark} \label{rem}
A straightforward calculation shows that the eigenspaces for the isolated eigenvalues $\lambda =0$ and $\lambda=1$ of the operator $\mathbf{L}_{\alpha}$ are spanned respectively by
\begin{align*}
& \mathbf{g}_{\alpha}(\xi) = 
\begin{pmatrix}
A_{0}(\alpha) \left( A_{0}(\alpha) - A_{j} (\alpha) \xi ^{j}  \right)^{-2} \\
2 A_{0}^{2} (\alpha) \left( A_{0}(\alpha) - A_{j} (\alpha) \xi ^{j}  \right)^{-3}
\end{pmatrix}, \\
& \mathbf{h}_{\alpha,j}(\xi) = \partial _{\alpha^{j}} \mathbf{\Psi}_{\alpha} (\xi),~~~j \in \{ 1,2,3,4,5\}.
\end{align*}
and hence $\{0,1\} \subseteq \sigma _{p} (\mathbf{L}_{\alpha})$. Finally, the above derivation shows that the algebraic multiplicities of the eigenvalues $0$ and $1$ are equal to $5-$dimensional and $1-$dimensional, respectively.
\end{remark}
With these results at hand we can now prove Proposition \ref{spectrumfulllineara}.
\begin{proof}[Proof of Proposition $\ref{spectrumfulllineara}$]
To start with, we choose $k_{0},\omega_{0}$ sufficiently large so that $\overline{ \Omega _{k_{0},\omega_{0}}^{\prime} } \subseteq \rho (\mathbf{L}_{\alpha})$ (Lemma \ref{awayorigin}) and $\delta$ sufficiently small so that $\partial \Omega _{k_{0},\omega_{0}} \subseteq \rho (\mathbf{L}_{\alpha})$ for all $|\alpha| \leq \frac{\delta}{M}$ where $M:=\max \{1,\sup_{\zeta \in \partial \Omega _{k_{0},\omega_{0}} } \| \mathbf{R}_{\mathbf{L}_{0}}(\zeta) \| \}$ (Lemma \ref{compactdomains}). Now, we define the projection
\begin{align*}
\mathbf{P}_{\alpha}^{total} := \frac{1}{2 \pi i} \int _{ \partial \Omega _{k_{0},\omega_{0}} } \mathbf{R}_{\mathbf{L}_{\alpha}} (\zeta) d \zeta.
\end{align*}
Lemma \ref{continuously} shows that $\mathbf{P}_{\alpha}^{total} $ depends continuously on $\alpha$ and therefore, from Lemma 4.10 page 34 in \cite{Kat95}, it follows that $\rg ( \mathbf{P}_{\alpha}^{total}  )$ are all isomorphic to one another and the $\mathrm{rank}\, \mathbf{P}_{\alpha}^{total}  =\dim\mathrm{rg}\, \mathbf{P}_{\alpha}^{total} $ is constant for all $\alpha$ and Lemma \ref{rank} shows that $\dim\mathrm{rg}\, \mathbf{P}_{0}^{total} =6$. In addition, the total geometric multiplicity of the eigenvalues $\lambda = 0$ and $\lambda =1$ equals $6$ and since $\mathbf{P}_{\alpha}^{total}$ has rank $6$, there can be no other eigenvalues besides $\lambda =0$ and $\lambda =1$ in $\Omega _{k_{0},\omega_{0}} $. In addition, the algebraic multiplicity of the eigenvalues $0$ and $1$ must be $5$ and $1$ respectively.
\end{proof} 
\subsection{Growth estimates for the full linearized problem}
The above spectral analysis leads to a description of the full linearised evolution. In particular, we start by partitioning the space of initial data $\mathcal{H}$ into disjoint parts and we establish growth estimates for the semigroup $\mathbf{S}_{\alpha}$ in each of these parts. Namely, we define the projections
\begin{align*}
& \mathbf{Q}_{\alpha}:=\frac{1}{2\pi i} \int_{\gamma_{0}} \mathbf{R}_{\mathbf{L}_{\alpha}}(\zeta) d \zeta, \\
& \mathbf{P}_{\alpha}:=\frac{1}{2\pi i} \int_{\gamma_{1}} \mathbf{R}_{\mathbf{L}_{\alpha}}(\zeta) d \zeta,
\end{align*}
where $\gamma _{0},\gamma_{1}:[0,1] \rightarrow \mathbb{C}$ stand for the circles centered at $\lambda=0$ and $\lambda=1$,
\begin{align*}
\gamma _{0}(s):=\frac{1}{2} e^{2\pi i s},\quad \gamma _{1}(s):=1+\frac{1}{2} e^{2\pi i s},
\end{align*}
respectively. By remark \ref{rem}, we have
\begin{align*}
\mathrm{rg}\,\mathbf{Q}_{\alpha}= \langle \mathbf{h}_{\alpha,1}, \mathbf{h}_{\alpha,2}, \mathbf{h}_{\alpha,3}, \mathbf{h}_{\alpha,4}, \mathbf{h}_{\alpha,5} \rangle
\end{align*}
and hence we may write
\begin{align*}
\mathbf{Q}_{\alpha} \mathbf{f} = \sum _{j=1}^{5} a_{j} \mathbf{h}_{\alpha,j}
\end{align*}
for coefficients $a_{j} \in \mathbb{C}$ and for all $\mathbf{f} \in \mathcal{H}$. We define the projection onto the subspace generated by $\mathbf{h}_{\alpha,j}$, that is 
\begin{align*}
\mathbf{Q}_{\alpha,j} \mathbf{f} := a_{j} \mathbf{h}_{\alpha,j}, 
\end{align*}
for all $\mathbf{f} \in \mathcal{H}$. We show that the solution operator grows exponentially on $\rg (\mathbf{P}_{\alpha})$, is constant in time on $\rg (\mathbf{Q}_{\alpha,j})$ and decays exponentially on the remaining infinite-dimensional subspace.
\begin{lemma} \label{instability}
Let $\alpha \in \mathbb{R}^{5}$ be sufficiently small. Then, the projections $\mathbf{P}_{\alpha}$ and $\mathbf{Q}_{\alpha,j}$ for $j \in \{ 1,2,3,4,5\}$ have rank one and commute with the semigroup. In addition, 
\begin{align*}
& \mathbf{S}_{\alpha}(\tau) \mathbf{P}_{\alpha} = e^{\tau} \mathbf{P}_{\alpha}, \\
& \mathbf{S}_{\alpha}(\tau) \mathbf{Q}_{\alpha,j} = \mathbf{Q}_{\alpha,j}, \\
& \| \mathbf{S}_{\alpha}(\tau) \widetilde{ \mathbf{P}}_{\alpha}  \| \lesssim e^{-\frac{2}{3} \tau}  \| \widetilde{ \mathbf{P}}_{\alpha}  \|,
\end{align*}
where $\widetilde{ \mathbf{P}}_{\alpha} := \mathbf{I} - \mathbf{P}_{\alpha} - \mathbf{Q}_{\alpha}$. Furthermore, 
\begin{align*}
& \rg ( \mathbf{P}_{\alpha}) =  \langle \mathbf{g}_{\alpha} \rangle, \\
& \rg ( \mathbf{Q}_{\alpha,j})=  \langle \mathbf{h}_{\alpha,j}  \rangle, ~~~j \in \{ 1,2,3,4,5\},
\end{align*}
where $\mathbf{g}_{\alpha}$ and $\mathbf{h}_{\alpha,j}$ are eigenfunctions of $\mathbf{L}_{\alpha}$ with eigenvalues $1$ and $0$, respectively.
\end{lemma}
\begin{proof}
The growth estimates follow from the Gearhart-Pr{\"u}{\ss} Theorem (\cite{EngNag00}, page 302, Theorem 1.11) since Lemma \ref{spectrumfulllineara} and Lemma \ref{awayorigin} yield $\sup _{ \mathrm{Re}\zeta \geq -\frac{3}{4}} \| \mathbf{R}_{\mathbf{L}_{\alpha} } (\zeta) \widetilde{\mathbf{P} }_{\alpha} \|  < \infty$. The remaining statements are consequences of Lemma \ref{spectrumfulllineara}. For more details see Proposition 4.8, page 30, \cite{DonSch16a}.
\end{proof}
 \begin{remark} \label{propertiesQ}
It follows that $ \mathbf{Q}_{\alpha,j}  \mathbf{Q}_{\alpha,k} = \delta_{jk}  \mathbf{Q}_{\alpha,j}$ and  $ \mathbf{Q}_{\alpha,j}  \mathbf{P}_{\alpha} = \mathbf{P}_{\alpha}   \mathbf{Q}_{\alpha,j}  =0 $.
 \end{remark}


\section{Non-Linear Estimates}

In this section, we establish Lipschitz-type estimates for the eigenfunctions $\mathbf{g}_{\alpha}$, $\mathbf{h}_{\alpha,j}$, the projections $\mathbf{P}_{\alpha}$, $\mathbf{Q}_{\alpha}$, the semigroup $\mathbf{S}_{\alpha}$ as well as for the nonlinearity $\mathbf{N}_{\alpha}$. These estimates will be used later for the main fixed point theorem. To begin with, we prove the following result.
\begin{lemma} \label{nonlinearestimates2}
For all $\alpha,\beta \in \mathbb{R}^{5}$ and for all $j \in \{ 1,2,3,4,5\}$, we have
\begin{align}
& \| \mathbf{g}_{\alpha}  -\mathbf{g}_{\beta} \| + \| \mathbf{h}_{\alpha,j}  -\mathbf{h}_{\beta,j} \| \lesssim |\alpha-\beta|, \label{Lipg} \\
&\| \mathbf{P}_{\alpha}  -\mathbf{P}_{\beta} \| + \| \mathbf{Q}_{\alpha}  -\mathbf{Q}_{\beta} \| \lesssim |\alpha-\beta|, \label{LipP} \\
&  \| \mathbf{S}_{\alpha} (\tau) \widetilde{\mathbf{P}}_{\alpha}  -   \mathbf{S}_{\beta} (\tau) \widetilde{\mathbf{P}}_{\beta} \| \lesssim |\alpha-\beta| e^{-\frac{1}{2} \tau}, \label{LipS}
\end{align}
for all $\tau >0$.
\end{lemma}
\begin{proof}
The estimate \eqref{Lipg} follows immediately from the fundamental theorem of calculus. Furthermore, the estimate \eqref{LipP} follows from a Lipschitz-type estimate for the resolvent operator, namely
\begin{align*}
 \| \mathbf{R}_{\mathbf{L}_{\alpha}} (\lambda) -  \mathbf{R}_{\mathbf{L}_{\beta}}  (\lambda) \|  \| \lesssim  \|  \mathbf{R}_{\mathbf{L}_{\alpha}} (\lambda)  \|  \|  \mathbf{R}_{\mathbf{L}_{\beta}} (\lambda)  \|    |\alpha-\beta|,
\end{align*}
which in turn follows from the identity
\begin{align*}
\mathbf{A}^{-1} - \mathbf{B}^{-1} = \mathbf{B}^{-1} (\mathbf{B}-\mathbf{A}) \mathbf{A}^{-1},
\end{align*}
valid for all invertible operators $\mathbf{A}$ and $\mathbf{B}$. Finally, we establish the estimate \eqref{LipS} for the semigroup. To do so, we first observe that the function
\begin{align*}
\Phi_{\alpha,\beta} (\tau):= \frac{   \mathbf{S}_{\alpha} (\tau) \widetilde{\mathbf{P}}_{\alpha} \mathbf{u} -   \mathbf{S}_{\beta} (\tau) \widetilde{\mathbf{P}}_{\beta} \mathbf{u} }{ |\alpha-\beta| }
\end{align*}  
for $\mathbf{u} \in \mathcal{D} (\mathbf{L}) \subseteq \mathcal{H}$, solves the initial value problem
\begin{align*}
  \left \{
  \begin{aligned}
    & \partial_{\tau} \Phi_{\alpha,\beta} (\tau) = \mathbf{L}_{\alpha} \widetilde{ \mathbf{P} }_{\alpha} \Phi_{\alpha,\beta} (\tau)  +
 \frac{   \mathbf{L}_{\alpha} \widetilde{\mathbf{P}}_{\alpha} - \mathbf{L}_{\beta}  \widetilde{\mathbf{P}}_{\beta} }{ |\alpha-\beta| } \mathbf{S}_{\beta} (\tau)  \widetilde{\mathbf{P}}_{\beta} \mathbf{u},&& \ \\
    &  \Phi_{\alpha,\beta} (0) =  \frac{ \widetilde{\mathbf{P}}_{\alpha} - \widetilde{\mathbf{P}}_{\beta} }{ |\alpha-\beta| }  \mathbf{u}. && 
  \end{aligned} \right.
\end{align*} 
The key observation here is that 
\begin{align*}
\mathbf{L}_{\alpha}  \widetilde{\mathbf{P}}_{\alpha} - \mathbf{L}_{\beta}  \widetilde{\mathbf{P}}_{\beta} = \mathbf{L}_{\alpha}^{\prime} - \mathbf{L}_{\beta}^{\prime} +\mathbf{P}_{\beta} - \mathbf{P}_{\alpha}
\end{align*}
and therefore the apparently unbounded operator $\mathbf{L}_{\alpha}  \widetilde{\mathbf{P}}_{\alpha} - \mathbf{L}_{\beta}  \widetilde{\mathbf{P}}_{\beta}$ is in fact bounded, that is
\begin{align*}
\| \mathbf{L}_{\alpha}  \widetilde{\mathbf{P}}_{\alpha} - \mathbf{L}_{\beta}  \widetilde{\mathbf{P}}_{\beta} \| \lesssim |\alpha - \beta|.
\end{align*}
Now, it remains to apply Duhamel's principle, write down the general solution formula for $\Phi_{\alpha,\beta} (\tau) $ and use the previous estimates. For more details see Lemma 4.9 in \cite{DonSch16a}.
\end{proof}
Next, we establish a Lipschitz-type estimate for the nonlinearity $\mathbf{N}_{\alpha}$. To begin with, recall \eqref{systempsi}, \eqref{Lprime} and \eqref{static}, i.e.,
\begin{align*}
\mathbf{N} \left( \mathbf{u} \right)(\xi):=
 \begin{pmatrix}
0 \\
u_{1}^{3}(\xi)  
\end{pmatrix}
\end{align*}
and
\begin{align*}
 \mathbf{L}^{\prime}_{\alpha } (\mathbf{u}(\xi)):=
 \begin{pmatrix}
0 \\
V_{\alpha} (\xi) u_{1}(\xi)  
\end{pmatrix},\quad
V_{\alpha } (\xi):=3\psi_{\alpha }^{2}(\xi), \quad
\psi_{\alpha} (\xi):=\frac{\sqrt{2}}{A_{0}(\alpha)-A_{j}(\alpha)\xi^{j} }.
\end{align*}
Furthermore, recall that $A_{0}(\alpha) = \mathcal{O}(1)$ whereas $A_{j}(\alpha) = \mathcal{O}(\alpha)$ for all sufficiently small $\alpha \in \mathbb{R}^{d}$. Hence, we find $\epsilon>0$ small enough so that
\begin{align} \label{psibound}
\sup_{ |\alpha|<\epsilon } ~~ \sup_{j \in \{ 0,1,2,3 \}} \| \partial ^{j} \psi_{\alpha} \|_{L^{\infty}(\mathbb{B}^{5})} \lesssim 1.
\end{align}
A direct calculation shows that the full non-linearity defined in \eqref{nonlinear} can be written as follows
\begin{align}  \label{Naformula}
 \mathbf{N}_{\alpha} (\mathbf{u}) :=  \mathbf{N} (\mathbf{u} + \mathbf{\Psi}_{\alpha})  - \mathbf{N} (\mathbf{\Psi}_{\alpha})-\mathbf{L}^{\prime}_{\alpha} \mathbf{u}   
=
 \begin{pmatrix}
 0 \\
\hat{N}(\psi_{\alpha},u_{1})
 \end{pmatrix},
\end{align}
where
\begin{align*}
\hat{N}(\psi_{\alpha}(\xi),u_{1}(\xi)):= 3 \psi_{\alpha}(\xi) u_{1}^{2}(\xi)  + u_{1}^{3}(\xi).
\end{align*}
Also, we define
\begin{align*}
\hat{M}(\psi_{\alpha}(\xi),u_{1}(\xi)):=\partial_{2} \hat{N}(\psi_{\alpha}(\xi),u_{1}(\xi)):= 6 \psi_{\alpha}(\xi) u_{1}(\xi)  +3  u_{1}^{2}(\xi).
\end{align*}
Finally, we write $\| \mathbf{f} \|:= \| \mathbf{f} \|_{\mathcal{H}}$ where $\mathcal{H}:=H^{3} \big( \mathbb{B}^{5} \big) \times H^{2} \big( \mathbb{B}^{5} \big) $. We prove the following result.
\begin{lemma} \label{nonlinearestimate}
Fix sufficiently small $\alpha \in \mathbb{R}^{5} $ and sufficiently small $\delta>0$. Then, we have
\begin{align} \label{Lipschitz}
\left \|  \mathbf{N}_{\alpha} (\mathbf{u}) -  \mathbf{N}_{\beta} (\mathbf{v}) \right \| \lesssim \big( \| \mathbf{u} \| + \| \mathbf{v} \| \big) \| \mathbf{u} - \mathbf{v} \| +  (  \| \mathbf{u} \|^{2} + \| \mathbf{v} \|^{2} ) |\alpha - \beta|,
\end{align}
for all $\mathbf{u},\mathbf{v} \in \mathcal{H}$ with $\| \mathbf{u} \| \leq \delta$ and $\| \mathbf{v} \| \leq \delta$ and for all $\alpha,\beta \in \overline{ \mathbb{B}_{\delta}^{5} }$.
\end{lemma}
\begin{proof}
To begin with, we fix sufficiently small $\delta>0$, sufficiently small $\alpha \in \overline{ \mathbb{B}_{\delta}^{5} }$ and pick any $\mathbf{u},\mathbf{v} \in \mathcal{H}$ with $\| \mathbf{u} \| \leq \delta$ and $\| \mathbf{v} \| \leq \delta$. First, we show that
\begin{align} \label{first}
\big \| \mathbf{N_{\alpha} (u)} -\mathbf{N_{\alpha} (v)} \big\| \lesssim \left \| \mathbf{u} - \mathbf{v}  \right \| \left( \left \|   \mathbf{u} \right \|  + \left \|  \mathbf{v}   \right \| \right).
\end{align} 
Notice that the function $G(\xi,\zeta):=\hat{M} (\psi_{\alpha}(\xi), \zeta)= 6 \psi_{\alpha}(\xi) \zeta +3  \zeta^2,~(\xi,\zeta) \in \mathbb{R}^{5}\times\mathbb{R}$ belongs to $C^{\infty}(\mathbb{R}^{5}\times\mathbb{R};\mathbb{R})$ and $G(\xi,0)=0$. Furthermore, for any compact set $K \subseteq \mathbb{R}$, we have $\partial_{\xi,\zeta}^{\alpha} G \in L^{\infty} \left( \mathbb{R}^{5} \times K \right)$, for all multi-indices $\alpha$ with $|\alpha| \leq 4$, due to \eqref{psibound}.
Consequently, Moser's inequality (see \cite{Rau12}, p.~224, Theorem 6.4.1) and Sobolev extension imply
\begin{align} \label{Mhat}
\| \hat{M}(\psi_{\alpha},w)\|_{H^{3}(\mathbb B^{5})} \lesssim \|w\|_{H^{3}(\mathbb B^{5})}, 
\end{align}
for all $w \in H^{3}(\mathbb B^{5})$. For any fixed $\sigma \in [0,1]$, we define $ \zeta(\sigma):= \sigma u_{1}+(1- \sigma) v_{1}$. Now, since $3>\frac{5}{2}$, the algebra property
\begin{align} \label{algebraproperty}
\| fg \|_{ H ^{3} (\mathbb{B}^{5}) } \lesssim \| f \|_{ H ^{3} (\mathbb{B}^{5}) } \| g \|_{ H ^{3} (\mathbb{B}^{5}) }
\end{align}
holds and we can use this together with \eqref{Mhat} to estimate
\begin{align*} 
\big \| \mathbf{N_{\alpha} (u)} -\mathbf{N_{\alpha} (v)} \big\|  &
= \big \|  \hat N (\psi_{\alpha}, u_{1} ) - \hat N (\psi_{\alpha}, v_{1} ) \big \| _{H^{2} (\mathbb{B}^{5}) } \\
&\leq  \big \|  \hat N (\psi_{\alpha}, u_{1} ) - \hat N (\psi_{\alpha}, v_{1} ) \big \| _{H^{3} (\mathbb{B}^{5}) } \\
&=\left \|  \int _{v_{1}}^{u_{1} }  \partial _{2} \hat N (\psi_{\alpha}, \zeta )  d \zeta  \right \|  _{H^{3} (\mathbb{B}^{5})}   \\
&=\left \|   \left( u_{1} -  v_{1} \right) \int _{0}^{1} \partial _{2} \hat N (\psi_{\alpha} ,\zeta(\sigma))  d \sigma \right \| _{H^{3} (\mathbb{B}^{5}) }   \\ \nonumber 
&\lesssim \left \|   u_{1} - v_{1} \right \| _{ H ^{3} (\mathbb{B}^{5}) }  \left \|  \int _{0}^{1} \partial _{2} \hat N (\psi_{\alpha} , \zeta (\sigma) ) d \sigma \right \| _{H^{3} (\mathbb{B}^{5}) }  \\
& \lesssim \left \|   u_{1} - v_{1} \right \| _{ H^{3} (\mathbb{B}^{5}) }   \int _{0}^{1}  \left \|  \hat{M} (\psi_{\alpha}, \zeta (\sigma) ) \right \| _{H^{3} (\mathbb{B}^{5}) }  d \sigma   \\
&  \lesssim  \left \| u_{1} - v_{1}  \right \| _{ H^{3} (\mathbb{B}^{5}) }  \int _{0}^{1}  \left \|  \zeta (\sigma)  \right \| _{H^{3} (\mathbb{B}^{5}) }   d \sigma  \\
& \lesssim \left \| u_{1} - v_{1}  \right \| _{ H^{3} (\mathbb{B}^{5}) }  \int _{0}^{1} \left( \sigma  \left \|  u_{1} \right \| _{H^{3} (\mathbb{B}^{5}) }  +(1-\sigma) \left \|  v_{1}    \right \| _{H^{3} (\mathbb{B}^{5}) }  \right)  d \sigma   \\
& \lesssim \left \| u_{1} - v_{1}  \right \| _{ H^{3} (\mathbb{B}^{5}) }  \left( \left \|   u_{1} \right \| _{H^{3} (\mathbb{B}^{5}) }  + \left \|  v_{1}    \right \| _{H^{3} (\mathbb{B}^{5}) }  \right) \\
&\lesssim   \left \| \mathbf{u} - \mathbf{v}  \right \| \left( \left \|   \mathbf{u} \right \|  + \left \|  \mathbf{v}   \right \| \right).
\end{align*}
To complete the proof, it suffices show that
\begin{align*}
\|  \mathbf{N}_{\alpha} (\mathbf{u}) -  \mathbf{N}_{\beta} (\mathbf{u}) \|  \lesssim  \| u_{1}  \|_{H^{3}(\mathbb{B}^{5})}^{2} |\alpha - \beta|,
\end{align*}
which is a consequence of the fundamental theorem of calculus. Indeed, we fix $\alpha,\beta \in \mathbb{R}^{5}$ sufficiently small and let $\gamma (t):=t\beta +(1-t)\alpha,~t\in[0,1]$ be a parametrisation of the line segment $E[\alpha,\beta]$ joining $\alpha$ and $\beta$. Then,
\begin{align*}
\psi _{\alpha} - \psi _{\beta} & = \psi _{\gamma(0)} - \psi _{\gamma(1)} = \int_{E[\alpha,\beta]} \partial \psi _{\gamma} \cdot d \ell = \sum _{j=1}^{5} (\beta ^{j}-\alpha ^{j}) \int_{0}^{1} \partial _{\gamma ^{j}} \psi _{\gamma(t)} dt,
\end{align*} 
and the triangle inequality implies the bound
\begin{align*} 
\| \partial^{m} \left( \psi _{\alpha} - \psi _{\beta} \right) \|_{L^{2}(\mathbb{B}^{5})} \lesssim \sum _{j=1}^{5} |\beta ^{j} - \alpha ^{j} | 
\sup _{s \in E[\alpha,\beta]} \|  \partial ^{m} \partial _{\gamma ^{j}} \psi _{s} \|_{L^{2}(\mathbb{B}^{5})}
 \lesssim  |\beta- \alpha  |,
\end{align*}
for all $m \in \{0,1,2,3 \}$, due to \eqref{psibound}. Therefore, \eqref{algebraproperty} yields
\begin{align*}
\|  \mathbf{N}_{\alpha} (\mathbf{u}) -  \mathbf{N}_{\beta} (\mathbf{u}) \| &  =  \| 3 u_{1}^{2} (\psi _{\alpha} - \psi _{\beta}) \|_{H^{2}(\mathbb{B}^{5})}  \leq  \| 3 u_{1}^{2} (\psi _{\alpha} - \psi _{\beta}) \|_{H^{3}(\mathbb{B}^{5})} \\
& \lesssim  \| u_{1} \|_{H^{3}(\mathbb{B}^{5})}^{2} 
  \| \psi _{\alpha} - \psi _{\beta} \|_{H^{3}(\mathbb{B}^{5})} 
   \lesssim  \| u_{1}  \|_{H^{3}(\mathbb{B}^{5})}^{2} |\alpha - \beta|,
\end{align*}
which concludes the proof.
\end{proof}

%
%
%
%
%
%

\section{The modulation equation}

To begin with, we apply Duhamel's principle to rewrite the modulation equation \eqref{modulation} coupled with initial data in a weak formulation. Due to \eqref{Duhamel}, we may write the Cauchy problem
  \[
    \left\{\begin{array}{lr}
        \partial _{\tau} \mathbf{\Phi}(\tau) - \big( \mathbf{L} + \mathbf{L}^{\prime}_{\alpha _{\infty}} \big) \mathbf{\Phi}(\tau) = 
\hat{\mathbf{L}}_{\alpha(\tau)} \mathbf{\Phi}(\tau) + \mathbf{N}_{\alpha(\tau)} (\mathbf{\Phi}(\tau)) - \partial _{\tau} \mathbf{\Psi}_{\alpha(\tau)}, & \\
        \Phi (0) = \mathbf{u} \in \mathcal{H}, & 
        \end{array}\right.
  \]
  as an integral equation, that is
\begin{align} \label{Duhamelmodulation}
\mathbf{\Phi }(\tau) = \mathbf{S}_{\alpha _{\infty}} (\tau) \mathbf{u} + \int_{0}^{\tau} \mathbf{S}_{\alpha _{\infty}} (\tau - \sigma) \Big( 
\hat{\mathbf{L}}_{\alpha(\sigma)} \mathbf{\Phi}(\sigma) + \mathbf{N}_{\alpha(\sigma)} (\mathbf{\Phi}(\sigma)) - \partial _{\sigma} \mathbf{\Psi}_{\alpha(\sigma)}
\Big) d \sigma,
\end{align}
provided that $\alpha_{\infty}$ is sufficiently small which we later verify, see \eqref{ainfty}. We use this formulation to define the notion of light-cone solutions.
\begin{definition} \label{def}
Fix $\alpha \in \mathbb{R}^{5}$ sufficiently small. We say that $u: C_{T} \longrightarrow \mathbb{R}$ is a solution to \eqref{Cauchy} if the corresponding $\mathbf{\Phi}:[0,\infty) \longrightarrow \mathcal{H}$ belongs to 
$C \big( [0,\infty);\mathcal{H} \big)$ and satisfies \eqref{Duhamelmodulation} for all $\tau \ge 0$.
\end{definition} 
Consequently, in order to establish a solution $u= u(t,x)$ to the initial Cauchy problem \eqref{Cauchy} we need to construct a global in $\tau$ solution $\mathbf{\Phi }(\tau)$ to \eqref{Duhamelmodulation}. To prove the existence of a global solution, we would like to apply a fixed point argument to the integral equation \eqref{Duhamelmodulation}. However, the solution operator $\mathbf{S}_{\alpha_{\infty}}$ for the linearized equation has two unstable subspaces $\rg \mathbf{Q}_{\alpha_{\infty}} $, $\rg \mathbf{P}_{\alpha_{\infty}} $ which appear due to the symmetries of the original equation, namely the Lorentz and time-translation symmetry, respectively (Lemma \ref{instability}). Specifically, initial data from $\rg \mathbf{Q}_{\alpha_{\infty}}$ and $\rg \mathbf{P}_{\alpha_{\infty}} $ lead to solutions which stay constant or grow exponentially in time, respectively. These growths prevent us from applying a fixed point argument directly. We overcome this obstruction as follows. In the first case, we choose the rapidity parameter $\alpha = \alpha (\tau)$ in such a way that this instability is suppressed. In the second case, we proceed differently and add a correction term to the initial data which stabilizes the evolution. In both cases, we use fixed point arguments to establish existence and uniqueness of the respective modified equations and hence we first introduce the Banach spaces.
\subsection{Banach spaces}
We define the following sets. 
\begin{align*}
& \mathcal{X}:=\left \{ \Phi \in C( [0,\infty);\mathcal{H}): \| \Phi \|_{\mathcal{X}} < \infty \right  \}, \\
& X:=\left \{ \alpha \in C^{1}( [0,\infty);\mathbb{R}^{5}): \alpha(0)=0 \text{~and~} \| \alpha \|_{X} < \infty \right  \},
\end{align*}
endowed with the norms
\begin{align*}
& \| \Phi \|_{\mathcal{X}}:=\sup_{\tau>0}  \left(e^{\frac{1}{2} \tau} \| \Phi (\tau) \| \right), \\
& \| \alpha \|_{X}:=\sup_{\tau>0} \left( e^{\frac{1}{2} \tau} | \dot{\alpha} (\tau) | + |\alpha(\tau)| \right),
\end{align*}
on $\mathcal{X}$ and $X$ respectively. Furthermore, we denote by
\begin{align*}
& \mathcal{X}_{\delta}:= \left \{ \Phi \in \mathcal{X}: \| \Phi \|_{\mathcal{X}} \leq \delta  \right \}, \\
& X_{\delta}:= \left \{ \alpha \in X: |\dot{\alpha} (\tau)| \leq \delta e^{-\frac{1}{2} \tau}  \right \},
\end{align*}
the closed subsets of $\mathcal{X}$ and $X$ respectively. Recall that $\mathcal{H}:=H^{3}(\mathbb{B}^{5}) \times H^{2}(\mathbb{B}^{5})$ and $\| \cdot \|:=\| \cdot \|_{H^{3}(\mathbb{B}^{5}) \times H^{2}(\mathbb{B}^{5})}$. First, notice that for an element $\alpha \in X_{\delta}$, the limit $\alpha_{\infty}:=\lim_{\tau \rightarrow \infty} \alpha(\tau)$ exists. Indeed, for all $0 < \tau_{1} \leq \tau_{2}$ with $\tau_{1},\tau_{2} \rightarrow \infty$,
\begin{align*}
\left |  \alpha(\tau_{2}) -\alpha(\tau_{1}) \right | \leq \int_{\tau_{1}}^{\tau_{2}} \left | \dot{\alpha}(\tau) \right | d \tau \lesssim \delta \left( e^{-\frac{1}{2} \tau_{1}} - e^{-\frac{1}{2} \tau_{2}} \right) \longrightarrow 0.
\end{align*}
Fixing $\tau_{1}$ and letting $\tau_{2}$ go to infinity, we obtain 
\begin{align} \label{aclose}
\forall \alpha \in X_{\delta}: \quad \left |  \alpha_{\infty} -\alpha(\tau) \right |  \lesssim \delta e^{-\frac{1}{2} \tau},~~~ \forall \tau >0.
\end{align}
In particular for $\tau=0$ we get the smallness condition 
\begin{align} \label{ainfty}
\left |  \alpha_{\infty}  \right |  \lesssim \delta.
\end{align}
Furthermore, by Lemma \ref{nonlinearestimate}, Lemma \ref{continuously}, Proposition \ref{instability} and the fact that $\partial _{\tau} \Psi _{\alpha (\tau)}=\dot{\alpha}^{k}(\tau) \mathbf{h}_{\alpha(\tau),k}$ we get the following result.
\begin{lemma} \label{bounds1}
Let $\delta >0$ be sufficiently small. Then, for all $\Phi \in \mathcal{X}_{\delta}$ and $\alpha \in X_{\delta}$, 
\begin{align*}
& \left \| \hat{ \mathbf{L} }_{\alpha (\tau)} \Phi (\tau) \right \| + \left \| \mathbf{N}_{\alpha(\tau)}( \Phi (\tau) ) \right \| \lesssim \delta ^{2} e^{- \tau}, \\
& \left \|  \mathbf{P}_{\alpha_{\infty}} \partial _{\tau} \Psi_{\alpha (\tau)} \right \| + \left \| (\mathbf{I}-\mathbf{Q}_{\alpha_{\infty}}) \partial _{\tau} \Psi _{\alpha (\tau)} \right \| \lesssim \delta ^{2} e^{- \tau}.
\end{align*}
for all $\tau>0$.
\end{lemma}
\begin{proof}
The proof coincides with the proof of Lemma 5.4 in \cite{DonSch16a}. 
\end{proof}
We also prove the corresponding Lipschitz bounds.
\begin{lemma} \label{bounds2}
Let $\delta >0$ be sufficiently small. Then, for all $\Phi, \Psi \in \mathcal{X}_{\delta}$ and $\alpha, \beta \in X_{\delta}$, 
\begin{align*}
& \left \| \hat{ \mathbf{L} }_{\alpha (\tau)} \Phi (\tau) - \hat{ \mathbf{L} }_{\beta (\tau)} \Psi (\tau) \right \| 
\lesssim \delta ^{2} e^{- \tau} \left( \| \Phi - \Psi \|_{\mathcal{X}} + \| \alpha - \beta \|_{X}  \right), \\
& \left \| \mathbf{N}_{\alpha (\tau)} (\Phi (\tau)) - \mathbf{N}_{\beta (\tau)} (\Psi (\tau))  \right \| 
\lesssim \delta ^{2} e^{- \tau} \left( \| \Phi - \Psi \|_{\mathcal{X}} + \| \alpha - \beta \|_{X}  \right), \\
& \left \|  \mathbf{P}_{\alpha_{\infty}} \partial _{\tau} \Psi_{\alpha (\tau)}-  \mathbf{P}_{\beta_{\infty}} \partial _{\tau} \Psi_{\beta (\tau)} \right \|  
\lesssim \delta ^{2} e^{- \tau}  \| \alpha - \beta \|_{X}, \\
& \left \| (\mathbf{I}-\mathbf{Q}_{\alpha_{\infty}}) \partial _{\tau} \Psi _{\alpha(\tau)} - (\mathbf{I}-\mathbf{Q}_{\beta_{\infty}}) \partial _{\tau} \Psi _{\beta(\tau)} \right \| 
\lesssim \delta ^{2} e^{- \tau} \left( \| \Phi - \Psi \|_{\mathcal{X}} + \| \alpha - \beta \|_{X}  \right),
\end{align*}
for all $\tau>0$.
\end{lemma}
\begin{proof}
The proof coincides with the proof of Lemma 5.5 in \cite{DonSch16a}. 
\end{proof}
\subsection{The Lorentz symmetry instability} Now, we focus on the instability induced by the Lorentz symmetry and in particular we will choose $\alpha = \alpha (\tau)$ in such a way that this instability is suppressed. To do so, we need an equation for $\alpha = \alpha (\tau)$. By Proposition \ref{instability}, we have $\mathbf{Q}_{\alpha_{\infty},j} \mathbf{S}_{\alpha_{\infty}}=\mathbf{Q}_{\alpha_{\infty},j}$ and therefore applying $\mathbf{Q}_{\alpha_{\infty},j}$ to the weak formulation of the modulation equation, that is \eqref{Duhamelmodulation}, we infer
\begin{align*}
\mathbf{Q}_{\alpha_{\infty},j} \mathbf{\Phi }(\tau) = \mathbf{Q}_{\alpha_{\infty},j} \mathbf{u} + \mathbf{Q}_{\alpha_{\infty},j} \int_{0}^{\tau} \mathbf{S}_{\alpha _{\infty}} (\tau - \sigma) \Big( 
\hat{\mathbf{L}}_{\alpha(\sigma)} \mathbf{\Phi}(\sigma) + \mathbf{N}_{\alpha(\sigma)} (\mathbf{\Phi}(\sigma)) - \partial _{\sigma} \mathbf{\Psi}_{\alpha(\sigma)}
\Big) d \sigma,
\end{align*}
for all $j \in \{1,2,3,4,5 \}$. To suppress the instability we would like to trivialize the range and set the right-hand side equal to zero. However, this is not possible since for $\tau=0$ the condition $\mathbf{Q}_{\alpha_{\infty},j} \mathbf{u}=0$ on the initial data would be required which is not true in general. Since we are only interested in the long-term evolution it however suffices to assume that $\mathbf{Q}_{\alpha_{\infty},j} \mathbf{\Phi }(\tau) $ vanishes for large $\tau$. Hence, we set
\begin{align*}
\mathbf{Q}_{\alpha_{\infty},j} \mathbf{\Phi }(\tau)  = \chi (\tau) \mathbf{h},\quad  \mathbf{h}:=\mathbf{Q}_{\alpha_{\infty},j} \mathbf{u} \in \rg \mathbf{Q}_{\alpha_{\infty}}
\end{align*}
where $\chi$ is a smooth cut-off function, which equals to $1$ on $[0, 1]$, $0$ for $\tau \geq 4$ and satisfies $|\dot{\chi}| \leq 1$ everywhere. Now, evaluation at $\tau=0$ yields $ \mathbf{h}=\mathbf{Q}_{\alpha_{\infty},j} \mathbf{u} $ which now holds true in general. This ansatz yields an equation for $\alpha$, namely
\begin{align} \label{mod2}
(1-\chi(\tau))\mathbf{h} + \mathbf{Q}_{\alpha_{\infty},j} \int_{0}^{\tau} \Big( 
\hat{\mathbf{L}}_{\alpha(\sigma)} \mathbf{\Phi}(\sigma) + \mathbf{N}_{\alpha(\sigma)} (\mathbf{\Phi}(\sigma)) - \partial _{\sigma} \mathbf{\Psi}_{\alpha(\sigma)} \Big) d \sigma = \mathbf{0}.
\end{align}
In particular, we define the auxiliary function
\begin{align*}
\hat{\mathbf{h}}_{\alpha(\tau),k}:=\mathbf{h}_{\alpha(\tau),k} - \mathbf{h}_{\alpha_{\infty},k},
\end{align*}
assume that $\alpha(0)=0$ and use the properties of $ \mathbf{Q}_{\alpha_{\infty},j}$ from remark \ref{propertiesQ} to write
\begin{align*}
 \mathbf{Q}_{\alpha_{\infty},j} \int _{0}^{\tau} \partial _{\sigma} \mathbf{\Psi}_{\alpha(\sigma)} d \sigma & = 
  \mathbf{Q}_{\alpha_{\infty},j} \int _{0}^{\tau}  \dot{\alpha}^{k}(\sigma) \mathbf{h}_{\alpha(\sigma),k} d \sigma \\
&  =   \mathbf{Q}_{\alpha_{\infty},j} \int _{0}^{\tau}  \dot{\alpha}^{k}(\sigma)\left( \hat{\mathbf{h}}_{\alpha(\tau),k} + \mathbf{h}_{\alpha_{\infty},k}   \right) d \sigma  \\
& =  \mathbf{Q}_{\alpha_{\infty},j} \int _{0}^{\tau}  \dot{\alpha}^{k}(\sigma) \hat{\mathbf{h}}_{\alpha(\tau),k}  d \sigma + \alpha^{j}(\tau) \mathbf{h}_{\alpha_{\infty},j}.  
\end{align*}
Therefore, we can write equation \eqref{mod2} as
\begin{align} \label{modsolve}
\alpha^{j}(\tau) \mathbf{h}_{\alpha_{\infty},j} & = (1-\chi(\tau)) \mathbf{Q}_{\alpha_{\infty},j} \mathbf{u} \nonumber \\
& +  \mathbf{Q}_{\alpha_{\infty},j} \int_{0}^{\tau} \Big( 
\hat{\mathbf{L}}_{\alpha(\sigma)} \mathbf{\Phi}(\sigma) + \mathbf{N}_{\alpha(\sigma)} (\mathbf{\Phi}(\sigma)) \Big) d \sigma \nonumber  \\
& -  \mathbf{Q}_{\alpha_{\infty},j} \int_{0}^{\tau} \dot{\alpha}^{k}(\sigma) \hat{\mathbf{h}}_{\alpha(\tau),k} d \sigma \nonumber \\
& := \int_{0}^{\tau} \mathbf{G}_{j}(\alpha,\Phi,\mathbf{u})(\sigma) d \sigma.
\end{align}
for the functions $\alpha^{j} = \alpha ^{j} (\tau) \in \mathbb{R}^{5}$, $j \in \{1,2,3,4,5\}$. Then, we have a fixed point formulation for $\alpha$,
\begin{align} \label{modsolve2}
\alpha(\tau)= \int_{0}^{\tau} G(\alpha,\Phi,\mathbf{u}):=\widetilde{G}(\alpha,\Phi,\mathbf{u}),
\end{align}
where $G=(G_{1},G_{2},G_{3},G_{4},G_{5})$ and
\begin{align*}
	G_{j} ( \alpha,\Phi,\mathbf{u})(\sigma) := \frac{1}{ \| \mathbf{h}_{\alpha_{\infty,j} } \|^{2}} \left( \mathbf{G}_{j} (\alpha,\Phi,\mathbf{u})(\sigma) | \mathbf{h}_{\alpha_{\infty},j} \right).
\end{align*}
 Finally, we use a fixed point argument to show that the function $\alpha:[0,\infty) \rightarrow \mathbb{R}^{5}$ can be chosen in such a way that \eqref{modsolve2} (equivalently \eqref{modsolve}) holds provided that $\Phi$ satisfies a smallness condition. Consequently, the instability induced by the Lorentz symmetry is suppressed.  
\begin{prop} \label{suppressedLorentz}
Let $\delta>0$ be sufficiently small, $c>0$ sufficiently large and suppose that $\Phi \in \mathcal{X}_{\delta}$. Then, there exists a unique function $\alpha \in X_{\delta}$ such that equation \eqref{modsolve2} holds for each $j \in \{1,2,3,4,5\}$ provided $\| \mathbf{u} \| \leq \frac{\delta}{c}$. Furthermore, the map $\Phi \longmapsto \alpha$ is Lipschitz continuous. 
\end{prop}
\begin{proof}
The proof relies on a fixed point argument. The fact that $\widetilde{G}(\cdot,\Phi,\mathbf{u})$ maps $X_{\delta}$ to itself follows from Lemma \ref{nonlinearestimates2} and Lemma \ref{bounds2}. Furthermore, the contraction property is a direct consequence of Lemma \ref{bounds2} and finally the Lipschitz continuity follows from Lemma \ref{bounds2}. For more details see Lemma 5.6 in \cite{DonSch16a}. 
\end{proof}
\subsection{The time translation instability} 
Next, we turn our attention to the instability induced by the time translation symmetry. However, this time we proceed differently and we add a correction term to the initial data $\mathbf{\Phi}(0)=\mathbf{u}$ in the equation \eqref{Duhamelmodulation} which stabilizes the evolution. In other words, we consider the modified equation
\begin{align} \label{Modified}
\mathbf{\Phi }(\tau) = \mathbf{K}(\Phi,\alpha,\mathbf{u}),
\end{align}
where
\begin{align} \label{Kappa}
 \mathbf{K}(\Phi,\alpha,\mathbf{u})&:= \mathbf{S}_{\alpha _{\infty}} (\tau) \left(\mathbf{u}-\mathbf{C}(\Phi,\alpha,\mathbf{u}) \right) \nonumber \\
 & + \int_{0}^{\tau} \mathbf{S}_{\alpha _{\infty}} (\tau - \sigma) \Big( 
\hat{\mathbf{L}}_{\alpha(\sigma)} \mathbf{\Phi}(\sigma) + \mathbf{N}_{\alpha(\sigma)} (\mathbf{\Phi}(\sigma)) - \partial _{\sigma} \mathbf{\Psi}_{\alpha(\sigma)}
\Big) d \sigma,
\end{align}
and
\begin{align} \label{correction}
\mathbf{C}(\Phi,\alpha,\mathbf{u}):=\mathbf{P}_{\alpha_{\infty}} \mathbf{u} + \mathbf{P}_{\alpha_{\infty}} \int_{0}^{\infty} e^{-\sigma}  \Big( 
\hat{\mathbf{L}}_{\alpha(\sigma)} \mathbf{\Phi}(\sigma) + \mathbf{N}_{\alpha(\sigma)} (\mathbf{\Phi}(\sigma)) - \partial _{\sigma} \mathbf{\Psi}_{\alpha(\sigma)}
\Big) d \sigma.
\end{align}
Here, all integrals exist as Riemann integrals over continuous functions. Now, we can expect that the evolution \eqref{Modified} will have a solution provided that the initial data are sufficiently small. This is precisely our next result.
\begin{prop} \label{solutionmodified}
Let $\delta>0$ be sufficiently small and $c>0$ sufficiently large. If $\| \mathbf{u} \| \leq \frac{\delta}{c}$, then there exists a unique functions $\alpha \in X_{\delta}$ and $\Phi \in \mathcal{X}_{\delta}$ such that equation \eqref{Modified} holds for all $\tau>0$. 
\end{prop}
\begin{proof}
Here, $\alpha \in X_{\delta}$ is associated to $\Phi$ via Lemma \ref{suppressedLorentz}. The proof relies on a fixed point argument. The fact that $\mathbf{K}(\cdot,\alpha,\mathbf{u}) $ maps $\mathcal{X}_{\delta}$ to itself follows from Lemma \ref{bounds1} and Proposition \ref{instability}. Furthermore, the contraction property is a direct consequence of Lemma \ref{bounds2} and Lemma \ref{suppressedLorentz} and finally the Lipschitz continuity follows from Lemma \ref{bounds2}, Lemma \ref{nonlinearestimates2} and Lemma \ref{suppressedLorentz}. For more details see Proposition 5.7 in \cite{DonSch16a}. 
\end{proof}
Recall that our initial goal is to solve the modulation equation \eqref{Duhamelmodulation} so that we can establish a solution to the initial Cauchy problem \eqref{Cauchy}. So far, we can do this only for the modified equation \eqref{Modified} where the correction term is included. However, the correction term $\mathbf{C}(\Phi,\alpha,\mathbf{u})$ is closely related to the time translation symmetry and therefore we can choose $T$ in such a way that the correction term vanishes. On the other hand, the blowup time $T$ appears explicitly only in the initial data and not in the equation itself. To be precise, we have that 
\begin{align*}
\Phi(0)(\xi)& = \Psi(0)(\xi) - \Psi_{\alpha(0)} (\xi) 
= T \begin{pmatrix}
\psi_{0,1}(T\xi) + \tilde{f}(T\xi) \\
\psi_{0,2}(T\xi) + \tilde{g}(T\xi)
\end{pmatrix} - \Psi_{0} (\xi), 
\end{align*}
for some fixed and given functions $(\widetilde{f}, \widetilde{g})$ which stand for a perturbation of the initial data, see \eqref{datatilde}. Note, that we may write the initial data as 
\begin{align} \label{initialdata}
\Phi(0)(\xi)=\mathbf{U}(T,\mathbf{v}),
\end{align}
to distinguish between the blowup time $T$ and the perturbation  
\begin{align} \label{v}
\mathbf{v}:=
\begin{pmatrix}
\tilde{f} \\
 \tilde{g}
\end{pmatrix},
\end{align}
where
\begin{align} \label{U}
\mathbf{U}(T,\mathbf{v}):=\mathbf{v}^{T} + \Psi ^{T}_{0} - \Psi _{0}.
\end{align}
Here, we also write
\begin{align*}
\mathbf{w}^{T}:= 
\begin{pmatrix}
T w_{1}(T\xi) \\
T w_{2}(T\xi)
\end{pmatrix},
\end{align*}
for a generic function $\mathbf{w}=(w_{1},w_{2}) \in \mathcal{H}$. Before describing how one can choose $T$ in such a way that the correction term vanishes, we must ensure that, for all $T\in[1-\frac{\delta}{c},1+\frac{\delta}{c}]$, the modified equation \eqref{Modified} has a solution with initial data $\mathbf{u}=\mathbf{U}(T,\mathbf{v})$ provided that the perturbation $\mathbf{v}$ is sufficiently small. This fact is a direct consequence of Proposition \ref{solutionmodified} and the following lemma.
\begin{lemma}
Let $\delta>0$ be sufficiently small. If $\mathbf{v} \in H^{3}(\mathbb{B}_{1+\delta}^{5}) \times H^{2}(\mathbb{B}_{1+\delta}^{5})  $ such that $\| \mathbf{v} \|_{H^{3}(\mathbb{B}_{1+\delta}^{5}) \times H^{2}(\mathbb{B}_{1+\delta}^{5}) } \leq \delta$ then
\begin{align*}
\| \mathbf{U} (T,\mathbf{v}) \|_{H^{3}(\mathbb{B}_{1+\delta}^{5}) \times H^{2}(\mathbb{B}_{1+\delta}^{5}) } \lesssim \delta,
\end{align*}
for all $T \in [1-\delta,1+\delta]$. Furthermore, the map $\mathbf{U}(\cdot,\mathbf{v}) \rightarrow \mathcal{H}$ is continuous. 
\end{lemma}
\begin{proof}
The smallness condition on $\mathbf{U} (T,\mathbf{v})$ follows immediately from the fundamental theorem of calculus since $\psi_{0,1},\psi_{0,2} \in C^{\infty}(\mathbb{R}^{5})$. Furthermore, the continuity of the map follows from the triangle inequality and an approximation argument using the density of $C^{\infty}(\overline{ \mathbb{B}_{1+\delta}^{5}) }$ in $H^{k}(\mathbb{B}_{1+\delta}^{5})$. For a detailed proof see Lemma 5.8 in \cite{DonSch16a}. 
\end{proof}
Now, one can apply Proposition \ref{solutionmodified} to get the following result.
\begin{cor} \label{solutionmodifiedu}
Let $\delta>0$ be sufficiently small and $c$ sufficiently large. Furthermore, fix $\mathbf{v} \in H^{3}(\mathbb{B}_{1+ \delta / c}^{5}) \times H^{2}(\mathbb{B}_{1+ \delta / c}^{5})$ such that $\| \mathbf{v} \|_{H^{3}(\mathbb{B}_{1+\delta / c}^{5}) \times H^{2}(\mathbb{B}_{1+ \delta / c}^{5}) } \leq \frac{\delta}{c}$ and $T\in[1-\frac{\delta}{c},1+\frac{\delta}{c}]$. Then, the modified equation \eqref{Modified} with $\mathbf{u}=\mathbf{U}(T,\mathbf{v})$ has a solution $(\Phi,\alpha) \in \mathcal{X}_{\delta} \times X_{\delta}$. Furthermore, the map $T \longmapsto (\Phi,\alpha)$ is continuous. 
\end{cor}
Now, we focus on the correction term. To begin with we fix $\delta>0$ sufficiently small, $c$ sufficiently large and let $\mathbf{v} \in H^{3}(\mathbb{B}_{1+ \delta / c}^{5}) \times H^{2}(\mathbb{B}_{1+ \delta / c}^{5})$ such that $\| \mathbf{v} \|_{H^{3}(\mathbb{B}_{1+ \delta / c}^{5}) \times H^{2}(\mathbb{B}_{1+ \delta / c}^{5}) } \leq \frac{\delta}{c^2}$. Furthermore, pick an arbitrary $T= T_{\mathbf{v} }\in[1-\frac{\delta}{c},1+\frac{\delta}{c}]$ and let $(\Phi,\alpha) = (\Phi_{T},\alpha_{T}) \in \mathcal{X}_{\delta} \times X_{\delta}$ be a solution to the modified equation \eqref{Modified} with $\mathbf{u}=\mathbf{U}(T,\mathbf{v})$ by corollary \ref{solutionmodifiedu}. 
\begin{lemma} \label{correctionvanish}
There exists $T_{\mathbf{v}} \in [1-\frac{\delta}{c},1+\frac{\delta}{c}]$ such that $\mathbf{C} \left( \Phi _{T_{\mathbf{v}}}, \alpha _{T_{\mathbf{v}}}, \mathbf{U} \left( T_{\mathbf{v}},\mathbf{v} \right) \right) = 0$.
\end{lemma}
\begin{proof}
Since $\mathbf C$ has values in $\rg \mathbf {P}_{\alpha_{\infty}}=\langle\mathbf {g}_{\alpha_{\infty}}\rangle$ (see Lemma \ref{instability}), the vanishing of the correction term is equivalent to 
\begin{align}
\label{eq:Tv}
\exists T_{\mathbf{v}} \in [1-\frac{\delta}{c},1+\frac{\delta}{c}]: \quad \Big< \mathbf{C} \left( \Phi _{T_{\mathbf{v}}}, \alpha _{T_{\mathbf{v}}}, \mathbf{U} \left( T_{\mathbf{v}},\mathbf{v} \right) \right),\mathbf{g}_{\alpha_{\infty}} \Big>_{\mathcal H} = 0.
\end{align}
The key observation here is that
\begin{align*}
\partial _{T} \Psi _{0}^{T}  \big |_{T=1} = 2 \mathbf{g}_{0}
\end{align*}
and thus expanding $\Psi _{0}^{T}$ in Taylor with respect to $T$ around $T=1$ we get 
\begin{align*}
\mathbf U(T,\mathbf v) =\mathbf{v}^{T} + 2 \mathbf{g}_{0} (T-1) + \mathbf{R}_{T} (T-1)^{2}, 
\end{align*}
for some remainder term $ \mathbf{R}_{T}$, which we rewrite as
\begin{align*}
\mathbf U(T,\mathbf v) =\mathbf{v}^{T} + 2\mathbf{g}_{\alpha_{\infty}}(T-1)+2 \left( \mathbf{g}_{0} - \mathbf{g}_{\alpha_{\infty}} \right)(T-1) + \mathbf{R}_{T} (T-1)^{2}. 
\end{align*}
Now, the fact fact $\alpha(0)=0$ and \eqref{aclose} yield $|\alpha_{\infty}-\alpha(0)| \lesssim \delta$ and from Lemma \ref{nonlinearestimates2} (in particular \eqref{Lipg}) we get $\|  \mathbf{g}_{0} -  \mathbf{g}_{\alpha_{\infty}} \|\lesssim \delta$. In addition, $\| \mathbf{R}_{T} \| \lesssim 1$ for all $T\in[1-\frac{\delta}{c},1+\frac{\delta}{c}]$. Hence, using $\| \mathbf{v} \|_{H^{3}(\mathbb{B}_{1+ \delta / c}^{5}) \times H^{2}(\mathbb{B}_{1+ \delta / c}^{5}) } \leq \frac{\delta}{c^2}$ and $\rg \mathbf {P}_{\alpha_{\infty}}=\langle\mathbf {g}_{\alpha_{\infty}}\rangle$ from Lemma \ref{instability}, 
we infer
\begin{align*}
\Big< \mathbf{P}_{\alpha_{\infty}} \mathbf U(T,\mathbf v) ,\mathbf{g}_{\alpha_{\infty}} \Big>_{\mathcal H} = 
O \left( \frac{\delta}{c^2} \right) +2 \left \| \mathbf{g}_{\alpha_{\infty}} \right \|^{2} (T-1) + O \left( \frac{\delta ^2}{c} \right) + O \left( \frac{\delta^2}{c^2} \right).
\end{align*}
Moreover, the bounds of Lemma \ref{bounds1} imply
\begin{align*}
\Big< \mathbf{P}_{\alpha_{\infty}}  \int_{0}^{\infty} e^{-\sigma}  \Big( 
\hat{\mathbf{L}}_{\alpha(\sigma)} \mathbf{\Phi}(\sigma) + \mathbf{N}_{\alpha(\sigma)} (\mathbf{\Phi}(\sigma)) - \partial _{\sigma} \mathbf{\Psi}_{\alpha(\sigma)}
\Big) d \sigma ,\mathbf{g}_{\alpha_{\infty}} \Big>_{\mathcal H} = 
O \left( \delta^2 \right) \left \| \mathbf{g}_{\alpha_{\infty}} \right \|^{2}.
\end{align*}
Finally, summing up we get
\begin{align*}
\Big< \mathbf{C} \left( \Phi _{T_{\mathbf{v}}}, \alpha _{T_{\mathbf{v}}}, \mathbf{U} \left( T_{\mathbf{v}},\mathbf{v} \right) \right),\mathbf{g}_{\alpha_{\infty}} \Big>_{\mathcal H}  = 
2 \left \| \mathbf{g}_{\alpha_{\infty}} \right \|^{2} (T-1) + O \left( \frac{\delta}{c^2} \right).
\end{align*}
Setting the left hand side equal to zero we obtain the equation
\begin{align*}
T=1 + F(T)
\end{align*}
where $F$ is a continuous function in $T$ such that $F(T)=O\left( \frac{\delta}{c} \right)$. We choose $c$ sufficiently large and $\delta = \delta (c)$ sufficiently small so that $| F(T) | \leq \frac{\delta}{c}$. Now, the continuous function $T \longmapsto 1 + F (T )$ maps the closed interval $[1-\frac{\delta}{c},1+\frac{\delta}{c}]$ to itself and from Brouwer's fixed point theorem we get a fixed point $T = T_{\mathbf{v}}$. This proves \eqref{eq:Tv} and concludes the proof. 
\end{proof}
%
%
%
%
%
%
%
%
%
\section{Proof of the main theorem}
To begin with, we summarise the results of the previous section.
\begin{theorem} \label{main}
Let $\delta>0$ be sufficiently small, $c$ sufficiently large and pick an arbitrary $\mathbf{v} \in H^{3}(\mathbb{B}_{1+ \delta / c}^{5}) \times H^{2}(\mathbb{B}_{1+ \delta / c}^{5})$ such that $\| \mathbf{v} \|_{H^{3}(\mathbb{B}_{1+\delta / c}^{5}) \times H^{2}(\mathbb{B}_{1+ \delta / c}^{5}) } \leq \frac{\delta}{c^2}$. Then, there exists $T = T_{\mathbf{v}} \in[1-\frac{\delta}{c},1+\frac{\delta}{c}]$ such that the full, non-corrected equation \eqref{Duhamelmodulation} 
with initial data $\mathbf{u}=\mathbf{U}(T_{\mathbf{v}},\mathbf{v})$, that is 
\begin{align*}
\mathbf{\Phi }(\tau) = \mathbf{S}_{\alpha _{\infty}} (\tau) \mathbf{U}(T_{\mathbf{v}},\mathbf{v})+ \int_{0}^{\tau} \mathbf{S}_{\alpha _{\infty}} (\tau - \sigma) \Big( 
\hat{\mathbf{L}}_{\alpha(\sigma)} \mathbf{\Phi}(\sigma) + \mathbf{N}_{\alpha(\sigma)} (\mathbf{\Phi}(\sigma)) - \partial _{\sigma} \mathbf{\Psi}_{\alpha(\sigma)}
\Big) d \sigma,
\end{align*}
has a solution $(\Phi,\alpha) = (\Phi_{T_{\mathbf{v}}},\alpha_{T_{\mathbf{v}}}) \in \mathcal{X}_{\delta} \times X_{\delta}$. 
\end{theorem}
Now, we are in position to prove our main result. 
\begin{proof}[Proof of Theorem $\ref{mainresult}$ for $d=5$]
Fix $\delta >0$ sufficiently small and $c>0$ sufficiently large according to Theorem \ref{main}. Set $\delta^{\prime}:=\frac{\delta}{c}$ and $M:=c$. Furthermore, pick any initial data
\begin{align*}
(f,g)  \in H^{3} (\mathbb{B}_{1+\delta^{\prime}}^{5}) \times H^{2} (\mathbb{B}_{1+\delta^{\prime}}^{5}  )
\end{align*}
satisfying 
\begin{align*}
\Big \|  (f,g) -  u_{1,0} [0] \Big \|_{ H^{3} (\mathbb{B}_{1+\delta^{\prime}}^{5}) \times H^{2 } (\mathbb{B}_{1+\delta^{\prime}}^{5}  )} \leq \frac{\delta^{\prime}}{M}.
\end{align*}
Then, the perturbed initial data $ \mathbf{v}:=(\tilde{f},\tilde{g})$ (see \eqref{datatilde}) satisfy
\begin{align*}
 \| \mathbf{v}  \|_{ H^{3} (\mathbb{B}_{1+\frac{\delta}{c}}^{5}) \times H^{2 } (\mathbb{B}_{1+\frac{\delta}{c}}^{5}  )} =
\Big \| (f,g) - u_{1,0}[0] \Big \|_{ H^{3} (\mathbb{B}_{1+\delta^{\prime}}^{5}) \times H^{2 } (\mathbb{B}_{1+\delta^{\prime}}^{5}  )} \leq 
\frac{\delta^{\prime}}{M} =
\frac{\delta}{c^2}
\end{align*}
and Theorem $\ref{main}$ yields the existence of 
$T = T_{\mathbf{v}} \in [1-\delta^{\prime},1+\delta^{\prime}]$ such that equation \eqref{Duhamelmodulation} has a unique solution $(\Phi,\alpha) \in \mathcal{X}_{\delta} \times X_{\delta}$ 
with initial data $\Phi(0)=\mathbf{U}(T_{\mathbf{v}},\mathbf{v})$. Translating back this statement to the origin setting we obtain a weak solution $\Psi(\tau)=\Psi_{\alpha(\tau)} + \Phi(\tau)$ to the initial system \eqref{systempsi} with initial data $\Psi(0)=\Psi_{0}+\mathbf{U}(T_{\mathbf{v}},\mathbf{v})$. This means that 
\begin{align*}
u(t,x)=\frac{1}{T-t} \psi_{1} \left( \log \left( \frac{T}{T-t} \right),\frac{x}{T-t} \right)
\end{align*}
solves the cubic wave equation \eqref{cubicwave} with initial data
\begin{align*}
& u(0,x)=\frac{1}{T} \psi_{1}(0,\frac{x}{T}) = \psi_{1,0}(x) + \tilde{f}(x) = u_{1,0}(x) + \tilde{f}(x) \\
& \partial _{t} u(0,x)=\frac{1}{T^2} \psi_{2} (0,\frac{x}{T}) = \psi_{2,0}(x) + \tilde{g}(x) = \partial _{t} u_{1,0}(x) + \tilde{g}(x)
\end{align*}
for all $x\in\mathbb{B}_{1+\delta^{\prime}}^{5}$ and therefore is a solution to the Cauchy problem \eqref{Cauchy}. Finally, the fact that $\Phi  \in \mathcal{X}_{\delta}$ implies 
\begin{align*}
\| \Phi (\tau) \| \leq \delta e^{-\frac{1}{2} \tau},\quad \forall \tau>0
\end{align*}
and hence, for all $t\in [0,T)$ and $k=0,1,2,3$ we can estimate
\begin{align*}
& (T-t)^{k-\frac{5}{2}+1}\left \| u(t,\cdot) - u_{T,\alpha_{\infty}}(t,\cdot)
\right \|_{\dot H^{k}(\mathbb B^{5}_{T-t})}   = \\
&(T-t)^{k-\frac{5}{2}+1}\left \| \frac{1}{T-t} \psi_{1} \left (\log \left( \frac{T}{T-t} \right),\frac{\cdot}{T-t}\right ) -  
 \frac{1}{T-t} \psi_{\alpha_{\infty},1} \left (\frac{\cdot}{T-t}\right ) \right \|_{\dot H^{k}(\mathbb B^{5}_{T-t})} = \\
&(T-t)^{k-\frac{5}{2}}\left \| \psi_{1}  \left (\log \left( \frac{T}{T-t} \right),\frac{\cdot}{T-t}\right ) - \psi_{\alpha_{\infty},1}  \left (\frac{\cdot}{T-t}\right ) \right \|_{\dot H^{k}(\mathbb B^{5}_{T-t})} = \\
& \left \| \psi_{1}  \left (\log \left( \frac{T}{T-t} \right),\cdot \right ) - \psi_{\alpha_{\infty},1} \right \|_{\dot H^{k}(\mathbb B^{5}_{1})} \leq \\
& \left \| \psi_{1}  \left ( \log \left( \frac{T}{T-t} \right),\cdot \right ) - \psi_{\alpha (\log \left( \frac{T}{T-t} \right)),1} \right \|_{\dot H^{k}(\mathbb B^{5}_{1})} + 
 \left \| \psi_{\alpha(\log \left( \frac{T}{T-t} \right)),1}  - \psi_{\alpha_{\infty},1} \right \|_{\dot H^{k}(\mathbb B^{5}_{1})}.
\end{align*}
For the first term, we get
\begin{align*}
\left \| \psi_{1}  \left ( \log \left( \frac{T}{T-t} \right),\cdot \right ) - \psi_{\alpha (\log \left( \frac{T}{T-t} \right)),1} \right \|_{\dot H^{k}(\mathbb B^{5}_{1})} & \leq
\left \| \psi_{1}  \left ( \log \left( \frac{T}{T-t} \right),\cdot \right ) - \psi_{\alpha (\log \left( \frac{T}{T-t} \right)),1} \right \|_{ H^{3}(\mathbb B^{5}_{1})} \\
&\leq \left \| \Psi  \left ( \log \left( \frac{T}{T-t} \right) \right ) - \Psi_{\alpha (\log \left( \frac{T}{T-t} \right))} \right \|  \\
& = \left \| \Phi  \left ( \log \left( \frac{T}{T-t} \right) \right )  \right \| \\
& \lesssim (T-t)^{\frac{1}{2}}.
\end{align*}
For the second term, fix $t\in [0,T)$ and let $\gamma (s):=s  \alpha_{\infty} +(1-s) \alpha \left ( \log \left( \frac{T}{T-t} \right) \right ),~s \in[0,1]$ be a parametrisation of the line segment $E[\alpha \left ( \log \left( \frac{T}{T-t} \right) \right ) , \alpha_{\infty}]$ joining $\alpha \left ( \log \left( \frac{T}{T-t} \right) \right )$ and $ \alpha_{\infty}$. Then, the fundamental theorem of calculus yields
\begin{align*}
\psi _{\alpha \left ( \log \left( \frac{T}{T-t} \right) \right ),1} - \psi _{\alpha_{\infty},1} & = 
\psi _{\alpha \left ( \log \left( \frac{T}{T-t} \right) \right )} - \psi _{\alpha_{\infty}}  \\
&= \psi _{\gamma(0)} - \psi _{\gamma(1)} \\
&= \int_{ E[\alpha \left ( \log \left( \frac{T}{T-t} \right) \right ) , \alpha_{\infty}] } \partial \psi _{\gamma} \cdot d \ell \\
& = \sum _{j=1}^{5}  \left ( \alpha^{j} \left ( \log \left( \frac{T}{T-t} \right) \right ) - \alpha_{\infty}^{j} \right ) \int_{0}^{1} \partial _{\gamma ^{j}} \psi _{\gamma(t)} dt,
\end{align*} 
which implies the bound
\begin{align*} 
 \left \| \psi_{\alpha(\log \left( \frac{T}{T-t} \right)),1}  - \psi_{\alpha_{\infty},1} \right \|_{\dot H^{k}(\mathbb B^{5}_{1})} & = 
\left \| \partial^{k} \left( \psi_{\alpha(\log \left( \frac{T}{T-t} \right))}  - \psi_{\alpha_{\infty}}\right) \right \|_{L^{2}(\mathbb{B}^{5})} \\
 & \lesssim \sum _{j=1}^{5} \left | \alpha^{j} \left ( \log \left( \frac{T}{T-t} \right) \right ) - \alpha_{\infty}^{j} \right | \sup _{s \in E[\alpha \left ( \log \left( \frac{T}{T-t} \right) \right ) , \alpha_{\infty}]  } 
 \| \partial ^{k} \partial _{\gamma ^{j}} \psi _{\gamma(s)} \|_{L^{2}(\mathbb{B}^{5})} 
 \\
 & \lesssim \left | \alpha \left ( \log \left( \frac{T}{T-t} \right) \right ) - \alpha_{\infty} \right  | \\
 & \lesssim (T-t)^{\frac{1}{2}}
\end{align*}
due to \eqref{psibound} and \eqref{aclose} since $\alpha \in X_{\delta}$. The second estimate for $\partial_{t} \left(u(t,\cdot) - u_{T,\alpha_{\infty}}(t,\cdot)\right)$ follows similarly. These estimates conclude the proof.
\end{proof}
\begin{proof}[Proof of Theorem $\ref{mainresult}$ for $d\in \{ 7,9,11,13\}$]
All the results of the previous sections can be carried on for any $d \in \{ 7,9,11,13\}$ with slight modifications. The important parts are the function spaces which lead to a sharp decay for the free evolution and the spectral 
equation for $\alpha=0$. 
\\ \\
Referring to the spectral equation for $\alpha=0$ in higher space dimensions, one can readily verify that the potential $V_{0}$ in the definition of $\mathbf{L}^{\prime}_{0}$, see \eqref{Lprime}, will still turn out to be a constant function. Consequently, the spectral equation will be solved explicitly, solutions will belong to the hypergeometric class as well and we can still use the connection formula which is well known for this class. Then, one can proceed to the case where $\alpha \neq 0$ and since we are only interested in small $\alpha$ we can still apply a perturbative approach. To be precise, all estimates, Lipschitz bounds and decay rates will stay the same in all higher space dimensions since our results are formulated and proved using elements of abstract semigroup theory.
\\ \\
On the other hand, regarding the function spaces in higher space dimensions, one can still define a suitable inner product on 
\begin{align*}
\mathcal{H}=H^{\frac{d+1}{2}} \left( \mathbb{B}^{d} \right) \times H^{\frac{d-1}{2}} \left( \mathbb{B}^{d} \right)
\end{align*}
which yields a sharp decay for the "free" evolution operator. To be precise, we let
\begin{align*}
\mathcal{\widetilde{H}}=C^{\frac{d+1}{2}} (\overline{ \mathbb{B}^{d} } ) \times C^{\frac{d-1}{2}} ( \overline{ \mathbb{B}^{d}} ),
\end{align*}
and define
\begin{align*} 
\left( \cdot \big{|} \cdot \right): \mathcal{\widetilde{H}} \times \mathcal{\widetilde{H}} \longrightarrow \mathbb{R},\quad \left( \mathbf{u} \big{|} \mathbf{v}  \right):=\sum _{i =1}^{d}  \left( \mathbf{u} \big{|} \mathbf{v}  \right)_{i},
\end{align*} 
where, for $d=2k+1$, the sesquilinear forms are
\begin{align*} 
& \left( \mathbf{u} \big{|} \mathbf{v}  \right)_{1} := 
\int _{\mathbb{B}^{2k+1}} 
\partial _{m} \partial _{i_{1}} \cdots \partial _{i_{k}}  u_{1} (\xi)
\overline{ \partial ^{m}  \partial ^{i_{1}} \cdots \partial ^{i_{k}}  v_{1} (\xi) } 
d \xi \\
& \quad\quad\quad +  
\int _{\mathbb{B}^{2k+1}} 
 \partial _{i_{1}} \cdots \partial _{i_{k}} u_{2} (\xi)
\overline{ \partial ^{i_{1}} \cdots \partial ^{i_{k}} v_{2} (\xi) } 
d \xi \\ 
& \quad\quad\quad + 
\int _{\mathbb{S}^{2k}} 
 \partial _{i_{1}} \cdots \partial _{i_{k}}  u_{1} (\omega)
 \overline{  \partial ^{i_{1}} \cdots \partial ^{i_{k}} v_{1} (\omega) } 
 d \sigma (\omega), \\
& \left( \mathbf{u} \big{|} \mathbf{v}  \right)_{2}:= 
\int _{\mathbb{B}^{2k+1}}
\partial_{m} \partial ^{m}  \partial _{i_{1}} \cdots \partial _{i_{k-1}} u_{1}  (\xi)
\overline{ \partial_{n} \partial ^{n}  \partial ^{i_{1}} \cdots \partial ^{i_{k-1}} v_{1} (\xi) } d\xi \\
& \quad\quad\quad +
\int _{\mathbb{B}^{2k+1}}
  \partial _{i_{1}} \dots \partial _{i_{k}} u_{2}  (\xi)
\overline{  \partial ^{i_{1}} \cdots \partial ^{i_{k}} v_{2} (\xi) } d\xi \\
& \quad\quad\quad+ \int _{\mathbb{S}^{2k}} 
 \partial _{i_{1}} \cdots \partial _{i_{k-1}}  u_{2} (\omega)
 \overline{  \partial ^{i_{1}} \cdots \partial ^{i_{k-1}} v_{2} (\omega) } 
 d \sigma (\omega), \\
 &\quad\quad\quad \vdots \\
& \left( \mathbf{u} \big{|} \mathbf{v}  \right)_{3+2q}:=
 \sum _{p=1}^{2q+2} A_{q}^{p}(d) \left( \mathbf{u} \big{|} \mathbf{v}  \right)_{p} +
 \int _{\mathbb{S}^{2k}} 
 \partial _{i_{1}} \cdots \partial _{i_{k-2-q}}  u_{2} (\omega)
 \overline{  \partial ^{i_{1}} \cdots \partial ^{i_{k-2-q}} v_{2} (\omega) } 
 d \sigma (\omega), \\
& \left( \mathbf{u} \big{|} \mathbf{v}  \right)_{4+2q}:= 
 \sum _{r=1}^{2q+3} B_{q}^{r}(d)  \left( \mathbf{u} \big{|} \mathbf{v}  \right)_{r} +
 \int _{\mathbb{S}^{2k}} 
 \partial _{i_{1}} \cdots \partial _{i_{k-1-q}}  u_{1} (\omega)
 \overline{  \partial ^{i_{1}} \cdots \partial ^{i_{k-1-q}} v_{1} (\omega) } 
 d \sigma (\omega),
\end{align*}
for some constants $A_{q}^{p}(d)$ and $B_{q}^{r}(d)$ and for all $q=0,1,\dots,k-2$ and all $\mathbf{u},\mathbf{v} \in \mathcal{\widetilde{H}}$. In addition, the missing piece for it to define a norm is given by
\begin{align*}
\left( \mathbf{u} \big{|} \mathbf{v}  \right)_{2k+1}:= 
\left(  \int _{\mathbb{S}^{2k}} \zeta \left(\omega,\mathbf{u} (\omega) \right) d \sigma (\omega)  \right) 
\left(  \int _{\mathbb{S}^{2k}} \overline{ \zeta \left(\omega,\mathbf{v}(\omega)  \right) } d \sigma (\omega)  \right) 
\end{align*}
where
\begin{align*}
\zeta \left(\omega,\mathbf{w}(\omega)  \right):=  D_{2k+1} w_{1} (\omega) + \tilde{D}_{2k+1} w_{2}  (\omega)
\end{align*}
and
\begin{align*}
& D_{2k+1} w_{1} (\omega) := \sum_{j=1}^{k} a_{j} \omega ^{i_{1}}  \cdots  \omega ^{i_{j}} \partial _{i_{1}} \cdots \partial _{i_{j}} w_{1} (\omega) +  a_{0} w_{1} (\omega), \\
& \tilde{D}_{2k+1} w_{2} (\omega) := \sum_{j=1}^{k-1} b_{j} \omega ^{i_{1}}  \cdots  \omega ^{i_{j}} \partial _{i_{1}} \cdots \partial _{i_{j}} w_{2} (\omega)+  b_{0} w_{2} (\omega),
\end{align*}
for appropriate constants $a_{j},b_{j},a_{0}$ and $b_{0}$. Recall that in all these definitions the Einstein summation convention is assumed. Now, the constants $a_{j},b_{j},a_{0}$ and $b_{0}$ are chosen in such a way that the identity
\begin{align} \label{identityzeta}
 \zeta \left(\omega, \widetilde{ \mathbf{L} } \mathbf{u} (\omega) \right) = - \zeta \left(\omega,  \mathbf{u} (\omega) \right) +
 \Delta^{\mathbb{S}^{2k}} \left( \tilde{D}_{2k-1} \Big( u_{1} (\omega) + \omega ^{j} \partial _{j} u_{1} (\omega) \Big) \right)
\end{align}
holds which is the key identity to obtain the decay
\begin{align} \label{decay1}
\mathrm{Re}~ \big( \widetilde{ \mathbf{L} }  \mathbf{u} \big{|} \mathbf{u}  \big)_{2k+1} = - \| \mathbf{u} \|_{2k+1}^{2},
\end{align}
see \eqref{identityzeta5}. In higher space dimensions, although it is easy to prove that the inner product $\left( \cdot \big{|} \cdot \right)$ defines indeed a norm equivalent to $\mathcal{H}$, there are two main difficulties. On the one hand, we can find a defining recurrence relation for the coefficients $a_{j},b_{j},a_{0}$ and $b_{0}$ which unfortunately is not convenient to write it down nor easy to use and therefore proving \eqref{identityzeta} turns out to be too difficult for us. On the other hand, we can use induction to prove that
\begin{align} \label{decay2}
\mathrm{Re}~ \big( \widetilde{ \mathbf{L} }  \mathbf{u} \big{|} \mathbf{u}  \big)_{i} \leq - \frac{3}{2} \| \mathbf{u} \|_{i}^{2},
\end{align}
for all $i \in \{1,2,\dots,2k \}$, but the proof is rather involved. 
\\ \\
However, for small $d$, say $d \in \{7,9,11,13 \}$, we can find the coefficients $a_{j},b_{j},a_{0}$ and $b_{0}$ explicitly, define $D_{2k+1}$ and $\tilde{D}_{2k+1}$ without recurrence relations and successfully verify \eqref{decay2}, \eqref{identityzeta} and therefore \eqref{decay1}. Furthermore, in this case, the proof of \eqref{decay2} rely on similar estimates to the ones in Lemma \ref{Lumer1a} without any additional tools. Specifically, for $d=7$, we define
\begin{align*}
\mathcal{D} \big(\widetilde{ \mathbf{L} }  \big):= C^{5} \big( \overline{ \mathbb{B}^{7} } \big) \times C^{4} \big( \overline{ \mathbb{B}^{7} } \big)
\end{align*}
and
\begin{align*} 
\left( \cdot \big{|} \cdot \right): \left( C^{4} ( \overline{ \mathbb{B}^{7} } ) \times C^{3} (\overline{ \mathbb{B}^{7} } ) \right)^2
 \longrightarrow \mathbb{R},\quad \left( \mathbf{u} \big{|} \mathbf{v}  \right):=\sum _{i =1}^{7}  \left( \mathbf{u} \big{|} \mathbf{v}  \right)_{i},
\end{align*} 
where the sesquilinear forms are
\begin{align*} 
& \left( \mathbf{u} \big{|} \mathbf{v}  \right)_{1}:= \int _{\mathbb{B}^{7}} 
\partial _{i} \partial _{j} \partial _{k} \partial _{\ell} u_{1} (\xi )
\overline{ \partial ^{i} \partial ^{j} \partial ^{k}  \partial ^{\ell}  v_{1} (\xi ) } d \xi  +  \int _{\mathbb{B}^{7}} \partial _{i} \partial _{j}  \partial _{k}  u_{2} (\xi ) \overline{\partial ^{i} \partial ^{j}  \partial ^{k} v_{2} (\xi ) } d \xi \\
&\quad\quad\quad + \int _{\mathbb{S}^{6}} \partial _{i} \partial _{j}  \partial _{k}  u_{1}  (\omega)  \overline{ \partial ^{i} \partial ^{j}  \partial ^{k} v_{1}  (\omega) } d \sigma (\omega), \\
& \left( \mathbf{u} \big{|} \mathbf{v}  \right)_{2}:= \int _{\mathbb{B}^{7}} \partial _{i} \partial _{j} \partial ^{k}  \partial _{k} u_{1} (\xi ) \overline{ \partial ^{i} \partial ^{j} \partial ^{l} \partial _{l} v_{1}  (\xi ) } d \xi 
+  \int _{\mathbb{B}^{7}} \partial _{i} \partial _{j} \partial _{k} u_{2} (\xi ) \overline{ \partial ^{i} \partial ^{j} \partial ^{k} v_{2}(\xi ) } d \xi \\
& \quad\quad\quad + \int _{\mathbb{S}^{6}} \partial _{i} \partial _{j} u_{2} (\omega) \overline{\partial ^{i} \partial ^{j} v_{2}(\omega) } d \sigma (\omega), \\
& \left( \mathbf{u} \big{|} \mathbf{v}  \right)_{3}:= \sum_{j=1}^{2} A_{3}^{j} \left( \mathbf{u} \big{|} \mathbf{v}  \right)_{j} + \int _{\mathbb{S}^{6}} \partial _{i} u_{2} (\omega) \overline{ \partial ^{i} v_{2} (\omega) } d \sigma (\omega), \\
& \left( \mathbf{u} \big{|} \mathbf{v}  \right)_{4}:= \sum_{j=1}^{3} A_{4}^{j}  \left( \mathbf{u} \big{|} \mathbf{v}  \right)_{j} + \int _{\mathbb{S}^{6}} \partial _{i} 
\partial _{j} u_{1} (\omega)  \overline{ \partial ^{i} \partial ^{j} v_{1} (\omega) }d \sigma (\omega), \\
& \left( \mathbf{u} \big{|} \mathbf{v}  \right)_{5}:= \sum_{j=1}^{4}  A_{5}^{j}  \left( \mathbf{u} \big{|} \mathbf{v}  \right)_{j} + \int _{\mathbb{S}^{6}}  u_{2}(\omega)  \overline{  v_{2} (\omega) } d \sigma (\omega), \\
& \left( \mathbf{u} \big{|} \mathbf{v}  \right)_{6}:= \sum_{j=1}^{5}  A_{6}^{j}  \left( \mathbf{u} \big{|} \mathbf{v}  \right)_{j} + \int _{\mathbb{S}^{6}}  \partial _{i} u_{1} (\omega) \overline{  \partial ^{i} v_{1}(\omega) } d \sigma (\omega), \\
& \left( \mathbf{u} \big{|} \mathbf{v}  \right)_{7}:= 
\left(  \int _{\mathbb{S}^{6}} \zeta \left(\omega,\mathbf{u} (\omega) \right) d \sigma (\omega)  \right) 
\left(  \int _{\mathbb{S}^{6}} \overline{ \zeta \left(\omega,\mathbf{v}(\omega)  \right) } d \sigma (\omega)  \right), 
\end{align*}
for some constants $A_{i}^{j}$ and for all $\mathbf{u},\mathbf{v} \in C^{4} (\overline{ \mathbb{B}^{7} } ) \times C^{3} (\overline{ \mathbb{B}^{7} })$. Here,
\begin{align*}
& \zeta \left(\omega,\mathbf{w}(\omega)  \right):=  D_{7} w_{1} (\omega) + \tilde{D}_{7} w_{2} (\omega), \\ 
& D_{7} w_{1}(\omega) := \omega ^{i}  \omega ^{j}  \omega ^{k} \partial _{i} \partial _{j}  \partial _{k} w_{1}(\omega) + 12 \omega ^{i} \omega ^{j} \partial _{i} \partial _{j} w_{1} (\omega)+ 33 \omega ^{i} \partial _{i} w_{1} (\omega)+  15 w_{1}(\omega), \\
& \tilde{D}_{7} w_{2}(\omega) := \omega ^{i} \omega ^{j} \partial _{i} \partial _{j} w_{2}(\omega) + 9 \omega ^{j} \partial _{j} w_{2}(\omega) +  15 w_{2}(\omega).
\end{align*}
One can prove that this inner product defines indeed a norm equivalent to $H^{4} \left( \mathbb{B}^{7} \right) \times H^{3} \left( \mathbb{B}^{7} \right)$ and the decay estimates \eqref{decay1} and \eqref{decay2} hold. Furthermore, for $d=9$, we define
\begin{align*}
\mathcal{D} \big(\widetilde{ \mathbf{L} }  \big):= C^{6} \big( \overline{ \mathbb{B}^{9} } \big) \times C^{5} \big( \overline{ \mathbb{B}^{9} } \big).
\end{align*}
and
\begin{align*} 
\left( \cdot \big{|} \cdot \right): \left( C^{5} (\overline{ \mathbb{B}^{9} } ) \times C^{4} ( \overline{ \mathbb{B}^{9} } ) \right)^2
 \longrightarrow \mathbb{R},\quad \left( \mathbf{u} \big{|} \mathbf{v}  \right):=\sum _{i =1}^{9}  \left( \mathbf{u} \big{|} \mathbf{v}  \right)_{i},
\end{align*} 
where the sesquilinear forms are
\begin{align*} 
& \left( \mathbf{u} \big{|} \mathbf{v}  \right)_{1}:= \int _{\mathbb{B}^{9}} 
\partial _{i} \partial _{j} \partial _{k} \partial _{\ell} \partial _{m} u_{1} (\xi)
\overline{ \partial ^{i} \partial ^{j} \partial ^{k}  \partial ^{\ell}  \partial ^{m} v_{1}(\xi) } d \xi 
+  \int _{\mathbb{B}^{9}} \partial _{i} \partial _{j}  \partial _{k} \partial _{\ell}  u_{2}(\xi) \overline{\partial ^{i} \partial ^{j}  \partial ^{k}  \partial ^{\ell} v_{2} (\xi) } d \xi \\
& \quad\quad\quad+ \int _{\mathbb{S}^{8}} \partial _{i} \partial _{j}  \partial _{k}  \partial _{\ell} u_{1}  (\omega) \overline{ \partial ^{i} \partial ^{j}  \partial ^{k}  \partial ^{\ell} v_{1} (\omega) } d \sigma (\omega), \\
& \left( \mathbf{u} \big{|} \mathbf{v}  \right)_{2}:= \int _{\mathbb{B}^{9}} \partial _{i} \partial _{j} \partial ^{k}  \partial _{\ell}  \partial ^{\ell} u_{1}(\xi)  \overline{ \partial ^{i} \partial ^{j}  \partial _{k} \partial ^{m} \partial _{m} v_{1}(\xi)  } d \xi 
+  \int _{\mathbb{B}^{9}} \partial _{i} \partial _{j} \partial _{k}  \partial _{\ell} u_{2} (\xi)   \overline{ \partial ^{i} \partial ^{j} \partial ^{k}  \partial ^{\ell} v_{2} (\xi)  } d \xi  \\
&\quad\quad\quad+ \int _{\mathbb{S}^{8}} \partial _{i} \partial _{j}  \partial _{k} u_{2} (\omega) \overline{\partial ^{i} \partial ^{j}  \partial ^{k} v_{2} (\omega) } d \sigma (\omega), \\
& \left( \mathbf{u} \big{|} \mathbf{v}  \right)_{3}:= \sum_{j=1}^{2} B_{3}^{j} \left( \mathbf{u} \big{|} \mathbf{v}  \right)_{j} + \int _{\mathbb{S}^{8}} \partial _{i}  \partial _{j} u_{2} (\omega) \overline{ \partial ^{i}  \partial ^{j} v_{2} (\omega) } 
d \sigma (\omega), \\
& \left( \mathbf{u} \big{|} \mathbf{v}  \right)_{4}:= \sum_{j=1}^{3}  B_{4}^{j} \left( \mathbf{u} \big{|} \mathbf{v}  \right)_{j} + \int _{\mathbb{S}^{8}} \partial _{i} 
\partial _{j}  \partial _{k} u_{1} (\omega) \overline{ \partial ^{i} \partial ^{j}  \partial _{k} v_{1}(\omega) }d \sigma (\omega), \\
& \left( \mathbf{u} \big{|} \mathbf{v}  \right)_{5}:= \sum_{j=1}^{4} B_{5}^{j} \left( \mathbf{u} \big{|} \mathbf{v}  \right)_{j} + \int _{\mathbb{S}^{8}}   \partial _{i} u_{2} (\omega) \overline{  \partial ^{i} v_{2}(\omega) }d \sigma (\omega), \\
& \left( \mathbf{u} \big{|} \mathbf{v}  \right)_{6}:= \sum_{j=1}^{5} B_{6}^{j} \left( \mathbf{u} \big{|} \mathbf{v}  \right)_{j} + \int _{\mathbb{S}^{8}}  \partial _{i}  \partial _{j} u_{1}(\omega)  \overline{  \partial ^{i}  \partial ^{j} v_{1}(\omega) }d \sigma (\omega), \\
& \left( \mathbf{u} \big{|} \mathbf{v}  \right)_{7}:= \sum_{j=1}^{6}B_{7}^{j} \left( \mathbf{u} \big{|} \mathbf{v}  \right)_{j} + \int _{\mathbb{S}^{8}}  u_{2} (\omega) \overline{   v_{2}(\omega) }d \sigma (\omega), \\
& \left( \mathbf{u} \big{|} \mathbf{v}  \right)_{8}:= \sum_{j=1}^{7} B_{8}^{j}\left( \mathbf{u} \big{|} \mathbf{v}  \right)_{j} + \int _{\mathbb{S}^{8}}  \partial _{i}   u_{1} (\omega) \overline{  \partial ^{i}  v_{1}(\omega) }d \sigma (\omega), \\
& \left( \mathbf{u} \big{|} \mathbf{v}  \right)_{9}:= 
\left(  \int _{\mathbb{S}^{6}} \zeta \left(\omega,\mathbf{u} (\omega) \right) d \sigma (\omega)  \right) 
\left(  \int _{\mathbb{S}^{6}} \overline{ \zeta \left(\omega,\mathbf{v}(\omega)  \right) } d \sigma (\omega)  \right), 
\end{align*}
for some constants $B_{i}^{j}$ and for all $\mathbf{u},\mathbf{v} \in C^{5} (\overline{ \mathbb{B}^{9} }) \times C^{4} (\overline{ \mathbb{B}^{9} })$. Here,
\begin{align*}
& \zeta \left(\omega,\mathbf{w}(\omega)  \right):=  D_{9} w_{1}(\omega) + \tilde{D}_{9} w_{2}(\omega), \\ 
& D_{9} w_{1} (\omega) := \omega ^{i}  \omega ^{j}  \omega ^{k} \omega ^{\ell} \partial _{i} \partial _{j}  \partial _{k} \partial _{\ell}  w_{1} (\omega)+
22  \omega ^{i}  \omega ^{j}  \omega ^{k} \partial _{i} \partial _{j}  \partial _{k} w_{1} (\omega)+ 141 \omega ^{i} \omega ^{j} \partial _{i} \partial _{j} w_{1} (\omega) \\
& \quad \quad \quad \quad + 279 \omega ^{i} \partial _{i} w_{1}(\omega) +  105 w_{1} (\omega), \\
& \tilde{D}_{9} w_{2}(\omega)  :=  \omega ^{i}  \omega ^{j}  \omega ^{k} \partial _{i} \partial _{j}  \partial _{k} w_{2} (\omega)+
18 \omega ^{i} \omega ^{j} \partial _{i} \partial _{j} w_{2} (\omega)+ 87 \omega ^{j} \partial _{j} w_{2}(\omega) +  105 w_{2}(\omega).
\end{align*}
We can verify that this inner product defines indeed a norm equivalent to $H^{5} \left( \mathbb{B}^{9} \right) \times H^{4} \left( \mathbb{B}^{9} \right)$ and the decay estimates \eqref{decay1} and \eqref{decay2} hold. In addition, for $d=11$, we define
\begin{align*}
\mathcal{D} \big(\widetilde{ \mathbf{L} }  \big):= C^{7} \big( \overline{ \mathbb{B}^{11} } \big) \times C^{6} \big( \overline{ \mathbb{B}^{11} } \big).
\end{align*}
and
\begin{align*} 
\left( \cdot \big{|} \cdot \right): \left( C^{6} (\overline{ \mathbb{B}^{11} }) \times C^{5} (\overline{ \mathbb{B}^{11} }) \right)^2
 \longrightarrow \mathbb{R},\quad \left( \mathbf{u} \big{|} \mathbf{v}  \right):=\sum _{i =1}^{11}  \left( \mathbf{u} \big{|} \mathbf{v}  \right)_{i},
\end{align*} 
where the sesquilinear forms are
\begin{align*} 
& \left( \mathbf{u} \big{|} \mathbf{v}  \right)_{1}:= \int _{\mathbb{B}^{11}} 
\partial _{i} \partial _{j} \partial _{k} \partial _{\ell} \partial _{m}  \partial _{n} u_{1}  (\xi)
\overline{ \partial ^{i} \partial ^{j} \partial ^{k}  \partial ^{\ell}  \partial ^{m} \partial ^{n} v_{1}  (\xi) } d \xi 
+  \int _{\mathbb{B}^{11}} \partial _{i} \partial _{j}  \partial _{k} \partial _{\ell}  \partial _{n} u_{2}  (\xi) \overline{\partial ^{i} \partial ^{j}  \partial ^{k}  \partial ^{\ell} \partial ^{n} v_{2}  (\xi) } d \xi \\
& \quad\quad\quad+ \int _{\mathbb{S}^{10}} \partial _{i} \partial _{j}  \partial _{k}  \partial _{\ell} \partial _{n} u_{1} (\omega)  \overline{ \partial ^{i} \partial ^{j}  \partial ^{k}  \partial ^{\ell} \partial ^{n} v_{1}(\omega) } d \sigma (\omega), \\
& \left( \mathbf{u} \big{|} \mathbf{v}  \right)_{2}:= \int _{\mathbb{B}^{11}} \partial _{i} \partial _{j} \partial _{k}  \partial _{\ell} \partial_{n} \partial ^{n} u_{1}  (\xi)
\overline{ \partial ^{i} \partial ^{j}  \partial _{k} \partial _{\ell} \partial ^{m} \partial _{m} v_{1}  (\xi) } d \xi 
+  \int _{\mathbb{B}^{11}} \partial _{i} \partial _{j} \partial _{k}  \partial _{\ell} \partial _{n} u_{2}  (\xi) \overline{ \partial ^{i} \partial ^{j} \partial ^{k}  \partial ^{\ell} \partial ^{n} v_{2} (\xi) } d \xi  \\
&\quad\quad\quad+ \int _{\mathbb{S}^{10}} \partial _{i} \partial _{j}  \partial _{k}  \partial _{\ell} u_{2}  (\omega) \overline{\partial ^{i} \partial ^{j}  \partial ^{k}  \partial ^{\ell} v_{2} (\omega) } d \sigma (\omega), \\
& \left( \mathbf{u} \big{|} \mathbf{v}  \right)_{3}:= \sum_{j=1}^{2}C_{3}^{j} \left( \mathbf{u} \big{|} \mathbf{v}  \right)_{j} + \int _{\mathbb{S}^{1}} \partial _{i}  \partial _{j} \partial _{k} u_{2} (\omega)
 \overline{ \partial ^{i}  \partial ^{j} \partial ^{k} v_{2}(\omega) } d \sigma (\omega), \\
& \left( \mathbf{u} \big{|} \mathbf{v}  \right)_{4}:= \sum_{j=1}^{3}C_{4}^{j}  \left( \mathbf{u} \big{|} \mathbf{v}  \right)_{j} + \int _{\mathbb{S}^{10}} \partial _{i} 
\partial _{j}  \partial _{k} \partial _{\ell} u_{1} (\omega) \overline{ \partial ^{i} \partial ^{j}  \partial _{k} \partial ^{\ell} v_{1}(\omega) }d \sigma (\omega), \\
& \left( \mathbf{u} \big{|} \mathbf{v}  \right)_{5}:= \sum_{j=1}^{4} C_{5}^{j} \left( \mathbf{u} \big{|} \mathbf{v}  \right)_{j} + \int _{\mathbb{S}^{10}}   \partial _{i} \partial _{j}  u_{2} (\omega) \overline{  \partial ^{i} \partial _{j} v_{2} (\omega) }d \sigma (\omega), \\
& \left( \mathbf{u} \big{|} \mathbf{v}  \right)_{6}:= \sum_{j=1}^{5} C_{6}^{j} \left( \mathbf{u} \big{|} \mathbf{v}  \right)_{j} + \int _{\mathbb{S}^{10}}  \partial _{i}  \partial _{j} \partial _{k} u_{1} (\omega)  \overline{  \partial ^{i}  \partial ^{j} \partial _{k} v_{1}(\omega) }d \sigma (\omega), \\
& \left( \mathbf{u} \big{|} \mathbf{v}  \right)_{7}:= \sum_{j=1}^{6} C_{7}^{j}  \left( \mathbf{u} \big{|} \mathbf{v}  \right)_{j} + \int _{\mathbb{S}^{10}}\partial _{i}   u_{2} (\omega)  \overline{ \partial ^{i}   v_{2} (\omega) }d \sigma (\omega), \\
& \left( \mathbf{u} \big{|} \mathbf{v}  \right)_{8}:= \sum_{j=1}^{7} C_{8}^{j}  \left( \mathbf{u} \big{|} \mathbf{v}  \right)_{j} + \int _{\mathbb{S}^{10}}  \partial _{i}  \partial _{j}  u_{1} (\omega) \overline{  \partial ^{i}  \partial ^{j} v_{1}(\omega) }d \sigma (\omega), \\
&  \left( \mathbf{u} \big{|} \mathbf{v}  \right)_{9}:= \sum_{j=1}^{8}C_{9}^{j}  \left( \mathbf{u} \big{|} \mathbf{v}  \right)_{j} + \int _{\mathbb{S}^{10}}   u_{2}(\omega)  \overline{  v_{2}(\omega) }d \sigma (\omega), \\
&  \left( \mathbf{u} \big{|} \mathbf{v}  \right)_{10}:= \sum_{j=1}^{7} C_{10}^{j} \left( \mathbf{u} \big{|} \mathbf{v}  \right)_{j} + \int _{\mathbb{S}^{10}}  \partial _{i}  u_{1}(\omega)  \overline{  \partial ^{i}  v_{1}(\omega) }d \sigma (\omega), \\
& \left( \mathbf{u} \big{|} \mathbf{v}  \right)_{11}:= 
\left(  \int _{\mathbb{S}^{10}} \zeta \left(\omega,\mathbf{u} (\omega) \right) d \sigma (\omega)  \right) 
\left(  \int _{\mathbb{S}^{10}} \overline{ \zeta \left(\omega,\mathbf{v}(\omega)  \right) } d \sigma (\omega)  \right), 
\end{align*}
for some constants $C_{i}^{j}$ and for all $\mathbf{u},\mathbf{v} \in C^{6} (\overline{ \mathbb{B}^{11} }) \times C^{5} (\overline{ \mathbb{B}^{11} })$. Here,
\begin{align*}
& \zeta \left(\omega,\mathbf{w}(\omega)  \right):=  D_{11} w_{1} (\omega)+ \tilde{D}_{11} w_{2}(\omega), \\ 
& D_{11} w_{1}(\omega) := \omega ^{i}  \omega ^{j}  \omega ^{k} \omega ^{\ell}  \omega ^{m} \partial _{i} \partial _{j}  \partial _{k} \partial _{\ell}  \partial _{m}  w_{1} (\omega)+
35 \omega ^{i}  \omega ^{j}  \omega ^{k} \omega ^{\ell} \partial _{i} \partial _{j}  \partial _{k} \partial _{\ell}  w_{1}(\omega) +
405  \omega ^{i}  \omega ^{j}  \omega ^{k} \partial _{i} \partial _{j}  \partial _{k} w_{1}(\omega) \\
&\quad\quad\quad \quad + 
1830 \omega ^{i} \omega ^{j} \partial _{i} \partial _{j} w_{1} (\omega)+ 
2895 \omega ^{i} \partial _{i} w_{1}(\omega) +  945 w_{1}(\omega), \\
& \tilde{D}_{11} w_{2}(\omega) :=  \omega ^{i}  \omega ^{j}  \omega ^{k}  \omega ^{\ell} \partial _{i} \partial _{j}  \partial _{k}  \partial _{\ell}  w_{2}(\omega) +
30 \omega ^{i}  \omega ^{j}  \omega ^{k} \partial _{i} \partial _{j}  \partial _{k} w_{2} (\omega)+
285 \omega ^{i} \omega ^{j} \partial _{i} \partial _{j} w_{2}(\omega) \\
&\quad\quad\quad \quad + 
975 \omega ^{j} \partial _{j} w_{2}(\omega) +  945 w_{2}(\omega).
\end{align*}
We can verify that this inner product defines indeed a norm equivalent to $H^{6} \left( \mathbb{B}^{11} \right) \times H^{5} \left( \mathbb{B}^{11} \right)$ and the decay estimates \eqref{decay1} and \eqref{decay2} hold. Similarly, we get analogous formulas for the case $d=13$ and verify \eqref{decay1} and \eqref{decay2}.
\end{proof}

 \bibliographystyle{plain}
 \bibliography{waveodd_nosym}
\end{document}